\documentclass[11pt, oneside]{article}
\usepackage[utf8]{inputenc}
\usepackage{amsmath, amsthm, amssymb}
\usepackage[dvipsnames]{xcolor}
\usepackage{enumerate, comment, multirow}
\usepackage{parskip}
\usepackage{dsfont}
\usepackage{enumitem}
\usepackage{mathtools}
\usepackage{aliascnt}
\usepackage{booktabs}

\usepackage{amsmath,amssymb,amsthm,mathtools}
\usepackage{booktabs,longtable,array}
\usepackage{enumitem}
\usepackage{needspace}
\usepackage{tikz}
\usepackage{xcolor}
\usepackage[hidelinks]{hyperref}
\usepackage{microtype}

\makeatletter
\newenvironment{certblock}[1]{%
  \par\needspace{9\baselineskip}%
  \phantomsection\def\@currentlabel{#1}%
  \noindent\textbf{Certificate #1.}\par\nopagebreak\smallskip}{\par\medskip}
\makeatother

\usepackage{comment}
\usepackage{makecell}

\usepackage{subcaption}
\usepackage{todonotes}
\usepackage{multicol}
\usepackage{tikz}
\usetikzlibrary{decorations.markings, arrows.meta}
\usetikzlibrary{decorations.pathmorphing}
\usetikzlibrary{calc,backgrounds}
\usetikzlibrary{patterns}
\usepackage{standalone}
\usepackage{mathtools,mdframed}
\usepackage{adjustbox}
\usepackage[left=0.8in,right=0.8in,top=0.8in,bottom=0.8in,bindingoffset=0cm]{geometry}

\usepackage[normalem]{ulem}

\usepackage{hyperref}
\hypersetup{colorlinks=true,linkcolor=blue,anchorcolor=blue,citecolor=green,filecolor=blue,urlcolor=blue,bookmarksnumbered=true,pdfview=FitB}

\usepackage[capitalise]{cleveref}

\usepackage{tabularx}
\usepackage{booktabs}

\definecolor{green}{RGB}{50,205,50}
\definecolor{yellow}{RGB}{240,240,10}

\newtheorem{maintheorem}{Theorem}

\newtheorem{mainlemma}{Key Lemma}
\crefname{keylemma}{Key Lemma}{Key Lemmas}
\newtheorem*{keylemma*}{Key Lemma}

\newaliascnt{lemma}{theorem}
\newtheorem{lemma}[lemma]{Lemma}
\aliascntresetthe{lemma}

\newaliascnt{corollary}{theorem}
\newtheorem{corollary}[corollary]{Corollary}
\aliascntresetthe{corollary}

\newaliascnt{subclaim}{theorem}

\aliascntresetthe{subclaim}

\newaliascnt{conjecture}{theorem}

\aliascntresetthe{conjecture}

\newaliascnt{question}{theorem}

\aliascntresetthe{question}

\newaliascnt{observation}{theorem}

\aliascntresetthe{observation}

\newaliascnt{proposition}{theorem}

\aliascntresetthe{proposition}

\newtheorem*{theorem*}{Theorem}
\newtheorem*{corollary*}{Corollary}
\newtheorem*{lemma*}{Lemma}

\theoremstyle{definition}

\theoremstyle{definition}

\newtheorem{claim}{Claim}
\newenvironment{proofc}{\begin{proof}[Proof of Claim]}{\end{proof}}

\tikzset{vtx/.style={circle,draw,inner sep=1.1pt,fill=white,minimum size=4.5mm}}

\title{Yes, $(2K_2, K_4)$-free graphs are recolorable}

\author{
Henry Echeverr\'ia\footnote{Instituto de Ingenier\'ia Matem\'atica-CIMFAV, Universidad de Valpara\'iso, Chile.
Email: {\tt henry.echeverria@postgrado.uv.cl}. Supported by ANID BECAS/DOCTORADO NACIONAL 21231147.
}
\and 
Owen Henderschedt\footnote{Department of Mathematics, Iowa State University, Ames, IA, U.S.A.
  Email: {\tt owenhen@iastate.edu}.}}

\begin{document}
\date{}
\maketitle

\begin{abstract}
We prove that every $(2K_2,K_4)$-free graph is recolorable. Equivalently, for every such graph $G$ and every $\ell\geq \chi(G)+1$, the reconfiguration graph of proper $\ell$-colorings of $G$, in which two colorings are adjacent if they differ on exactly one vertex, is connected. This resolves the final remaining open case in the classification of recolorable $(F_1,F_2)$-free graphs when $F_1$ and $F_2$ have at most four vertices.
\end{abstract}

\section{Introduction}\label{sec:intro}

In this paper, all graphs are finite and simple. Given an integer $\ell$, a \textit{proper $\ell$-coloring} of $G$ assigns one of $\ell$ colors to each vertex so that adjacent vertices receive distinct colors. The \textit{chromatic number} of $G$, denoted by $\chi(G)$, is the minimum $\ell$ for which $G$ admits a proper $\ell$-coloring. Beyond the existence of proper colorings, one may ask how the set of all proper colorings is structured.

One way to investigate this is through the \textit{reconfiguration graph} $\mathcal{R}_{\ell}(G)$, whose vertices are the proper $\ell$-colorings of $G$, where two colorings are adjacent if they differ at exactly one vertex. We say that $G$ is \textit{recolorable} if $\mathcal{R}_{\ell}(G)$ is connected for every $\ell>\chi(G)$. The case $\ell=\chi(G)$ is excluded, since connectivity of $\mathcal{R}_{\ell}(K_{\ell})$ fails. For more on coloring reconfiguration, see~\cite{vandenHeuvel2013}.

Determining whether a graph is recolorable is computationally difficult: even deciding whether $\mathcal{R}_3(G)$ is connected for a bipartite graph is $\mathsf{coNP}$-complete \cite{CvHJ09}. Thus, it is natural to study recolorability within structural graph classes.

Given graphs $F_1,\dots,F_k$, we say that $G$ is $(F_1,\dots,F_k)$-free if $G$ contains no induced subgraph isomorphic to any $F_i$. When $k=1$, we simply say that $G$ is $F$-free. The recolorability of $F$-free graphs has been studied extensively over the past decade \cite{BCM24,BB18, BJLPP, CerecedaHeuvelJohnson08,Merkel22}, culminating in a complete characterization: every $F$-free graph is recolorable if and only if $F$ is an induced subgraph of $P_4$ or $P_3+P_1$ \cite{BCM24}. See \cite{EH26} for a more detailed discussion.

Much less is known when two induced subgraphs are forbidden. Belavadi, Cameron, and Merkel~\cite{BCM24} showed that every $(2K_2,C_4)$-free graph is recolorable, and Belavadi and Cameron~\cite{BelavadiCameron2024} investigated pairs $(F_1,F_2)$ where both graphs have at most four vertices. Combining these results with prior work leaves only the pair $(2K_2,K_4)$ unresolved in this range. In this paper, we solve this remaining case by proving the following theorem.

\begin{maintheorem}\label{thm:YES}
Every $(2K_2,K_4)$-free graph is recolorable.
\end{maintheorem}

The coloring structure of $(2K_2,K_4)$-free graphs gives several reductions that we will use throughout the paper. Gaspers and Huang~\cite{GH19} proved that such graphs are $4$-colorable; and it is best possible, as the wheel $W_5$, obtained from a $C_5$ by adding a vertex adjacent to all five vertices of the cycle, is $(2K_2,K_4)$-free and $4$-chromatic.

Let $H_4$ denote the graph obtained from $W_5$ by deleting one spoke edge. For recoloring, we will use the following three results. The first is due to Belavadi et al.~\cite{BCH25}, while the latter two were proved by the authors in~\cite{EH26}.

\begin{maintheorem}\label{thm:reductions}
The following hold.
\begin{enumerate}[label=\textup{(\roman*)},leftmargin=2.6em]
\item\label{it:3colorable} \cite[Theorems~5 and~8]{BCH25} Every $3$-colorable $(2K_2,K_4)$-free graph is recolorable.
\item\label{it:C5recolor} \cite[Theorem~1.5]{EH26} Every $(2K_2,K_4,C_5)$-free graph is recolorable.
\item\label{it:H4wheel} \cite[Theorem~1.7]{EH26} Every $(2K_2,K_4,H_4)$-free graph containing a $W_5$ is recolorable.
\end{enumerate}
\end{maintheorem}

We now describe the reductions used to prove \cref{thm:YES}. By \cref{thm:reductions}\ref{it:3colorable}, together with the $4$-colorability of $(2K_2,K_4)$-free graphs, any counterexample to \cref{thm:YES} must be $4$-chromatic. By \cref{thm:reductions}\ref{it:C5recolor}, it must also contain an induced $C_5$. We then successively narrow the structure of such a counterexample through the following three key lemmas.

\begin{mainlemma}\label{lem:H4}
Every $(2K_2,K_4,H_4)$-free graph is recolorable.
\end{mainlemma}

The proof of \cref{lem:H4} uses \cref{thm:reductions}\ref{it:H4wheel}: if an $H_4$-free graph contains a $W_5$, then it is already recolorable, leaving only the $W_5$-free case. Thus, after \cref{lem:H4}, any remaining counterexample must contain an induced $H_4$. Two further configurations on $10$ and $11$ vertices, denoted by $B_{10}$ and $B_{11}$ and defined later, arise naturally as we extend this induced $H_4$. Their precise structure is not needed here.

The next two reductions are stated for reduced graphs and concern connectivity of the $5$-coloring reconfiguration graph. This is exactly the form needed in the final minimal-counterexample argument; the reduction to this setting is explained in \cref{sec:ReductionsAndStructure}, where we also define reduced graphs.

\begin{mainlemma}\label{lem:B10}
If a reduced $(2K_2,K_4)$-free graph $G$ contains an induced copy of $B_{10}$, then $\mathcal R_5(G)$ is connected.
\end{mainlemma}

\begin{mainlemma}\label{lem:B11}
If a reduced $(2K_2,K_4,B_{10})$-free graph $G$ contains an induced copy of $B_{11}$, then $\mathcal R_5(G)$ is connected.
\end{mainlemma}

As explained in \cref{sec:ReductionsAndStructure}, these two lemmas are sufficient to exclude $B_{10}$ and $B_{11}$ from the minimal counterexample considered in the final proof. Consequently, after applying these reductions, it remains only to consider $(2K_2,K_4,B_{10},B_{11})$-free graphs that contain an $H_4$ (and hence also a $C_5$). This final structural case completes the proof of \cref{thm:YES}. The dependencies among these reductions are summarized in \cref{fig:proof-structure}.
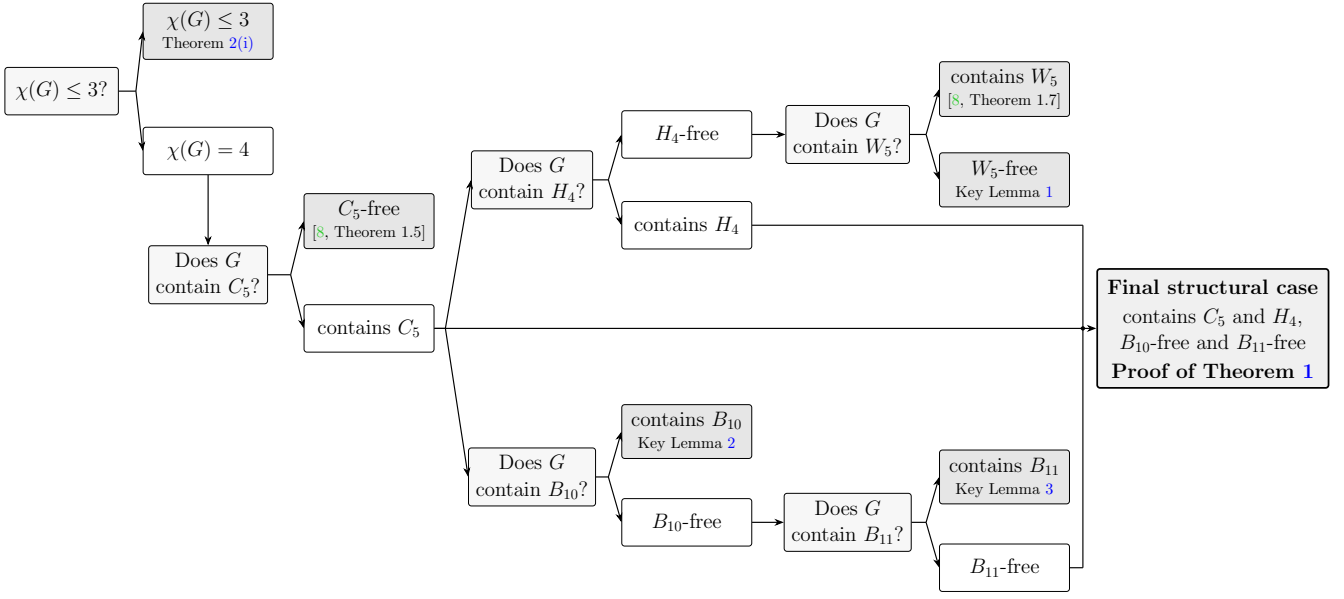
\begin{figure}[h!]
\centering
\resizebox{\textwidth}{!}{%
\begin{tikzpicture}[
    >=Stealth,
    every node/.style={font=\Large},
    question/.style={
        draw,
        rounded corners=2pt,
        fill=black!3,
        minimum width=2.75cm,
        minimum height=1.15cm,
        align=center,
        inner sep=5pt
    },
    closed/.style={
        draw,
        rounded corners=2pt,
        fill=black!10,
        minimum width=3.15cm,
        minimum height=1.15cm,
        align=center,
        inner sep=5pt
    },
    survive/.style={
        draw,
        rounded corners=2pt,
        fill=white,
        minimum width=3.15cm,
        minimum height=1.15cm,
        align=center,
        inner sep=5pt
    },
    final/.style={
        draw,
        very thick,
        rounded corners=2pt,
        fill=black!6,
        minimum width=4.75cm,
        minimum height=1.80cm,
        align=center,
        inner sep=7pt
    },
    arr/.style={->,thick}
]


\node[question] (chiq) at (0,0)
{$\chi(G)\leq 3$?};

\node[closed] (chi3) at (3.55,1.45)
{$\chi(G)\leq 3$\\[-1mm]
{\normalsize \cref{thm:reductions}\ref{it:3colorable}}};

\node[survive] (chi4) at (3.55,-1.45)
{$\chi(G)=4$};

\coordinate (chisplit) at (1.8,0);

\draw[thick] (chiq.east) -- (chisplit);
\draw[arr] (chisplit) -- (chi3.west);
\draw[arr] (chisplit) -- (chi4.west);


\node[question] (c5q) at (3.55,-4.45)
{Does $G$\\ contain $C_5$?};

\draw[arr] (chi4.south) -- ++(0,-0.70) -| (c5q.north);

\node[closed] (c5no) at (7.45,-3.15)
{$C_5$-free\\[-1mm]
{\normalsize \cite[Theorem~1.5]{EH26}}};

\node[survive] (c5yes) at (7.45,-5.75)
{contains $C_5$};

\coordinate (c5split) at (5.55,-4.45);

\draw[thick] (c5q.east) -- (c5split);
\draw[arr] (c5split) -- (c5no.west);
\draw[arr] (c5split) -- (c5yes.west);


\node[question] (h4q) at (11.4,-2.15)
{Does $G$\\ contain $H_4$?};

\node[question] (b10q) at (11.4,-9.35)
{Does $G$\\ contain $B_{10}$?};

\coordinate (structsplit) at (9.3,-5.75);

\draw[thick] (c5yes.east) -- (structsplit);
\draw[arr] (structsplit) -- (h4q.west);
\draw[arr] (structsplit) -- (b10q.west);


\node[survive] (h4no) at (15.15,-1.05)
{$H_4$-free};

\node[survive] (h4yes) at (15.15,-3.25)
{contains $H_4$};

\coordinate (h4split) at (13.25,-2.15);

\draw[thick] (h4q.east) -- (h4split);
\draw[arr] (h4split) -- (h4no.west);
\draw[arr] (h4split) -- (h4yes.west);


\node[question] (w5q) at (19.05,-1.05)
{Does $G$\\ contain $W_5$?};

\draw[arr] (h4no.east) -- (w5q.west);

\node[closed] (w5yes) at (22.85,0.05)
{contains $W_5$\\[-1mm]
{\normalsize \cite[Theorem~1.7]{EH26}}};

\node[closed] (w5no) at (22.85,-2.15)
{$W_5$-free\\[-1mm]
{\normalsize \cref{lem:H4}}};

\coordinate (w5split) at (20.95,-1.05);

\draw[thick] (w5q.east) -- (w5split);
\draw[arr] (w5split) -- (w5yes.west);
\draw[arr] (w5split) -- (w5no.west);


\node[closed] (b10yes) at (15.15,-8.25)
{contains $B_{10}$\\[-1mm]
{\normalsize \cref{lem:B10}}};

\node[survive] (b10no) at (15.15,-10.45)
{$B_{10}$-free};

\coordinate (b10split) at (13.25,-9.35);

\draw[thick] (b10q.east) -- (b10split);
\draw[arr] (b10split) -- (b10yes.west);
\draw[arr] (b10split) -- (b10no.west);


\node[question] (b11q) at (19.05,-10.45)
{Does $G$\\ contain $B_{11}$?};

\draw[arr] (b10no.east) -- (b11q.west);

\node[closed] (b11yes) at (22.85,-9.35)
{contains $B_{11}$\\[-1mm]
{\normalsize \cref{lem:B11}}};

\node[survive] (b11no) at (22.85,-11.55)
{$B_{11}$-free};

\coordinate (b11split) at (20.95,-10.45);

\draw[thick] (b11q.east) -- (b11split);
\draw[arr] (b11split) -- (b11yes.west);
\draw[arr] (b11split) -- (b11no.west);


\node[final] (final) at (27.9,-5.75)
{\textbf{Final structural case}\\[2pt]
contains $C_5$ and $H_4$,\\
$B_{10}$-free and $B_{11}$-free\\[2pt]
\textbf{Proof of \cref{thm:YES}}};

\coordinate (mergeTop) at (24.75,-3.25);
\coordinate (mergeMid) at (24.75,-5.75);
\coordinate (mergeBot) at (24.75,-11.55);

\draw[thick] (mergeTop) -- (mergeBot);

\draw[thick] (h4yes.east) -- (mergeTop);
\draw[thick] (c5yes.east) -- (mergeMid);
\draw[thick] (b11no.east) -- (mergeBot);

\draw[arr] (mergeMid) -- (final.west);

\fill (mergeMid) circle (1.5pt);

\end{tikzpicture}%
}
\caption{Structural reductions for a $(2K_2,K_4)$-free graph $G$ in the proof of \cref{thm:YES}. Shaded boxes indicate branches eliminated by the cited result; all containment is induced.}
\label{fig:proof-structure}
\end{figure}

We now outline the rest of the paper. In \cref{sec:ReductionsAndStructure}, we present additional recoloring reductions and structural lemmas for $(2K_2,K_4)$-free graphs that will be used throughout the proofs of the three key Lemmas and \cref{thm:YES}. In \cref{sec:H4}, we prove \cref{lem:H4}. To prove \cref{lem:B10} and \cref{lem:B11}, we then develop some new terminology and machinery in \cref{sec:OverviewAndDefinitions}. Since these ideas are new, we begin with a high-level overview in \cref{subsec:highlevel}, followed by the precise definitions and supporting lemmas in \cref{subsec:definitions}. As some of the resulting case analysis is cumbersome to verify by hand, we use computer-assisted verification, which we discuss in \cref{subsec:computer}, followed by a concrete example in \cref{subsec:example}. In \cref{sec:B10-B11}, we prove \cref{lem:B10} and \cref{lem:B11}, and finally prove \cref{thm:YES} in \cref{sec:main-proof}.

\section{Recoloring reductions and the structure of $(2K_2,K_4)$-free graphs}\label{sec:ReductionsAndStructure}

For distinct vertices $u,v\in V(G)$, we write $u\preceq_G v$ if $N_G(u)\subseteq N_G(v)$. We say that $u$ and $v$ are \emph{comparable} if either $u\preceq_G v$ or $v\preceq_G u$. A graph $G$ is \emph{reduced} if it contains no pair of comparable vertices. When the graph is clear from context, we simply write $u\preceq v$.

\begin{lemma}[Lemma 3 in~\cite{BelavadiCameron2024}]\label{lem:reduced}
Let $u,v\in V(G)$ satisfy $u\preceq_G v$. If $G-u$ is recolorable, then $G$ is recolorable.
\end{lemma}

\begin{lemma}[Lemma 2.2 in~\cite{EH26}]\label{lem:hereditary}
Let $\mathcal{G}$ be a hereditary class of graphs. If $\mathcal{R}_{\chi(G)+1}(G)$ is connected for every $G\in\mathcal{G}$, then $\mathcal{R}_{\ell}(G)$ is connected for every $G\in\mathcal{G}$ and every $\ell\geq\chi(G)+1$.
\end{lemma}

For recoloring $4$-chromatic $(2K_2,K_4)$-free graphs, we use the following consequence of \cref{lem:hereditary} and \cref{thm:reductions}\ref{it:3colorable}.

\begin{corollary}\label{cor:mainrecolor}
Let $G$ be a $4$-chromatic $(2K_2,K_4)$-free graph. Suppose that for every
$5$-coloring $\varphi$ of $G$, there is a $5$-coloring $\psi$ reachable from
$\varphi$ such that some color class of $\psi$ is contained in an independent
set $I$ with $\chi(G-I)\leq 3$. Then $\mathcal R_5(G)$ is connected.
\end{corollary}

\begin{proof}
Let $\varphi$ be a $5$-coloring. By hypothesis, $\varphi$ can be recolored to
a $5$-coloring $\psi$ having a color class $S$ and an independent set $I$ containing $S$ with $\chi(G-I)\le 3$. Recolor the vertices of $I\setminus S$ one at a time with the color of $S$. Since $I$ is independent and every vertex with the color of $S$ lies in $I$, every recoloring is valid, and we obtain a $5$-coloring in which $I$ is a color class. Since $G-I$ is $3$-colorable and $(2K_2,K_4)$-free, Theorem 2~\ref{it:3colorable} applied with the four remaining colors, allows us to recolor $G-I$ using only three colors. Thus $\varphi$ can be recolored to a $4$-coloring.

It remains to show that all $4$-colorings lie in the same component of $R_5(G)$. Let $\varphi_1$ and $\varphi_2$ be two $4$-colorings of $G$, and let $\beta$ be a color not used by $\varphi_1$. Note, that color is the fifth color of the 5-coloring. If $\varphi_2$ uses $\beta$,
let $I$ be the color class of $\varphi_2$ with color $\beta$; otherwise, let
$I$ be any color class of $\varphi_2$. Starting from $\varphi_1$, recolor the
vertices of $I$ one at a time with $\beta$. Now, we know that the graph $G-I$ is $3$-colorable, and $\varphi_2$ restricted to $G-I$ is a 3-coloring that does not use $\beta$, so by Theorem 2~\ref{it:3colorable}, applied with the four colors other than $\beta$, the current coloring of $G-I$ which is $\varphi_1$ restricted to $G-I$ can be recolored to $\varphi_2$ restricted to $G-I$. If $I$ is not already colored as in \(\varphi_2\), recolor its vertices one at a time to their color in \(\varphi_2\); these recolorings are valid because \(I\) is independent and \(G-I\) now agrees with \(\varphi_2\). Hence all $4$-colorings lie in one component of $\mathcal R_5(G)$. Since every $5$-coloring can be recolored to a $4$-coloring, it follows that
$\mathcal R_5(G)$ is connected.
\end{proof}

The following notation will be used throughout the entirety of the paper. Let $G$ be a $(2K_2,K_4)$-free graph containing an induced $C_5$ and denote the set of vertices of a 5-cycle cycle $C= c_1c_2c_3c_4c_5$. We now define the following sets, with relation to $C$, where indices are taken modulo $5$.
\begin{align*}
    U &= \{v\in V(G)\setminus V(C): N(v) \cap V(C) = V(C)\},\\
    F_i &= \{v\in V(G)\setminus V(C): N(v) \cap V(C) = V(C)\setminus \{c_i\}\},\\
    Y_i &= \{v\in V(G)\setminus V(C): N(v)\cap V(C) = \{c_i,c_{i+2},c_{i-2}\}\},\\
    R_i &= \{v\in V(G)\setminus V(C): N(v)\cap V(C) = \{c_{i-1},c_{i+1}\}\},\\
    Z &= \{v\in V(G)\setminus V(C): N(v)\cap V(C) = \varnothing\}.
\end{align*}
Most of the following structural properties were first proved by Gaspers and Huang~\cite{GH19}. Since they are simple and used repeatedly throughout the arguments, the properties needed here are also collected and proved in~\cite{EH26}.

\begin{lemma}[\cite{EH26,GH19}]\label{lem:basic}
With the notation above, let $G$ be a $(2K_2,K_4)$-free graph with an induced $C_5$. Then the following hold for all $i\in [5]$ where indices are taken modulo $5$.
\begin{enumerate}[label=\textup{(\roman*)},leftmargin=2.6em]
\item\label{it:indep} $U$, $Z$, and each of $F_i$, $Y_i$, and $R_i$ are independent.
\item\label{it:UF} $U$ is anticomplete to $F_i\cup Y_i$.
\item\label{it:FY} $F_i$ is complete to $Y_{i+2}\cup Y_{i-2}$ and anticomplete to $Y_{i-1}\cup Y_i\cup Y_{i+1}$.
\item\label{it:FF} At least one of $F_i$ and $F_{i+2}$ is empty, and $F_i$ is anticomplete to $F_j$ for $j\neq i$.
\item\label{it:YY} $Y_i$ is complete to $Y_{i+1}$.
\item \label{it:YY+2} Each vertex of $Y_i$ is anticomplete to at least one of $Y_{i-2}$, $Y_{i+2}$.
\item\label{it:RR} $R_i$ is complete to $R_{i+1}\cup R_{i-1}$, and $Y_i$ is complete to $R_i$.
\item\label{it:FR} $F_i$ is complete to $R_{i-1}\cup R_{i+1}$.
\item\label{it:ZR} $Z$ is anticomplete to $R_i$.
\item \label{itH4F} If $G$ is $H_4$-free, then for all $i\in [5]$ the set $F_i$ is empty.
\item \label{itH4RY} If $G$ is $H_4$-free, then $R_i$ is anticomplete to $Y_{i-1}\cup Y_{i+1}$
\end{enumerate}
\end{lemma}

\section{Proof of \cref{lem:H4}}\label{sec:H4}

We now prove \cref{lem:H4}.

\begin{proof}

Suppose for a contradiction that the lemma is false, and let $G$ be a vertex-minimal nonrecolorable $(2K_2,K_4,H_4)$-free graph. Since every proper induced subgraph of $G$ is recolorable, \cref{lem:reduced} implies that $G$ is reduced. Moreover, since $\chi(G)\leq 4$ by Gaspers and Huang~\cite{GH19}, and every $3$-colorable $(2K_2,K_4)$-free graph is recolorable by \cref{thm:reductions}\ref{it:3colorable}, we have $\chi(G)=4$.

We claim that it suffices to show that $\mathcal R_5(G)$ is connected. Indeed, let $\mathcal G$ be the hereditary class consisting of all induced subgraphs of $G$. Every proper induced subgraph $H$ of $G$ is recolorable by the minimality of $G$, and hence $\mathcal R_{\chi(H)+1}(H)$ is connected. Thus, if $\mathcal R_5(G)$ is connected, then $\mathcal R_{\chi(H)+1}(H)$ is connected for every $H\in\mathcal G$. By \cref{lem:hereditary}, every graph in $\mathcal G$, and in particular $G$, is recolorable, a contradiction.

By \cref{thm:reductions}\ref{it:C5recolor} and \cref{thm:reductions}\ref{it:H4wheel}, we may therefore assume that $G$ contains an induced $C_5$ and is $W_5$-free. Let $C=c_1c_2c_3c_4c_5c_1$ be an induced $5$-cycle. Since $G$ is $(H_4,W_5)$-free, we have $F_i=\varnothing$ for all $i\in[5]$ and
$U=\varnothing$.

\begin{claim}\label{claim:triangled}
Every triangle in $G$ is dominating.
\end{claim}

\begin{proofc}
Suppose that $abc$ is a triangle in $G$ and $v$ is anticomplete to $\{a,b,c\}$. Let $N_a(v)$, $N_b(v)$, and $N_c(v)$ denote the sets of neighbors of $v$ which are nonadjacent to $a$, $b$, and $c$, respectively.  We first claim that these three sets partition $N(v)$. If $z\in N(v)$ belongs to none of them, then $z$ is adjacent to each of $a,b,c$, so $\{z,a,b,c\}$ induces a $K_4$. On the other hand, by symmetry, if $z\in N_a(v)\cap N_b(v)$, then the edges $vz$ and $ab$ induce a $2K_2$. Next, $N_a(v)$, $N_b(v)$, and $N_c(v)$ are pairwise complete. Otherwise, by symmetry, there exist nonadjacent vertices $z_a\in N_a(v)$ and $z_b\in N_b(v)$. Then $abz_avz_ba$ is an induced $C_5$, and $c$ is adjacent to four vertices of this cycle, producing an induced $H_4$. At least one of $N_a(v)$, $N_b(v)$, and $N_c(v)$ must be empty. Indeed, if all three are nonempty, then choosing one vertex from each set gives three pairwise adjacent neighbors of $v$, and together with $v$ they induce a $K_4$. On the other hand, if, by symmetry, $N_a(v)=\varnothing$, then every neighbor of $v$ is adjacent to $a$, so $N(v)\subseteq N(a)$. Since $va\notin E(G)$, this contradicts reducedness.
\end{proofc}

\textbf{Splitting the sets $R_i$.}
We now further partition each $R_i$. Recall from \cref{lem:basic}\ref{itH4RY} that $R_i$ is anticomplete to $Y_{i-1}\cup Y_{i+1}$.

First, every vertex of $R_i$ has a neighbor in exactly one of $R_{i-2}$ and $R_{i+2}$. Indeed, suppose that $r\in R_i$ is anticomplete to $R_{i-2}\cup R_{i+2}$. Using \cref{lem:basic} and the fact that $U=F_j=\varnothing$ for every $j$, we have
\[N(r)\subseteq \{c_{i-1},c_{i+1}\}\cup Y_i\cup Y_{i-2}\cup Y_{i+2}\cup R_{i-1}\cup R_{i+1}\subseteq N(c_i),\]
contradicting reducedness.

Now suppose that $r$ has neighbors $r_{i+2}\in R_{i+2}$ and $r_{i-2}\in R_{i-2}$. By \cref{lem:basic}\ref{it:RR}, the vertices $r_{i+2}$ and $r_{i-2}$ are adjacent, so $r,r_{i+2},r_{i-2}$ induce a triangle. However, $c_i$ is nonadjacent to all three vertices, contradicting \cref{claim:triangled}.

For $j\in\{i-2,i+2\}$, let $R_i^j$ denote the vertices of $R_i$ having a neighbor in $R_j$. We have
\begin{equation}\label{eq:2}
R_i=R_i^{i-2}\mathbin{\dot\cup}R_i^{i+2},\qquad R_i^{i-2}\text{ is complete to }Y_{i+2},\qquad R_i^{i+2}\text{ is complete to }Y_{i-2}.
\end{equation}

The first assertion follows from the previous paragraph. The other two are symmetric, so it is enough to prove that $R_i^{i-2}$ is complete to $Y_{i+2}$. Suppose otherwise that $r\in R_i^{i-2}$ is nonadjacent to some $y\in Y_{i+2}$. Choose $r_{i-2}\in N(r)\cap R_{i-2}$. By \cref{lem:basic}\ref{itH4RY}, $r_{i-2}$ is nonadjacent to $y$, while both $r$ and $r_{i-2}$ are nonadjacent to $c_i$. Since $yc_i\in E(G)$, the edges $rr_{i-2}$ and $yc_i$ induce a $2K_2$, a contradiction.

\begin{claim}
    The set $Z$ is empty.
\end{claim}

\begin{proofc}
Suppose for contradiction that $z\in Z$. Then $z$ is complete to $Y_i$ for all $i\in [5]$. Otherwise, if $y\in Y_i$ is nonadjacent to $z$, then $y,c_{i-2},c_{i+2}$ form a triangle anticomplete to $z$, contradicting \cref{claim:triangled}. Thus, by \cref{lem:basic}\ref{it:indep} and \ref{it:ZR}, any two vertices of $Z$ have the same neighborhood and are therefore comparable, contradicting the reducedness of $G$. Hence we may assume that $Z=\{z\}$.

Suppose that there is an edge between $R_i$ and $R_{i+2}$. Then such an edge together with $c_{i+1}$ forms a triangle anticomplete to $z$, again contradicting \cref{claim:triangled}. Thus $R_i$ is anticomplete to $R_{i+2}\cup R_{i-2}$ for every $i\in[5]$, and by \eqref{eq:2}, we have $R_i=\varnothing$ for all $i\in[5]$.

At least one $Y_i$ is nonempty; otherwise $z$ is an isolated vertex and hence $N(z)\subseteq N(c_1)$, contradicting reducedness. Let $y\in Y_i$. We have $N(c_{i-1})=\{c_{i-2},c_i\}\cup Y_{i-1}\cup Y_{i+1}\cup Y_{i+2}$. Moreover, $y$ is adjacent to $c_{i-2}$ and $c_i$, and by \cref{lem:basic}\ref{it:YY}, $y$ is complete to $Y_{i-1}\cup Y_{i+1}$. Therefore, if $y$ were complete to $Y_{i+2}$, then $N(c_{i-1})\subseteq N(y)$, contradicting reducedness. Hence there exists some $y'\in Y_{i+2}$ nonadjacent to $y$, and in particular $Y_{i+2}\neq\varnothing$. Repeating this argument cyclically around the cycle gives $Y_i\neq\varnothing$ for all $i\in[5]$.

We claim that $Y_i$ and $Y_{i+2}$ are anticomplete for every $i\in[5]$. Otherwise, let $y_i\in Y_i$ and $y_{i+2}\in Y_{i+2}$ be adjacent. Since $Y_{i+1}\neq\varnothing$, choose $y_{i+1}\in Y_{i+1}$. By \cref{lem:basic}\ref{it:YY}, the vertices $y_i,y_{i+1},y_{i+2}$ form a triangle. Since $z$ is adjacent to all three, these four vertices induce a $K_4$, a contradiction.

Now choose $y_i\in Y_i$ for every $i\in[5]$. By \cref{lem:basic}\ref{it:YY}, $y_i$ is adjacent to $y_{i+1}$ for every $i$, while the previous paragraph shows that $y_i$ is nonadjacent to $y_{i+2}$. Thus $y_1y_2y_3y_4y_5y_1$ is an induced $5$-cycle. Since $z$ is adjacent to every $y_i$, these six vertices induce a $W_5$, contradicting that $G$ is $W_5$-free. Therefore $Z=\varnothing$.
\end{proofc}

Fix an arbitrary $5$-coloring $\varphi$ of $G$.

\begin{claim}\label{claim:nonrainbowC5}
If some $5$-coloring reachable from $\varphi$ colors $C$ non-rainbow, then $\varphi$ can be recolored to a $5$-coloring having a color class $I$ with $\chi(G-I)\leq 3$.
\end{claim}

\begin{proofc}
Suppose that after some sequence of recolorings, $C$ is not rainbow. We continue to denote the current coloring by $\varphi$. Up to symmetry, we can assume that $\varphi(c_1)=\varphi(c_4)=\alpha$. Then the only vertices adjacent to neither $c_1$ nor $c_4$, and hence the only vertices that can possibly be colored $\alpha$ together with $c_1$ and $c_4$, are
\[
\{c_1,c_4\}\cup Y_5\cup R_1\cup R_4.
\]
By \cref{lem:basic}\ref{it:indep}, \ref{itH4RY}, and \eqref{eq:2}, we can recolor
\[
Y_5\cup R_1^3\cup R_4^2 \to \alpha.
\]
Indeed, these sets are independent and pairwise anticomplete, and every vertex in these sets has no neighbor already colored $\alpha$. Next, recolor to $\alpha$ as many vertices as possible from $R_1^4\cup R_4^1$. Therefore, every vertex not colored $\alpha$ in $R_1^4\cup R_4^1$ is adjacent to a vertex colored $\alpha$. Let $I$ be the set of all vertices colored $\alpha$, and define
\[
P=R_1^4\setminus I \qquad \text{and} \qquad Q=R_4^1\setminus I.
\]
By maximality, every $p\in P$ has an $\alpha$-colored neighbor. Since $p\in R_1^4$, the only possible such neighbor lies in $R_4^1$; choose one and call it $p'$. Similarly, every $q\in Q$ has an $\alpha$-colored neighbor $q'\in R_1^4$.

We claim that $P$ is complete to $Q$. Otherwise, let $p\in P$ and $q\in Q$ be nonadjacent, and choose $\alpha$-colored vertices $p'\in R_4^1$ and $q'\in R_1^4$ with $pp',qq'\in E(G)$. Since $R_1$ and $R_4$ are independent, we have $pq',p'q\notin E(G)$, while $p'q'\notin E(G)$ because $p'$ and $q'$ are both colored $\alpha$. Thus the edges $pp'$ and $qq'$ induce a $2K_2$, a contradiction. Hence $P$ is complete to $Q$.

We now define the following three independent sets. We delay the verification, since this follows directly from the definitions and \cref{lem:basic}, and we will soon add to these sets and verify that they remain independent.
\begin{align*}
    J_1^0 = \{c_3\}\cup P \cup R_3 \cup Y_2 \qquad  J_2^0 = \{c_2\}\cup Q\cup R_2\cup Y_3 \qquad J_3^0 = \{c_5\} \cup R_5.
\end{align*}

Note that the $J_i^0$ independent sets cover all of $G\setminus (I\cup Y_1 \cup Y_4)$. Before applying \cref{cor:mainrecolor}, we must cleverly partition $Y_1$ and $Y_4$ in order to add them to the $J_i^0$ independent sets. To this end, let
\begin{align*}
    B = \{b\in Y_4: N(b)\cap (P\cup Y_2) \neq \varnothing\} \qquad \text{and} \qquad A = N(B) \cap Y_1.
\end{align*}
We now show that $A$ is anticomplete to $Q\cup Y_3$. Let $a\in A$ and choose $b\in B\cap N(a)$, which exists by definition of $A$. Since $ab\in E(G)$, \cref{lem:basic}\ref{it:YY+2} implies that $a$ is anticomplete to $Y_3$ and $b$ is anticomplete to $Y_2$. Thus, since $b\in B$, it must have a neighbor $p\in P$. Suppose that $a$ has a neighbor $q\in Q$. Then $\{a,b,p,q\}$ induces a $K_4$: the edges $ab,aq,bp$ hold by construction, $pq\in E(G)$ since $P$ is complete to $Q$, and $ap,bq\in E(G)$ by \cref{lem:basic}\ref{it:RR}. This is a contradiction.

Now, we can add $Y_1\cup Y_4$ to the $J^0$ independent sets as follows.
\begin{align*}
    J_1 = J_1^0 \cup (Y_4\setminus B)  \qquad  J_2 = J_2^0\cup A  \qquad J_3 = J_3^0 \cup B \cup (Y_1\setminus A).
\end{align*}
We claim that these sets are independent.

\textbf{The set $\mathbf{J_1=\{c_3\}\cup P\cup R_3\cup Y_2\cup(Y_4\setminus B)}$.}
We verify that $J_1$ is independent. By the definitions of the five-cycle types, $c_3$ is nonadjacent to every vertex of $P\cup R_3\cup Y_2\cup Y_4$. Moreover, $P\subseteq R_1^4$ is anticomplete to $R_3$ by \eqref{eq:2}, while $P$ is anticomplete to $Y_2$, and $R_3$ is anticomplete to $Y_2\cup Y_4$, by \cref{lem:basic}\ref{itH4RY}. Finally, by the definition of $B$, every vertex of $Y_4\setminus B$ is anticomplete to $P\cup Y_2$. Since each of $R_1,R_3,Y_2,Y_4$ is independent by \cref{lem:basic}\ref{it:indep}, it follows that $J_1$ is independent.

\textbf{The set $\mathbf{J_2=\{c_2\}\cup Q\cup R_2\cup Y_3\cup A}$.}
We verify that $J_2$ is independent. By the definitions of the five-cycle types, $c_2$ is nonadjacent to every vertex of $Q\cup R_2\cup Y_3\cup A$. Moreover, $Q\subseteq R_4^1$ is anticomplete to $R_2$ by \eqref{eq:2}, while $Q$ is anticomplete to $Y_3$, and $R_2$ is anticomplete to $Y_1\cup Y_3$, by \cref{lem:basic}\ref{itH4RY}. Since $A\subseteq Y_1$, this also shows that $R_2$ is anticomplete to $A$. Finally, as shown above, $A$ is anticomplete to $Q\cup Y_3$. Since each of $R_4,R_2,Y_3,Y_1$ is independent by \cref{lem:basic}\ref{it:indep}, it follows that $J_2$ is independent.

\textbf{The set $\mathbf{J_3=\{c_5\}\cup R_5\cup B\cup(Y_1\setminus A)}$.}
We verify that $J_3$ is independent. By the definitions of the five-cycle types, $c_5$ is nonadjacent to every vertex of $R_5\cup B\cup(Y_1\setminus A)$. Moreover, $R_5$ is anticomplete to $Y_4\cup Y_1$ by \cref{lem:basic}\ref{itH4RY}, and hence to $B\cup(Y_1\setminus A)$. Finally, since $A=N(B)\cap Y_1$, every vertex of $Y_1\setminus A$ is nonadjacent to every vertex of $B$. Since each of $R_5,Y_4,Y_1$ is independent by \cref{lem:basic}\ref{it:indep}, it follows that $J_3$ is independent.

Thus $I$ is the monochromatic $\alpha$-color class, and the independent sets $J_1,J_2,J_3$ cover all of $G\setminus I$. Therefore $\chi(G\setminus I)\leq 3$, proving the claim.
\end{proofc}

\begin{claim}\label{claim:make-nonrainbow}
There is a $5$-coloring reachable from $\varphi$ for which $C$ is not rainbow.
\end{claim}

\begin{proofc}
If $C$ is already non-rainbow under $\varphi$, then there is nothing to prove. Thus, suppose that $C$ is rainbow. Write $\varphi(c_i)=\alpha_i$ for $i\in\{1,2,3,4,5\}$, where the colors $\alpha_1,\ldots,\alpha_5$ are pairwise distinct.

For $k\in\{1,2,3,4,5\}$, define
$$B_k=\{v\in Y_{k-1}\cup R_k:\varphi(v)=\alpha_{k-2}\}.$$
Clearly, each $B_k$ is monochromatic and the sets $B_k$ are pairwise disjoint. If $B_k=\varnothing$ for some $k$, then we can recolor $c_{k+1}\to\alpha_{k-2}$. Indeed,
$$N(c_{k+1})=\{c_k,c_{k+2}\}\cup R_k\cup R_{k+2}\cup Y_{k-1}\cup Y_{k+1}\cup Y_{k+3}.$$
The vertices $c_k,c_{k+2}$ are not colored $\alpha_{k-2}$, while every vertex of $R_{k+2}\cup Y_{k+1}\cup Y_{k+3}$ is adjacent to $c_{k-2}$ and hence is not colored $\alpha_{k-2}$. Thus the only possible blockers of $c_{k+1}\to\alpha_{k-2}$ lie in $Y_{k-1}\cup R_k$, and these are precisely the vertices of $B_k$. Hence the recoloring is legal and makes $C$ non-rainbow.

We may therefore suppose that $B_k\neq\varnothing$ for every $k$. Repeatedly recolor a vertex of some $B_k$ from $\alpha_{k-2}$ to $\alpha_k$ whenever this is legal, and continue until no such recoloring is possible. Let $B_k'\subseteq B_k$ be the vertices that remain colored $\alpha_{k-2}$. If $B_k'=\varnothing$ for some $k$, then again $c_{k+1}\to\alpha_{k-2}$ is legal and we are done. Thus we may suppose that $B_k'\neq\varnothing$ for every $k$.

Every vertex in $B_k'$ has a neighbor colored $\alpha_k$, since otherwise it could be recolored to $\alpha_k$. If $v\in Y_{k-1}\cap B_k'$, then such a neighbor must lie in $Y_{k+1}\cup R_{k+2}$, and hence lies in $B_{k+2}'$. If $v\in R_k\cap B_k'$, then such a neighbor must lie in $R_{k+2}\cup R_{k-2}$.

Suppose that $v\in R_k\cap B_k'$ has a neighbor $v'\in R_{k-2}$. Then $v\in R_k^{k-2}$, so by \eqref{eq:2}, $v$ is complete to $Y_{k+2}$, and by \cref{lem:basic}\ref{it:RR}, $v$ is complete to $R_{k+1}$. Since $\varphi(v)=\alpha_{k-2}$, no vertex in $Y_{k+2}\cup R_{k+1}$ is colored $\alpha_{k-2}$. But the only possible blockers of $c_k\to\alpha_{k-2}$ lie in $Y_{k+2}\cup R_{k+1}$. Hence this recoloring is legal and makes $C$ non-rainbow.

We may therefore suppose that no such vertex $v'$ exists. It follows that every vertex in $B_k'$ has a neighbor in $B_{k+2}'$.

By \cref{lem:basic}\ref{it:indep}, \ref{itH4RY}, the sets $B_5'$ and $B_2'$ are independent. Since $G$ is $2K_2$-free, the sets $N(x)\cap B_2'$, for $x\in B_5'$, are linearly ordered by inclusion. Indeed, otherwise there exist $x,x'\in B_5'$ and $y,y'\in B_2'$ such that $xy,x'y'\in E(G)$ and $xy',x'y\notin E(G)$. Since $B_5'$ and $B_2'$ are independent, the edges $xy$ and $x'y'$ then induce a $2K_2$, a contradiction. Choose $x_5'\in B_5'$ such that $N(x_5')\cap B_2'$ is inclusion-minimal. Starting from $x_5'$, choose
\[
x_2\in B_2'\cap N(x_5'),\quad
x_4\in B_4'\cap N(x_2),\quad
x_1\in B_1'\cap N(x_4),\quad
x_3\in B_3'\cap N(x_1),\quad
x_5\in B_5'\cap N(x_3).
\]
Each choice is possible because every vertex in $B_k'$ has a neighbor in $B_{k+2}'$. Since $x_5\in B_5'$ and $x_5'$ was chosen with an inclusion-minimal neighborhood in $B_2'$, we have
\[
N(x_5')\cap B_2'\subseteq N(x_5)\cap B_2'.
\]
Since $x_2\in N(x_5')\cap B_2'$, it follows that $x_5x_2\in E(G)$. Therefore, in the cyclic ordering $(x_1,x_2,x_3,x_4,x_5)$, we have $x_ix_{i+2}\in E(G)$ for every $i$.

Suppose first that $x_i\in R_i$ for some $i$. Since $x_{i+2}\in Y_{i+1}\cup R_{i+2}$ and $x_i$ is anticomplete to $Y_{i+1}$ by \cref{lem:basic}\ref{itH4RY}, the edge $x_ix_{i+2}$ forces $x_{i+2}\in R_{i+2}$. Repeating this around the cycle gives $x_i\in R_i$ for every $i$. Then $x_ix_{i+1}\in E(G)$ for every $i$ by \cref{lem:basic}\ref{it:RR}. Together with the edges $x_ix_{i+2}$, the vertices $x_1,\ldots,x_5$ induce a $K_5$, a contradiction.

Thus $x_i\in Y_{i-1}$ for every $i$. By \cref{lem:basic}\ref{it:YY}, we again have $x_ix_{i+1}\in E(G)$ for every $i$, and together with the edges $x_ix_{i+2}$ this gives a $K_5$, a contradiction.

Therefore, $\varphi$ can be recolored to a $5$-coloring for which $C$ is not rainbow.
\end{proofc}

By \cref{claim:make-nonrainbow}, there is a $5$-coloring reachable from $\varphi$ for which $C$ is not rainbow. By \cref{claim:nonrainbowC5}, $\varphi$ can therefore be recolored to a $5$-coloring having a color class $I$ with $\chi(G-I)\leq 3$. Since $\varphi$ was arbitrary, \cref{cor:mainrecolor} implies that $\mathcal R_5(G)$ is connected. By the hereditary class argument at the beginning of the proof, $G$ is recolorable, a contradiction.

\end{proof}

\section{Proof technique; overview and definitions}\label{sec:OverviewAndDefinitions}

To prove \cref{lem:B10} and \cref{lem:B11}, and to arrive at our final contradiction for \cref{thm:YES}, a deeper detailed analysis of the structure of reduced $(2K_2,K_4)$-free graphs containing a $5$-cycle is needed. To do so, in this section we develop the following strategy, first described at a high level, and then followed by the rigorous necessary definitions and examples of the definitions and machinery.

\subsection{Overview of proof strategy}\label{subsec:highlevel}

Let $C$ be a $5$-cycle in $G$. By the structural properties of $(2K_2,K_4)$-free graphs developed by Gaspers and Huang, and summarized in \cref{sec:ReductionsAndStructure}, every \emph{off-cycle} vertex, that is, every vertex in $V(G)\setminus V(C)$, has one of a fixed collection of possible neighborhoods on $C$. We refer to this as the \emph{type} of the vertex, coming from the set
\[\{U,Z,R_i,Y_i,F_i:i\in [5]\}.
\]

The general strategy is to begin with the cycle $C$ and reveal off-cycle vertices one at a time, thereby building a larger and larger induced subgraph $K$ of $G$. As vertices are revealed, we must keep track not only of their type, but also of their adjacencies to the off-cycle vertices that have already been added. Thus, at any stage, every unrevealed vertex is described by a \textit{profile}: its type together with its adjacencies to the currently revealed off-cycle vertices. This allows us to organize the vertices outside $K$ into a finite collection of possible attachment profiles.

Most conceivable profiles cannot actually occur. Any new vertex must attach to the current subgraph in a way that creates neither an induced $2K_2$ nor a $K_4$, and it must also respect the structural relations between the various types established in \cref{sec:ReductionsAndStructure}. These restrictions typically leave only a small collection of possible ways for an unrevealed vertex to attach to $K$.

Reducedness provides a second source of restrictions. It may happen that, within the currently revealed subgraph $K$, the neighborhood of one vertex is contained in the neighborhood of another. Such a containment cannot persist in the whole graph, since $G$ is reduced. Therefore some unrevealed vertex must distinguish the two vertices. We call such a vertex a \emph{rescuer}. The forbidden-subgraph conditions severely restrict which profiles can serve as rescuers, and in many cases no rescuer is possible at all. When this happens, the original attachment profile is eliminated.

This leads naturally to an iterative procedure. We first list the profiles that are compatible with the current subgraph and the forbidden-subgraph conditions. We then eliminate those that force a neighborhood containment with no possible rescuer. Once some patterns have been eliminated, others may lose their only possible rescuers and can be eliminated in turn. Repeating this process leaves only a small number of highly constrained configurations.

The remainder of this section makes this procedure precise. We first introduce the notation used to encode attachment patterns and rescuers, and then isolate two finite configurations that cannot occur in a reduced counterexample. These exclusions will later allow the structural arguments in \cref{sec:B10-B11} and \cref{sec:main-proof} and the subsequent case analysis to be combined cleanly.

\subsection{Definitions of profiles, rescuers, and elimination rounds}\label{subsec:definitions}

A \emph{core} is an induced subgraph $K$ of $G$ that contains an induced $5$-cycle $C$, together with an ordering $(v_1,\ldots,v_k)$ of its off-cycle vertices, where
\[V(K)\setminus V(C)=\{v_1,\ldots,v_k\}.\]
Fix such a core $K$. The purpose of the notation below is to describe the possible ways in which vertices outside $K$ can attach to it, and then to determine which of these possibilities could actually occur in the reduced graph $G$.

An \emph{attachment profile} with respect to $K$, or simply a \emph{profile} when the core is clear from context, is an expression
\[\sigma=T[\varepsilon_1\cdots\varepsilon_k],\]
where $T\in\{U,Z,R_i,Y_i,F_i:i\in\mathbb Z_5\}$ is a type and $\varepsilon_i\in\{0,1\}$ for every $i\in[k]$. The bit $\varepsilon_i$ records adjacency to $v_i$. Thus, a vertex $u\in V(G)\setminus V(K)$ has profile
\[\sigma=T[\varepsilon_1\cdots\varepsilon_k]\]
if $u$ has type $T$ and
\[uv_i\in E(G)\quad\Longleftrightarrow\quad \varepsilon_i=1\]
for every $i\in[k]$.

A profile $\sigma=T[\varepsilon_1\cdots\varepsilon_k]$ prescribes a unique neighborhood in $K$: namely, the neighbors on $C$ determined by the type $T$, together with those vertices $v_i$ for which $\varepsilon_i=1$. We write $K+u_\sigma$ for the graph obtained from $K$ by adjoining a new vertex $u_\sigma$ with precisely this neighborhood in $K$. A profile $\sigma$ is \emph{admissible} if $K+u_\sigma$ is $(2K_2,K_4)$-free.

When a vertex $u\in V(G)\setminus V(K)$ realizes $\sigma$, the graph $K+u_\sigma$ is naturally identified with the induced subgraph $G[V(K)\cup\{u\}]$. Whenever an admissible vertex is added to the core, it is appended to the ordered list of off-cycle vertices.

The following lemma gives a direct way to test admissibility. It also characterizes when two individually admissible attachments can occur together.

\begin{lemma}\label{lem:attachment-rules}
\begingroup
\setlength{\abovedisplayskip}{4pt}
\setlength{\belowdisplayskip}{4pt}

Let $K$ be a $(2K_2,K_4)$-free graph. Adjoining a new vertex $x$ to $K$ produces a $(2K_2,K_4)$-free graph if and only if
\begin{enumerate}[label=(\roman*), topsep=2pt, itemsep=5pt, parsep=0pt, partopsep=0pt]
    \item $K[N_K(x)]$ is triangle-free; and
    \item for every $u\in N_K(x)$, the set
    \[V(K)\setminus\bigl(N_K(x)\cup N_K(u)\bigr)\]
    is independent.
\end{enumerate}

Now let $x$ and $y$ be two new vertices such that $K+x$ and $K+y$ are each $(2K_2,K_4)$-free. Then:
\begin{enumerate}[label=(\alph*), topsep=2pt, itemsep=5pt, parsep=0pt, partopsep=0pt]
    \item if $x$ and $y$ are nonadjacent, then adjoining both $x$ and $y$ to $K$ preserves $(2K_2,K_4)$-freeness if and only if
    \[N_K(x)\setminus N_K(y)\text{ is complete to }N_K(y)\setminus N_K(x);\]

    \item if $x$ and $y$ are adjacent, then adjoining both $x$ and $y$ to $K$ preserves $(2K_2,K_4)$-freeness if and only if both
    \[N_K(x)\cap N_K(y)\qquad\text{and}\qquad V(K)\setminus\bigl(N_K(x)\cup N_K(y)\bigr)\]
    are independent.
\end{enumerate}

\endgroup
\end{lemma}

\begin{proof}
Since $K$ is $(2K_2,K_4)$-free, every forbidden induced subgraph created by adjoining $x$ must contain $x$. The vertex $x$ lies in a $K_4$ precisely when three vertices of $N_K(x)$ induce a triangle, which gives (i). Similarly, $x$ lies in an induced $2K_2$ precisely when there is some $u\in N_K(x)$ and an edge whose two endpoints are adjacent to neither $x$ nor $u$. The possible endpoints of such an edge are exactly the vertices in $V(K)\setminus\bigl(N_K(x)\cup N_K(u)\bigr).$
Thus no such $2K_2$ exists if and only if this set is independent for every $u\in N_K(x)$, giving (ii).

Now suppose that $K+x$ and $K+y$ are each $(2K_2,K_4)$-free. Any forbidden induced subgraph created by adjoining both vertices must therefore contain both $x$ and $y$.

Suppose first that $x$ and $y$ are nonadjacent. Then they cannot lie together in a $K_4$, so we need only consider an induced $2K_2$. Such a $2K_2$ must consist of edges $xa$ and $yb$, where $a\in N_K(x)\setminus N_K(y)$ and $b\in N_K(y)\setminus N_K(x)$,
with $a$ and $b$ nonadjacent. Hence no such $2K_2$ exists if and only if $N_K(x)\setminus N_K(y)$ is complete to $N_K(y)\setminus N_K(x)$, proving (a).

Finally, suppose that $x$ and $y$ are adjacent. They lie together in a $K_4$ precisely when their common neighborhood $N_K(x)\cap N_K(y)$ contains an edge. Thus they are contained in no $K_4$ if and only if $N_K(x)\cap N_K(y)$ is independent. Likewise, any induced $2K_2$ containing both $x$ and $y$ must consist of the edge $xy$ together with an edge whose endpoints are adjacent to neither $x$ nor $y$. Such endpoints lie exactly in $V(K)\setminus\bigl(N_K(x)\cup N_K(y)\bigr).$
Therefore no such $2K_2$ exists if and only if this set is independent. This proves (b).
\end{proof}

Admissibility alone does not guarantee that a profile can occur in the reduced graph $G$. Suppose that an admissible profile $\sigma$ is realized by a vertex $u\in V(G)\setminus V(K)$. If, in the enlarged core $K+u$, we have
\[
u\preceq_{K+u}v
\]
for some $v\in V(K)$, then this containment cannot persist in $G$, since $G$ is reduced. Hence there must exist a vertex $w\in V(G)\setminus V(K+u)$ such that
\[
wu\in E(G)\qquad\text{and}\qquad wv\notin E(G).
\]
Similarly, if
\[
v\preceq_{K+u}u,
\]
then reducedness forces a vertex $w\in V(G)\setminus V(K+u)$ such that
\[
wv\in E(G)\qquad\text{and}\qquad wu\notin E(G).
\]
In either case, the profile of $w$ with respect to $K$, together with its required adjacency or nonadjacency to $u$, must give a $(2K_2,K_4)$-free extension of $K+u$. We call such a profile a \emph{rescuer}.

Suppose that $\sigma$ is an admissible profile and that a new vertex $u$ with profile $\sigma$ has been adjoined to $K$. Consider a comparison
\[
u\preceq_{K+u}v
\]
with $v\in V(K)$. A profile $\tau$ with respect to $K$ \emph{rescues the comparison $u\preceq_{K+u}v$} if a new vertex $w$ having profile $\tau$ relative to $K$ can be adjoined to $K+u$ so that
\[
wu\in E(G),\qquad wv\notin E(G),
\]
and the resulting profile of $w$ with respect to the enlarged core $K+u$ is admissible. Similarly, $\tau$ \emph{rescues the comparison $v\preceq_{K+u}u$} if $w$ can be adjoined so that
\[
wv\in E(G),\qquad wu\notin E(G),
\]
and its resulting profile with respect to $K+u$ is admissible.

Consequently, if an admissible profile $\sigma$ is realized in the reduced graph $G$, then every comparison created by adjoining a vertex with profile $\sigma$ must have a rescuing profile that is also realized in $G$. We use this necessary condition to eliminate profiles that cannot occur.

Let $\mathcal A_0(K)$ denote the set of all admissible profiles with respect to $K$. We refine this set in rounds. Suppose that $\mathcal A_{r-1}(K)$ is the collection of profiles remaining after the first $r-1$ rounds. A profile $\sigma\in\mathcal A_{r-1}(K)$ is \emph{eliminated in round $r$} if, after adjoining a new vertex $u$ with profile $\sigma$, there is some $v\in V(K)$ such that $u$ and $v$ are comparable in $K+u$, but the resulting comparison has no rescuing profile in $\mathcal A_{r-1}(K)$. We then let $\mathcal A_r(K)$ consist of the profiles in $\mathcal A_{r-1}(K)$ that are not eliminated in round $r$.

The use of rounds is important. A profile may survive initially because another profile could rescue the comparison it creates. If that rescuing profile is eliminated in an earlier round, however, it is no longer available, and the original profile may be eliminated in a later round. Thus
\[
\mathcal A_0(K)\supseteq \mathcal A_1(K)\supseteq \mathcal A_2(K)\supseteq\cdots.
\]
Since there are only finitely many profiles with respect to a fixed core, this sequence eventually stabilizes. When
\[
\mathcal A_r(K)=\mathcal A_{r-1}(K),
\]
we call the profiles in this final collection the \emph{surviving profiles}.

\subsection{Computer verification of the finite profile calculations}\label{subsec:computer}

The admissibility and rescuer calculations used later in the proof are finite and follow directly from \cref{lem:attachment-rules}, but carrying them out for every core by hand would be lengthy and repetitive. We therefore use a deterministic computer program to verify these calculations. The complete finite certificates are recorded in Appendices~\ref{app:finite-certificates} and~\ref{app:obstruction-witnesses}, and the verification code is available at
\url{https://recoloring.owenh-math.com}.

For each core appearing in a certificate, the program reconstructs the core and enumerates every attachment profile compatible with the hypotheses already established in the written proof. It then applies \cref{lem:attachment-rules} to verify the claimed list of admissible profiles. For each elimination row, the program checks the displayed neighborhood comparison and determines all admissible profiles that can rescue it. It then verifies that every such rescuer is unavailable, either because it was eliminated in an earlier round or because the required extension contains one of the explicitly recorded obstruction configurations. Profiles belonging to the same elimination round are removed simultaneously. The program also verifies the stated induced copies of $B_{10}$ and $B_{11}$ used as obstruction witnesses.

The computer calculation is used only for this finite bookkeeping. The structural arguments producing the relevant cores and branch assumptions, as well as the recoloring arguments that use the resulting certificates, are proved in the text. To illustrate exactly what is being checked, we work through one representative profile calculation and one elimination chain by hand below.

\subsection{An explicit example of profiles, rescuers, and elimination rounds}\label{subsec:example}

We give an explicit example to illustrate the definitions from the previous subsection. Let $K_0$ be a core consisting of a $5$-cycle $C=c_1c_2c_3c_4c_5c_1$ and ordered off-cycle vertices $(f,a)$, where $f\in F_1$, $a\in Y_1$, and $fa\notin E(G)$. Thus $f$ is adjacent to every vertex of $C$ except $c_1$, while $N_C(a)=\{c_1,c_3,c_4\}$ (see the left side of \cref{fig:core-example}).

Consider the profile $\sigma=Y_3[11]$. A vertex $x$ has this profile precisely when $x\in Y_3$ and $x$ is adjacent to both $f$ and $a$. We claim that this profile is admissible. Indeed, after adjoining such a vertex $x$, we have $N_{K_0}(x)=\{c_1,c_3,c_5,f,a\}$. The graph induced by this set is triangle-free, so \cref{lem:attachment-rules}(i) holds. For condition~(ii), as $u$ ranges over $c_1,c_3,c_5,f,a$, the set $V(K_0)\setminus\bigl(N_{K_0}(x)\cup N_{K_0}(u)\bigr)$ is, respectively, $\{c_4\},\varnothing,\{c_2\},\varnothing,\{c_2\}$. Each is independent, and hence $Y_3[11]$ is admissible (see the right side of \cref{fig:core-example}).

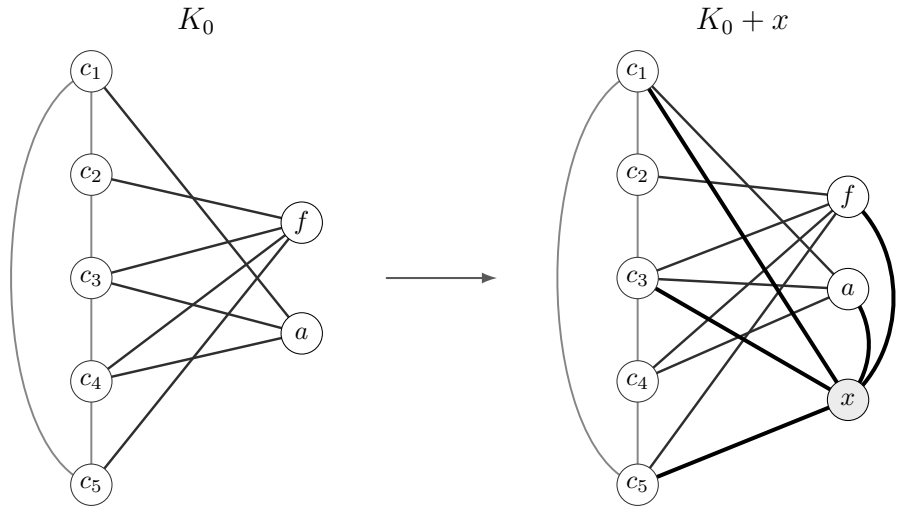
\begin{figure}[h!]
    \centering
    \begin{tikzpicture}[
    line cap=round,
    line join=round,
    cyclevertex/.style={
        circle,
        draw=black!80,
        fill=white,
        minimum size=5.6mm,
        inner sep=0pt,
        font=\small
    },
    offvertex/.style={
        circle,
        draw=black,
        fill=white,
        minimum size=5.6mm,
        inner sep=0pt,
        font=\small
    },
    newvertex/.style={
        circle,
        draw=black,
        fill=white,
        line width=1.1pt,
        minimum size=5.6mm,
        inner sep=0pt,
        font=\small
    },
    cycleedge/.style={
        draw=black!48,
        line width=0.75pt
    },
    attachmentedge/.style={
        draw=black!82,
        line width=0.95pt
    },
    xedge/.style={
        draw=black,
        line width=1.55pt
    }
]


\coordinate (c1L) at (0.00,  2.80);
\coordinate (c2L) at (0.00,  1.40);
\coordinate (c3L) at (0.00,  0.00);
\coordinate (c4L) at (0.00, -1.40);
\coordinate (c5L) at (0.00, -2.80);

\coordinate (fL) at (2.85,  0.75);
\coordinate (aL) at (2.85, -0.75);

\coordinate (c1R) at (7.40,  2.80);
\coordinate (c2R) at (7.40,  1.40);
\coordinate (c3R) at (7.40,  0.00);
\coordinate (c4R) at (7.40, -1.40);
\coordinate (c5R) at (7.40, -2.80);

\coordinate (fR) at (10.25,  1.10);
\coordinate (aR) at (10.25, -0.15);
\coordinate (xR) at (10.25, -1.65);


\draw[cycleedge] (c1L)--(c2L);
\draw[cycleedge] (c2L)--(c3L);
\draw[cycleedge] (c3L)--(c4L);
\draw[cycleedge] (c4L)--(c5L);

\draw[cycleedge]
    (c1L)
    .. controls (-1.45,2.30) and (-1.45,-2.30) ..
    (c5L);

\draw[attachmentedge] (fL)--(c2L);
\draw[attachmentedge] (fL)--(c3L);
\draw[attachmentedge] (fL)--(c4L);
\draw[attachmentedge] (fL)--(c5L);

\draw[attachmentedge] (aL)--(c1L);
\draw[attachmentedge] (aL)--(c3L);
\draw[attachmentedge] (aL)--(c4L);


\draw[cycleedge] (c1R)--(c2R);
\draw[cycleedge] (c2R)--(c3R);
\draw[cycleedge] (c3R)--(c4R);
\draw[cycleedge] (c4R)--(c5R);

\draw[cycleedge]
    (c1R)
    .. controls (5.95,2.30) and (5.95,-2.30) ..
    (c5R);

\draw[attachmentedge] (fR)--(c2R);
\draw[attachmentedge] (fR)--(c3R);
\draw[attachmentedge] (fR)--(c4R);
\draw[attachmentedge] (fR)--(c5R);

\draw[attachmentedge] (aR)--(c1R);
\draw[attachmentedge] (aR)--(c3R);
\draw[attachmentedge] (aR)--(c4R);


\draw[xedge] (xR)--(c1R);
\draw[xedge] (xR)--(c3R);
\draw[xedge] (xR)--(c5R);

\draw[xedge, bend right=50]
    (xR) to (fR);

\draw[xedge, bend right=40]
    (xR) to (aR);


\draw[-{Latex[length=2.1mm,width=1.5mm]}, draw=black!65, line width=0.8pt]
    (4,0) -- (5.5,0);


\node[cyclevertex] at (c1L) {$c_1$};
\node[cyclevertex] at (c2L) {$c_2$};
\node[cyclevertex] at (c3L) {$c_3$};
\node[cyclevertex] at (c4L) {$c_4$};
\node[cyclevertex] at (c5L) {$c_5$};

\node[offvertex] at (fL) {$f$};
\node[offvertex] at (aL) {$a$};

\node[cyclevertex] at (c1R) {$c_1$};
\node[cyclevertex] at (c2R) {$c_2$};
\node[cyclevertex] at (c3R) {$c_3$};
\node[cyclevertex] at (c4R) {$c_4$};
\node[cyclevertex] at (c5R) {$c_5$};

\node[offvertex] at (fR) {$f$};
\node[offvertex] at (aR) {$a$};
\node[offvertex, fill=lightgray!30!] at (xR) {$x$};


\node[font=\large] at (1.42,3.48) {$K_0$};
\node[font=\large] at (8.82,3.48) {$K_0+x$};

\end{tikzpicture}
    \caption{The core $K_0$ on the left, with off-cycle vertex order $(f,a)$, and the extension obtained by adjoining a vertex $x$ with admissible profile $Y_3[11]$.}
    \label{fig:core-example}
\end{figure}

In contrast, the profile $Z[10]$ is not admissible. A vertex $x$ with this profile is adjacent to $f$ and nonadjacent to $a$, while, since $x\in Z$, it has no neighbors on $C$. Then the two edges $xf$ and $ac_1$ induce a $2K_2$.

Next consider the profile $F_2[00]$. A direct application of \cref{lem:attachment-rules} shows that this profile is admissible. If $g$ is a vertex realizing this profile, then
\[N_{K_0}(g)=\{c_1,c_3,c_4,c_5\}\supseteq\{c_1,c_3,c_4\}=N_{K_0}(a).\]
Moreover, $ag\notin E(G)$, since the profile of $g$ is $F_2[00]$. Hence
\[a\preceq_{K_0+g}g.\]

Since $G$ is reduced, this comparison cannot persist in $G$. Thus there must be some vertex adjacent to $a$ and nonadjacent to $g$. A vertex $y$ with profile $Y_3[11]$ with respect to $K_0$ can do exactly this. By definition, $y$ is adjacent to $a$, and \cref{lem:attachment-rules}(a) shows that the nonedge $gy$ is permitted. Therefore $Y_3[11]$ rescues the comparison $a\preceq_{K_0+g}g$.

Let $K^*=K_0+\{g,y\}$, where $gy\notin E(G)$ and the ordered off-cycle vertices are now $(f,a,g,y)$. Applying \cref{lem:attachment-rules} exhaustively to each possible type and adjacency pattern gives exactly the following $23$ admissible profiles with respect to $K^*$:
\[\renewcommand{\arraystretch}{1.15}
\begin{array}{c@{\qquad}c@{\qquad}c@{\qquad}c@{\qquad}c}
Z[0000], & Z[1110], & Z[1111], & F_1[0001], & F_2[0000],\\
Y_1[0000], & Y_1[0001], & Y_2[0101], & Y_3[1000], & Y_3[1100],\\
Y_4[1011], & R_1[0111], & R_1[1110], & R_2[1000], & R_2[1110],\\
R_3[0011], & R_3[1111], & R_4[0000], & R_4[1110], & R_5[1000],\\
R_5[1001], & R_5[1011], & R_5[1110]. & &
\end{array}\]

Elimination rounds allow us to refine the initial list $\mathcal A_0(K^*)$. Consider the three admissible profiles
\[\rho=R_2[1110],\qquad \tau=R_5[1001],\qquad \sigma=F_2[0000].\]

Let $x$ be a vertex with profile $\rho$. Then
\[c_2\preceq_{K^*+x}x.\]
Checking the $23$ admissible profiles using \cref{lem:attachment-rules}(a)--(b) shows that none can rescue this comparison. Hence $\rho$ is eliminated in round~$1$, so $\rho\notin\mathcal A_1(K^*)$.

Next, let $x'$ have profile $\tau$. Then
\[x'\preceq_{K^*+x'}c_5.\]
The only admissible profile that can rescue this comparison is $\rho$. Since $\rho$ was eliminated in round~$1$, it is no longer available as a rescuer. Thus $\tau$ is eliminated in round~$2$.

Finally, let $x''$ have profile $\sigma$. Then
\[x''\preceq_{K^*+x''}g.\]
The only admissible profile that can rescue this comparison is $\tau$. Since $\tau$ was eliminated in round~$2$, the profile $\sigma$ is eliminated in round~$3$.

Thus the eliminations occur in the chain $\rho\longrightarrow\tau\longrightarrow\sigma$, where each profile is eliminated once its only possible rescuer has been removed in an earlier round. The finite certificates used later in the proof of \cref{thm:YES} record the same information for the larger cores that arise in the argument.

\section{Proofs of \cref{lem:B10} and \cref{lem:B11}}\label{sec:B10-B11}

We first define the two cores (and thus subgraphs) that appear in the statements of \cref{lem:B10} and \cref{lem:B11}. The two cores are illustrated in \cref{fig:B10B11}. In both cases,
$C:=c_1c_2c_3c_4c_5c_1$ denotes a distinguished induced $5$-cycle, and the vertices listed after $c_5$ are the off-cycle vertices in their prescribed order. 

\medskip
\noindent\textbf{The core $B_{10}$.}
The core $B_{10}$ has ordered vertex set $(c_1,c_2,c_3,c_4,c_5,f,a,g,y,r)$ where the off-cycle vertices are ordered as $(f,a,g,y,r)$ and they have types $(F_1,Y_1,F_2,Y_3,R_1)$, respectively. The only edges among the off-cycle vertices are $E\bigl(B_{10}[\{f,a,g,y,r\}]\bigr)=\{yf,ya,rf,ra,rg\}$.

\medskip
\noindent\textbf{The core $B_{11}$.}
The core $B_{11}$ has ordered vertex set $(c_1,c_2,c_3,c_4,c_5,f,a,g,x,b,q)$ where the off-cycle vertices are ordered as $(f,a,g,x,b,q)$ and they have types $(F_1,Y_1,F_2,Y_2,Y_5,R_2)$, respectively. The only edges among the off-cycle vertices are $E\bigl(B_{11}[\{f,a,g,x,b,q\}]\bigr)=\{xa,ba,bg,qf,qa,qg,qx\}$.

\begin{figure}[h!]
    \centering
    \begin{tikzpicture}[
    line cap=round,
    line join=round,
    vertex/.style={
        circle,
        draw=black,
        fill=white,
        minimum size=5.4mm,
        inner sep=0pt,
        font=\small
    },
    edge/.style={
        draw=black,
        line width=0.85pt
    }
]


\coordinate (c1L) at (0.0,  2.8);
\coordinate (c2L) at (0.0,  1.4);
\coordinate (c3L) at (0.0,  0.0);
\coordinate (c4L) at (0.0, -1.4);
\coordinate (c5L) at (0.0, -2.8);

\coordinate (fL)  at (4.2,  2.8);
\coordinate (aL)  at (4.2,  1.4);
\coordinate (gL)  at (4.2,  0.0);
\coordinate (yL)  at (4.2, -1.4);
\coordinate (rL)  at (4.2, -2.8);

\draw[edge] (c1L)--(c2L);
\draw[edge] (c2L)--(c3L);
\draw[edge] (c3L)--(c4L);
\draw[edge] (c4L)--(c5L);
\draw[edge] (c1L) to[bend right=45] (c5L);

\draw[edge] (fL)--(c2L);
\draw[edge] (fL)--(c3L);
\draw[edge] (fL)--(c4L);
\draw[edge] (fL)--(c5L);

\draw[edge] (aL)--(c1L);
\draw[edge] (aL)--(c3L);
\draw[edge] (aL)--(c4L);

\draw[edge] (gL)--(c1L);
\draw[edge] (gL)--(c3L);
\draw[edge] (gL)--(c4L);
\draw[edge] (gL)--(c5L);

\draw[edge] (yL)--(c1L);
\draw[edge] (yL)--(c3L);
\draw[edge] (yL)--(c5L);

\draw[edge] (rL)--(c2L);
\draw[edge] (rL)--(c5L);

\draw[edge] (fL) to[bend left=40] (yL);
\draw[edge] (aL) to[bend left=30] (yL);
\draw[edge] (fL) to[bend left=60] (rL);
\draw[edge] (aL) to[bend left=40] (rL);
\draw[edge] (gL) to[bend left=30] (rL);

\node[vertex] at (c1L) {$c_1$};
\node[vertex] at (c2L) {$c_2$};
\node[vertex] at (c3L) {$c_3$};
\node[vertex] at (c4L) {$c_4$};
\node[vertex] at (c5L) {$c_5$};

\node[vertex] at (fL) {$f$};
\node[vertex] at (aL) {$a$};
\node[vertex] at (gL) {$g$};
\node[vertex] at (yL) {$y$};
\node[vertex] at (rL) {$r$};

\node[font=\large] at (2.1, 3.7) {$B_{10}$};


\coordinate (c1R) at (9.4,  2.8);
\coordinate (c2R) at (9.4,  1.4);
\coordinate (c3R) at (9.4,  0.0);
\coordinate (c4R) at (9.4, -1.4);
\coordinate (c5R) at (9.4, -2.8);

\coordinate (fR)  at (13.6,  3.5);
\coordinate (aR)  at (13.6,  2.1);
\coordinate (gR)  at (13.6,  0.7);
\coordinate (xR)  at (13.6, -0.7);
\coordinate (bR)  at (13.6, -2.1);
\coordinate (qR)  at (13.6, -3.5);

\draw[edge] (c1R)--(c2R);
\draw[edge] (c2R)--(c3R);
\draw[edge] (c3R)--(c4R);
\draw[edge] (c4R)--(c5R);
\draw[edge] (c1R) to[bend right=45] (c5R);

\draw[edge] (fR)--(c2R);
\draw[edge] (fR)--(c3R);
\draw[edge] (fR)--(c4R);
\draw[edge] (fR)--(c5R);

\draw[edge] (aR)--(c1R);
\draw[edge] (aR)--(c3R);
\draw[edge] (aR)--(c4R);

\draw[edge] (gR)--(c1R);
\draw[edge] (gR)--(c3R);
\draw[edge] (gR)--(c4R);
\draw[edge] (gR)--(c5R);

\draw[edge] (xR)--(c2R);
\draw[edge] (xR)--(c4R);
\draw[edge] (xR)--(c5R);

\draw[edge] (bR)--(c2R);
\draw[edge] (bR)--(c3R);
\draw[edge] (bR)--(c5R);

\draw[edge] (qR)--(c1R);
\draw[edge] (qR)--(c3R);

\draw[edge] (aR) to[bend left=30] (xR);
\draw[edge] (aR) to[bend left=40] (bR);
\draw[edge] (gR) to[bend left=30] (bR);
\draw[edge] (xR) to[bend left=40] (qR);
\draw[edge] (fR) to[bend left=60] (qR);
\draw[edge] (aR) to[bend left=50] (qR);
\draw[edge] (gR) to[bend left=40] (qR);

\node[vertex] at (c1R) {$c_1$};
\node[vertex] at (c2R) {$c_2$};
\node[vertex] at (c3R) {$c_3$};
\node[vertex] at (c4R) {$c_4$};
\node[vertex] at (c5R) {$c_5$};

\node[vertex] at (fR) {$f$};
\node[vertex] at (aR) {$a$};
\node[vertex] at (gR) {$g$};
\node[vertex] at (xR) {$x$};
\node[vertex] at (bR) {$b$};
\node[vertex] at (qR) {$q$};

\node[font=\large] at (11.5, 4.4) {$B_{11}$};

\end{tikzpicture}
    \caption{The cores $B_{10}$ (left) and $B_{11}$ (right) with off-cycle vertex orders $(f,a,g,y,r)$ and $(f,a,g,x,b,q)$ respectively }
    \label{fig:B10B11}
\end{figure}
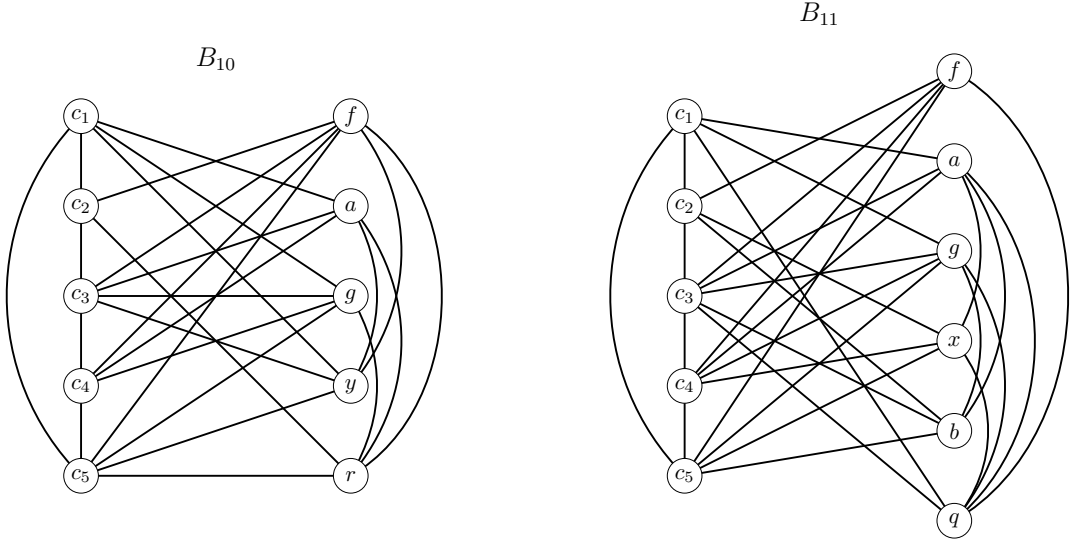

\begin{keylemma*}[\ref{lem:B10}]
If a reduced $(2K_2,K_4)$-free graph $G$ contains an induced copy of $B_{10}$, then $\mathcal R_5(G)$ is connected.
\end{keylemma*}

\begin{proof}
Let $B_{10}$ be an induced subgraph of a reduced $(2K_2,K_4)$-free graph $G$ with $C=c_1c_2c_3c_4c_5c_1$ inducing a $5$-cycle and off-cycle vertices ordered $(f,a,g,y,r)$. If $G$ is 3-colorable then we are done by \cref{thm:reductions}\ref{it:3colorable}, so we can assume $G$ is $4$-chromatic. We first determine the admissible profiles with respect to the core $B_{10}$.

\begin{claim}\label{clm:B10-admissible}
There are exactly $21$ admissible profiles with respect to $B_{10}$. They are listed below, where the bit strings are read in the order $(f,a,g,y,r)$.
\[
\begin{array}{@{}cl@{\qquad}|@{\qquad}cl@{}}
\toprule
\text{cycle type} & \text{bit strings} & \text{cycle type} & \text{bit strings}\\
\midrule
Z   & [00000],[11100]       & F_1 & [00011]\\
F_2 & [00001]               & Y_1 & [00001],[00011]\\
Y_2 & [01011]               & Y_3 & [10000],[11000]\\
Y_4 & [10110]               & R_1 & [01110],[11100]\\
R_2 & [10001],[11101]       & R_3 & [00110],[00111],[11110]\\
R_4 & [00000],[11100]       & R_5 & [10011],[10111]\\
\bottomrule
\end{array}
\]
\end{claim}

The verification of \cref{clm:B10-admissible} is finite and is carried out by the computer verification described in \cref{subsec:computer}. The program enumerates all attachment profiles with respect to $B_{10}$ and applies \cref{lem:attachment-rules}(i)--(ii) to each of them. The resulting list is exactly the $21$ profiles displayed above.

We next use reducedness to show that $15$ of these $21$ profiles cannot be realized in $G$. In each row below, $v$ denotes a hypothetical vertex realizing the displayed profile. Every comparison is taken in the induced graph $B_{10}+v$; to keep the table readable, we suppress the subscript and simply write $\preceq$. The final column gives the complete set of admissible profiles with respect to $B_{10}$ that can rescue the displayed comparison.

The local comparisons and the completeness of every rescuer list in the table are verified by the computer program using \cref{lem:attachment-rules}(a)--(b). Profiles in the same round are tested against the profiles remaining at the beginning of that round and are eliminated simultaneously. Thus a profile listed in round~$1$ has no possible rescuer, while for each profile in a later round every possible rescuer was eliminated in an earlier round.

\begin{center}
\small
\setlength{\tabcolsep}{4.5pt}
\renewcommand{\arraystretch}{0.96}
\begin{tabular}{@{}c l l l@{}}
\toprule
round & profile & comparison & possible rescuers\\
\midrule
1 & $Z[00000]$    & $v\preceq c_1$ & none\\
1 & $Z[11100]$    & $v\preceq c_4$ & none\\
1 & $F_1[00011]$  & $v\preceq f$   & none\\
1 & $Y_3[10000]$  & $v\preceq y$   & none\\
1 & $R_1[01110]$  & $v\preceq c_1$ & none\\
1 & $R_2[10001]$  & $v\preceq c_2$ & none\\
1 & $R_2[11101]$  & $c_2\preceq v$ & none\\
1 & $R_3[11110]$  & $v\preceq c_3$ & none\\
1 & $R_4[00000]$  & $v\preceq f$   & none\\
1 & $R_4[11100]$  & $v\preceq c_4$ & none\\
1 & $R_5[10011]$  & $v\preceq c_5$ & none\\
\midrule
2 & $F_2[00001]$  & $v\preceq g$   & $R_5[10011]$\\
2 & $Y_1[00011]$  & $v\preceq a$   & $Y_3[10000]$\\
\midrule
3 & $Y_3[11000]$  & $y\preceq v$   & $Y_1[00011]$\\
3 & $R_5[10111]$  & $c_5\preceq v$ & $F_2[00001]$\\
\bottomrule
\end{tabular}
\end{center}

The six profiles that remain after these rounds are
\[
\begin{aligned}
&\sigma_1=R_1[11100],\quad \sigma_2=R_3[00110],\quad \sigma_3=Y_4[10110],\\
&\sigma_4=Y_1[00001],\quad \sigma_5=Y_2[01011],\quad \sigma_6=R_3[00111].
\end{aligned}
\]

For each pair of these profiles, the computer verification applies the two-vertex tests in \cref{lem:attachment-rules}(a)--(b) and gives the following relation table. A $0$ means that the two corresponding profile classes are necessarily anticomplete, a $1$ means that they are necessarily complete, $*$ means that both an edge and a nonedge are permitted, and $\times$ means that the two profiles cannot be realized simultaneously. 

\[
\setlength{\arraycolsep}{5pt}
\renewcommand{\arraystretch}{0.95}
\begin{array}{c|cccccc}
 & \sigma_1 & \sigma_2 & \sigma_3 & \sigma_4 & \sigma_5 & \sigma_6\\
\hline
\sigma_1 & 0 & * & 0 & 1 & 1 & *\\
\sigma_2 &   & 0 & 0 & \times & \times & 0\\
\sigma_3 &   &   & 0 & 1 & 1 & \times\\
\sigma_4 &   &   &   & 0 & 1 & 0\\
\sigma_5 &   &   &   &   & 0 & \times\\
\sigma_6 &   &   &   &   &   & 0
\end{array}
\]

We now show that $\sigma_1$, $\sigma_2$, and $\sigma_4$ cannot be realized in $G$.

Suppose that $v$ has profile $\sigma_1=R_1[11100]$. By definition of the profile, $v$ and $r$ have the same neighborhood in $B_{10}$. Comparing the last bit of each surviving profile with the $\sigma_1$-row of the relation table shows that $v$ and $r$ also have the same adjacency to every vertex with profile $\sigma_1$, $\sigma_3$, $\sigma_4$, or $\sigma_5$. Therefore a rescuer cannot have such profiles. Thus they can differ only on vertices with profile $\sigma_2$ or $\sigma_6$.

The vertices with profiles $\sigma_2$ or $\sigma_6$ form an independent set by the relation table. The neighborhoods of $v$ and $r$ within this set must therefore be comparable by inclusion. Indeed, otherwise there are vertices $x$ and $y$ with profiles in $\{\sigma_2,\sigma_6\}$ such that $vx,ry\in E(G)$ and $vy,rx\notin E(G)$. Since $vr,xy\notin E(G)$, the edges $vx$ and $ry$ induce a $2K_2$, a contradiction. Since $v$ and $r$ have the same adjacency to every other vertex, it follows that either $N_G(v)\subseteq N_G(r)$ or $N_G(r)\subseteq N_G(v)$, contradicting reducedness. Hence no vertex outside $B_{10}$ realizes $\sigma_1$.

Next suppose that $p$ has profile $\sigma_2=R_3[00110]$. Since no vertex outside $B_{10}$ realizes $\sigma_1$, the relation table shows that $p$ has no neighbor outside $B_{10}$: it is anticomplete to vertices with profiles $\sigma_2$, $\sigma_3$, and $\sigma_6$, while $\sigma_4$ and $\sigma_5$ cannot coexist with $\sigma_2$. On the core,
\[N_{B_{10}}(p)=\{c_2,c_4,g,y\}\subseteq N_{B_{10}}(c_3).\]
Thus $p\preceq_G c_3$, contradicting reducedness. Hence no vertex outside $B_{10}$ realizes $\sigma_2$.

Now suppose that $b$ has profile $\sigma_4=Y_1[00001]$. On the core,
\[N_{B_{10}}(b)=\{c_1,c_3,c_4,r\}\subseteq N_{B_{10}}(a)\cap N_{B_{10}}(g),\]
so
\[b\preceq_{B_{10}+b}a\qquad\text{and}\qquad b\preceq_{B_{10}+b}g.\]
Since $G$ is reduced, neither of these neighborhood containments can persist in $G$. Therefore there exists a vertex $z\in N_G(b)\setminus N_G(a)$ and a vertex $x\in N_G(b)\setminus N_G(g)$.

Since neither $\sigma_1$ nor $\sigma_2$ is realized in $G$, the relation table shows that the only profiles whose vertices can be adjacent to $b$ are $\sigma_3$ and $\sigma_5$. A vertex with profile $\sigma_3=Y_4[10110]$ is nonadjacent to $a$, whereas a vertex with profile $\sigma_5=Y_2[01011]$ is adjacent to $a$. Hence $z$ has profile $\sigma_3$. Similarly, a vertex with profile $\sigma_3$ is adjacent to $g$, whereas a vertex with profile $\sigma_5$ is nonadjacent to $g$. Hence $x$ has profile $\sigma_5$.

The relation table gives $zx\in E(G)$. Moreover, $b$, $z$, and $x$ have cycle types $Y_1$, $Y_4$, and $Y_2$, respectively, so all three are adjacent to $c_4$. Hence $\{b,z,x,c_4\}$ induces a $K_4$, a contradiction. Therefore no vertex outside $B_{10}$ realizes $\sigma_4$.

Thus only
\[\sigma_3=Y_4[10110],\qquad \sigma_5=Y_2[01011],\qquad \sigma_6=R_3[00111]\]
can occur outside $B_{10}$. By the relation table, vertices with profiles $\sigma_3$ and $\sigma_5$ are complete to one another, while $\sigma_6$ cannot coexist with either $\sigma_3$ or $\sigma_5$.

Each of these three profiles is realized by at most one vertex. Indeed, suppose first that a vertex with profile $\sigma_6$ exists. Then no vertex with profile $\sigma_3$ or $\sigma_5$ exists. Two vertices with profile $\sigma_6$ would have the same neighborhood in $B_{10}$ and would be nonadjacent to one another, so they would have identical neighborhoods in $G$, contradicting reducedness.

We may therefore suppose that no vertex has profile $\sigma_6$. If two vertices have profile $\sigma_3$, then they have the same neighborhood in $B_{10}$, are nonadjacent to one another, and are both adjacent to every vertex with profile $\sigma_5$. Hence they have identical neighborhoods in $G$, again contradicting reducedness. Thus $\sigma_3$ is realized by at most one vertex. The same argument shows that $\sigma_5$ is realized by at most one vertex.

Consequently, $G$ is one of
\[B_{10},\qquad B_{10}+z,\qquad B_{10}+x,\qquad B_{10}+t,\qquad B_{10}+\{z,x\},\]
where $z$, $x$, and $t$ have profiles $\sigma_3$, $\sigma_5$, and $\sigma_6$, respectively, and $zx\in E(G)$ in the final case. Their neighborhoods in the core are
\[
\begin{aligned}
N_{B_{10}}(z)&=\{c_1,c_2,c_4,f,g,y\},\\
N_{B_{10}}(x)&=\{c_2,c_4,c_5,a,y,r\},\\
N_{B_{10}}(t)&=\{c_2,c_4,g,y,r\}.
\end{aligned}
\]

Let $\varphi$ be a $5$-coloring of $G$. We would like to prove that $\mathcal R_5(G)$ is connected. We show that, for each of the five possible graphs, we may apply \cref{cor:mainrecolor} directly with $\psi=\varphi$.

For $G=B_{10}$, we have the following partition of $V(G)$ into four independent sets:
\[
I_1=\{c_5\},\qquad I_2=\{c_1,c_3,r\},\qquad I_3=\{f,a,g\},\qquad I_4=\{c_2,c_4,y\}.
\]
In particular, $G-c_5$ is $3$-colorable. Let $S$ be the color class of $\varphi$ containing $c_5$. Then $S$ is independent and $G-S$ is an induced subgraph of $G-c_5$, so $\chi(G-S)\leq 3$. Therefore, \cref{cor:mainrecolor} applies with $I=S$, then $\mathcal R_5(G)$ is connected.

Next consider the three $11$-vertex graphs. For each graph, we choose a vertex $v$ and consider the color class $S$ of $\varphi$ containing $v$. Since $S$ is independent, every vertex of $S\setminus\{v\}$ is a nonneighbor of $v$. We extend $S$ to a maximal independent set $I$ containing $v$. Thus, it is enough to determine all maximal independent sets containing $v$ and verify that $\chi(G-I)\leq 3$ for each such set $I$.

For each of the three $11$-vertex graphs, the following table gives the chosen vertex $v$, every maximal independent set $I$ containing $v$, and a $3$-coloring of $G-I$.

\begin{center}
\small
\setlength{\tabcolsep}{3.5pt}
\renewcommand{\arraystretch}{0.98}
\begin{tabular}{@{}l l l@{}}
\toprule
graph; vertex & maximal independent set $I$ & $3$-coloring of $G-I$\\
\midrule
$B_{10}+z$; $f$
& $\{c_1,f\}$
& $\{c_3,c_5,z\};\ \{c_4,y,r\};\ \{c_2,a,g\}$\\
&
$\{f,a,g\}$
& $\{c_1,c_3,r\};\ \{c_2,c_4,y\};\ \{c_5,z\}$\\
\addlinespace[1pt]
$B_{10}+x$; $c_5$
& $\{c_3,c_5\}$
& $\{c_2,a,g\};\ \{c_4,y,r\};\ \{c_1,f,x\}$\\
&
$\{c_2,c_5,a\}$
& $\{c_1,c_3,r\};\ \{c_4,y\};\ \{f,g,x\}$\\
\addlinespace[1pt]
$B_{10}+t$; $c_3$
& $\{c_1,c_3,r\}$
& $\{c_2,c_4,y\};\ \{c_5,a,t\};\ \{f,g\}$\\
&
$\{c_1,c_3,t\}$
& $\{c_2,c_5,a\};\ \{f,g\};\ \{c_4,y,r\}$\\
&
$\{c_3,c_5,t\}$
& $\{c_1,f\};\ \{c_2,a,g\};\ \{c_4,y,r\}$\\
\bottomrule
\end{tabular}
\end{center}

To see that these lists are complete, the subgraphs induced by the nonneighbors of the selected vertices $f$, $c_5$, and $c_3$ are, respectively, the paths $ac_1g$, $c_2c_3a$, and $c_1c_5rt$.

The maximal independent sets containing the selected vertex are therefore exactly those listed in the table. Since the color class $S$ containing the selected vertex is contained in one of these maximal independent sets $I$, and the table shows that $\chi(G-I)\leq 3$, \cref{cor:mainrecolor} applies in each case. Thus the reconfiguration graph on $5$ colors of each of the three $11$-vertex graphs is connected.

It remains to consider $G=B_{10}+\{z,x\}$. The graph has the following two partitions into four independent sets:
\[
\begin{aligned}
\mathcal P:\;&\{c_2,a,g\},\ \{c_4,y,r\},\ \{c_3,c_5,z\},\ \{c_1,f,x\},\\
\mathcal Q:\;&\{c_2,c_4,y\},\ \{c_1,c_3,r\},\ \{c_5,a,z\},\ \{f,g,x\}.
\end{aligned}
\]

Consider any $5$-coloring of $G$. Among the six vertices $\{c_4,y,c_5,z,f,x\}$ two receive the same color. Since the coloring is proper, this monochromatic pair must be a nonedge. The only nonedges among these six vertices are
\[c_4y,\qquad c_5z,\qquad fx.\]

The common nonneighbors outside these pairs are, respectively,
\[\{c_2,r\},\qquad \{c_3,a\},\qquad \{c_1,g\},\]
and each displayed pair is an edge. Hence the entire color class containing the monochromatic pair is contained in one of the six independent triples
\[\{c_2,c_4,y\},\quad \{c_4,y,r\},\quad \{c_3,c_5,z\},\quad \{c_5,a,z\},\quad \{c_1,f,x\},\quad \{f,g,x\}.\]
Each of these triples is a class of one of the partitions $\mathcal P$ or $\mathcal Q$, and its complement is $3$-colorable by the other three classes of that partition. Thus every $5$-coloring has a color class contained in an independent set with $3$-colorable complement. By \cref{cor:mainrecolor}, then $\mathcal R_5(G)$ is connected.

This proves \cref{lem:B10}.
\end{proof}

\begin{keylemma*}[\ref{lem:B11}]
If a reduced $(2K_2,K_4,B_{10})$-free graph $G$ contains an induced copy of $B_{11}$, then $\mathcal R_5(G)$ is connected.
\end{keylemma*}

\begin{proof}
Let $B_{11}$ be an induced subgraph of a reduced $(2K_2,K_4)$-free graph $G$ with $C=c_1c_2c_3c_4c_5c_1$ inducing a $5$-cycle and off-cycle vertices ordered $(f,a,g,x,b,q)$. We first determine the admissible profiles with respect to the core $B_{11}$.

\begin{claim}\label{clm:B11-admissible}
There are exactly $16$ admissible profiles with respect to $B_{11}$. They are listed below, where the bit strings are read in the order $(f,a,g,x,b,q)$.
\[
\begin{array}{@{}cl@{\qquad}|@{\qquad}cl@{}}
\toprule
\text{cycle type} & \text{bit strings} & \text{cycle type} & \text{bit strings}\\
\midrule
U   & [000001]                 & Z   & [000000]\\
F_1 & [000001]                 & F_2 & [000011]\\
Y_1 & [000111]                 & Y_2 & [010001],[010011]\\
Y_5 & [011000],[011100]        & R_1 & [011001],[111001]\\
R_2 & [100110],[111100]        & R_3 & [111011]\\
R_4 & [111100]                 & R_5 & [101110]\\
\bottomrule
\end{array}
\]
\end{claim}

The verification of \cref{clm:B11-admissible} is finite and is carried out by the computer verification described in \cref{subsec:computer}. The program enumerates all attachment profiles with respect to $B_{11}$ and applies \cref{lem:attachment-rules}(i)--(ii) to each of them. The resulting list is exactly the $16$ profiles displayed above.

Four of these profiles cannot be realized in $G$ because adjoining a vertex with any one of them produces an induced copy of $B_{10}$:
\[
Y_5[011100],\qquad U[000001],\qquad R_1[111001],\qquad Y_2[010011].
\]
The computer verification checks each of these four exclusions by exhibiting and verifying an induced copy of $B_{10}$ in the corresponding extension of $B_{11}$.

It remains to show that none of the other twelve profiles can be realized in $G$. In each row below, $v$ denotes a hypothetical vertex realizing the displayed profile. Every comparison is taken in the induced graph $B_{11}+v$; to keep the table readable, we suppress the subscript and simply write $\preceq$. The final column gives the complete set of profiles, among the twelve still under consideration, that can rescue the displayed comparison.

\begin{center}
\small
\setlength{\tabcolsep}{4.5pt}
\renewcommand{\arraystretch}{0.96}
\begin{tabular}{@{}l l l@{}}
\toprule
profile & comparison & possible rescuers\\
\midrule
$Z[000000]$    & $v\preceq c_1$ & none\\
$Y_5[011000]$  & $v\preceq b$   & $Y_2[010001]$ via an edge\\
$R_2[111100]$  & $v\preceq q$   & none\\
$R_4[111100]$  & $v\preceq c_4$ & none\\
$R_2[100110]$  & $v\preceq c_2$ & none\\
$R_5[101110]$  & $v\preceq c_5$ & none\\
$F_1[000001]$  & $v\preceq f$   & $R_1[011001]$ via an edge\\
$Y_2[010001]$  & $v\preceq x$   & $Y_5[011000]$ via an edge\\
$R_1[011001]$  & $v\preceq c_1$ & $F_1[000001]$ via an edge\\
$F_2[000011]$  & $v\preceq g$   & none\\
$R_3[111011]$  & $v\preceq c_3$ & none\\
$Y_1[000111]$  & $v\preceq a$   & none\\
\bottomrule
\end{tabular}
\end{center}

The local comparisons and the completeness of every rescuer list in the table are verified by the computer program using \cref{lem:attachment-rules}(a)--(b). Thus eight of the twelve profiles have no possible rescuer. For the remaining four profiles, the only possible rescue occurs within one of the pairs
\[
\{Y_5[011000],Y_2[010001]\}
\qquad\text{or}\qquad
\{F_1[000001],R_1[011001]\},
\]
and in either case the rescue requires the two realizing vertices to be adjacent. The computer verification checks that each such edge produces an induced copy of $B_{10}$.

Since $G$ is $B_{10}$-free, neither of these rescue edges can occur. Hence none of the twelve displayed comparisons has a rescuer in $G$. Since $G$ is reduced, none of these twelve profiles can be realized outside $B_{11}$. Together with the four profiles excluded above, this shows that no admissible profile is realized outside the core. Therefore $G\cong B_{11}$.

It remains to show that if $G = B_{11}$, then $\mathcal R_5(G)$ is connected. To see this, the vertices nonadjacent to $c_3$ are $c_1,c_5,x$, and they induce the path $c_1c_5x$. Hence the maximal independent sets containing $c_3$ are exactly
\[
\{c_3,c_5\}\qquad\text{and}\qquad \{c_1,c_3,x\}.
\]
For each such independent set $I$, the graph $B_{11}-I$ is $3$-colorable, as shown below.

\begin{center}
\small
\setlength{\tabcolsep}{4pt}
\renewcommand{\arraystretch}{0.98}
\begin{tabular}{@{}c c@{}}
\toprule
maximal independent set $I$ & $3$-coloring of $B_{11}-I$\\
\midrule
$\{c_3,c_5\}$
& $\{c_2,c_4,q\};\ \{f,a,g\};\ \{c_1,x,b\}$\\
$\{c_1,c_3,x\}$
& $\{c_4,b,q\};\ \{c_2,c_5,a\};\ \{f,g\}$\\
\bottomrule
\end{tabular}
\end{center}

If $B_{11}$ is $3$-colorable, then \cref{thm:reductions}\ref{it:3colorable} applies. Otherwise $B_{11}$ is $4$-chromatic. Consider any $5$-coloring of $B_{11}$ and extend the color class containing $c_3$ to a maximal independent set $I$. By the preceding list, $\chi(B_{11}-I)\leq 3$. Hence \cref{cor:mainrecolor} applies, we have that $\mathcal R_5(G)$ is connected.

This proves \cref{lem:B11}.
\end{proof}

\section{Proof of \cref{thm:YES}}\label{sec:main-proof}

\begin{proof}[Proof of \cref{thm:YES}]
Suppose for a contradiction that \cref{thm:YES} is false. If $\mathcal R_5(H)$ were connected for every $4$-chromatic $(2K_2,K_4)$-free graph $H$, then \cref{thm:reductions}\ref{it:3colorable}, the $4$-colorability of $(2K_2,K_4)$-free graphs, and \cref{lem:hereditary} would imply \cref{thm:YES}. Thus there exists a $(2K_2,K_4)$-free graph with disconnected $\mathcal R_5$. Let $G$ be such a graph with the minimum number of vertices. Let $\mathcal H$ be the hereditary class of all proper induced subgraphs of $G$. For every $H\in\mathcal H$, the graph $\mathcal R_{\chi(H)+1}(H)$ is connected: this follows from \cref{thm:reductions}\ref{it:3colorable} when $\chi(H)\leq 3$, and from the choice of $G$ when $\chi(H)=4$. Hence \cref{lem:hereditary} implies that every proper induced subgraph of $G$ is recolorable. Moreover, $G$ is connected, since otherwise the $5$-coloring reconfiguration graph of each component would be connected. We therefore have the following properties.

\begin{itemize}
    \item $G$ is reduced. To see this, suppose that $G$ is not reduced. Then there exist distinct vertices $u,v\in V(G)$ with $u\preceq_G v$. Since $G-u$ is a proper induced subgraph of $G$, it is recolorable. Therefore, \cref{lem:reduced} implies that $G$ is recolorable, contradicting the fact that $\mathcal R_5(G)$ is disconnected. Thus $G$ is reduced.
    
    \item $\chi(G)=4$ by \cref{thm:reductions}\ref{it:3colorable} and 4-colorability of such graphs \cite{GH19}
    
    \item $G$ contains an induced $C_5$ by \cref{thm:reductions}\ref{it:C5recolor}.
    \item $G$ contains an induced $H_4$ by \cref{lem:H4}.
    \item $G$ is $B_{10}$-free by \cref{lem:B10}, since $G$ is reduced.
    \item $G$ is $B_{11}$-free by \cref{lem:B11}, since $G$ is reduced.
\end{itemize}

We start with a brief proof outline. Under these assumptions, we split the proof into cases depending on the type of core $K$. Then, as in the proofs of \cref{lem:B10} and \cref{lem:B11}, we list all admissible profiles with respect to $K$ and eliminate them in rounds using the assumptions of the current case. We support these eliminations using the tables in Appendices~\ref{app:finite-certificates} and~\ref{app:obstruction-witnesses}. In particular, in some rounds we eliminate profiles that would induce either a $B_{10}$ or a $B_{11}$; the corresponding obstruction witnesses are given in Appendix~\ref{app:obstruction-witnesses}. After this process, only finitely many graphs remain to check for recolorability. A contradiction then arises either from finding two comparable vertices, or from finding a vertex whose deletion leaves a $3$-colorable graph, in which case we apply Corollary~\ref{cor:mainrecolor} with the independent set consisting of that vertex.

We reiterate that Appendix~\ref{app:finite-certificates} is divided into sections according to the core. Each certificate lists the admissible profiles and records the elimination rounds or direct exclusions used in that case.

Let us proceed with the proof. For any induced $5$-cycle $C$ in $G$, we define the corresponding sets $U$, $R_i$, $Y_i$, $F_i$, and $Z$. Since $G$ contains an induced $H_4$, there is some choice of $C$ for which
\[
F=\bigcup_{i=1}^5 F_i\neq\varnothing.
\]
Among all pairs $(C,i)$ such that $C$ is an induced $5$-cycle and $F_i\neq\varnothing$, choose one that first minimizes $|U|$ and, subject to that, minimizes $|F_i|$. Up to symmetry, we may assume that $F_1\neq\varnothing$, and fix $f\in F_1$. We now prove a series of claims regarding the structure of $G$.

\begin{claim}\label{claim:min_rim}
For our choice of $C$ and $F_1$, the following hold.
\begin{enumerate}[label=\textup{(\roman*)},leftmargin=2.6em]
\item $U$ is complete to $R_i$ for every $i\in[5]$, and $|U|\leq 1$.
\item $F_3=F_4=\varnothing$ and $F_1=\{f\}$.
\item If $U=\{u\}$, then $u$ is complete to $Z$, $Y_1=\varnothing$, and $Y_3\cup Y_4\neq\varnothing$.
\end{enumerate}
\end{claim}

\begin{proofc}
We first show (i). Let $u\in U$ and suppose there exists some $r\in R_i\setminus N(u)$. Let $C^{(r)}$ be the $5$-cycle obtained from $C$ by replacing $c_i$ with $r$. To avoid confusion, we let $U^{(r)}, R^{(r)}_i, Y^{(r)}_i, F^{(r)}_i$, and $Z^{(r)}$ denote the sets relative to $C^{(r)}$, where the indexing agrees with that of $C$ and $r$ is the $i$th vertex of $C^{(r)}$. Then $u\in F^{(r)}_i$. Moreover, every vertex of $U^{(r)}$ belongs to $U\cup F_i$ with respect to $C$, since it is adjacent to every vertex of $C\setminus\{c_i\}$. Since $u\notin U^{(r)}$, the minimality of our choice of $(C,i)$ gives $|U^{(r)}|\geq |U|$, and hence there exists some $g\in F_i\cap U^{(r)}$. In particular, $gr\in E(G)$. But then the edges $gr$ and $uc_i$ induce a $2K_2$: we have $ur\notin E(G)$ by assumption, $gc_i\notin E(G)$ since $g\in F_i$, and $ug\notin E(G)$ by \cref{lem:basic}\ref{it:UF}. This is a contradiction. Thus, $U$ is complete to $R_i$ for every $i\in[5]$.

It remains to show that $|U|\leq 1$. By definition, every vertex of $U$ is complete to $C$, while by \cref{lem:basic}\ref{it:UF} it is anticomplete to every $F_i\cup Y_i$, and by the previous paragraph it is complete to every $R_i$. Thus, any two vertices of $U$ can differ only in their adjacency to $Z$. Since both $U$ and $Z$ are independent by \cref{lem:basic}\ref{it:indep}, the $2K_2$-freeness of $G$ implies that the neighborhoods in $Z$ of the vertices of $U$ are linearly ordered by inclusion. Hence, if $u,u'\in U$ are distinct, then either $N(u)\subseteq N(u')$ or $N(u')\subseteq N(u)$. Since $u$ and $u'$ are nonadjacent, they are comparable, contradicting that $G$ is reduced. Therefore, $|U|\leq 1$.

Next we show (ii). By \cref{lem:basic}\ref{it:FF}, we have $F_3=F_4=\varnothing$ since $f\in F_1$. Now suppose that $f'\in F_1\setminus\{f\}$. We first show that $f$ and $f'$ have the same neighbors in $R_3$. Indeed, suppose that $r\in R_3$ with $fr\in E(G)$ and $f'r\notin E(G)$. Let $C^{(r)}$ be the $5$-cycle obtained from $C$ by replacing $c_3$ with $r$, and again we use $U^{(r)}, R^{(r)}_i, Y^{(r)}_i, F^{(r)}_i$, and $Z^{(r)}$ to denote the sets relative to $C^{(r)}$, where the indexing agrees with that of $C$ and $r$ is the third vertex of $C^{(r)}$. Then, by (i) and since $F_3=\varnothing$, we have $U^{(r)}=U$. Since $f\in F_1^{(r)}$ and $f'\in Y_2^{(r)}$, there must exist some $y\in Y_2$ such that $y\in F_1^{(r)}$; otherwise $|F_1^{(r)}|<|F_1|$, contradicting the minimality of our choice of $(C,1)$. But then $yr\in E(G)$, and the edges $yr$ and $f'c_3$ induce a $2K_2$ by \cref{lem:basic}\ref{it:FY}. Thus, $f$ and $f'$ have the same neighbors in $R_3$, and the symmetric argument shows that they have the same neighbors in $R_4$ as well. Consequently, they can differ only in their adjacency to $R_1\cup Z$. Now, we show that $f$ and $f'$ also have comparable neighborhoods in $R_1 \cup Z$. We know that $R_1$ and $Z$ are independent by \cref{lem:basic}\ref{it:indep} and $Z$ is anticomplete to $R_1$ by \cref{lem:basic}\ref{it:ZR}. Suppose that there exist $x \in N(f) \setminus N(f')$ and $x' \in N(f') \setminus N(f)$. Then $x, x' \in R_1 \cup Z$, so $xx' \notin E(G)$. Moreover, $ff' \notin E(G)$ by \cref{lem:basic}\ref{it:indep}, since $F_1$ is independent. Hence the edges $fx$ and $f'x'$ induce a $2K_2$, a contradiction. Therefore $N(f) \subseteq N(f')$ or $N(f') \subseteq N(f)$, so $f$ and $f'$ are comparable, contradicting the reducedness of $G$. Hence $F_1 = \{f\}$.

Finally, we show (iii). Suppose $U = \{u\}$. If there exists $z\in Z\setminus N(u)$, then by \cref{lem:basic}\ref{it:ZR}, $z$ has a neighbor $x\in Y_i\cup F_i$ for some $i\in [5]$, and by \cref{lem:basic}\ref{it:UF} either the edges $zx,uc_{i}$ or $zx, uc_{i+1}$ induce a $2K_2$ depending on the type of $x$. Thus, $u$ is complete to $Z$. 

Next suppose that $y\in Y_1$. By \cref{lem:basic}\ref{it:YY+2}, $y$ is anticomplete to at least one of $Y_3$ and $Y_4$. Without loss of generality, suppose that $y$ is anticomplete to $Y_4$. Consider the induced $5$-cycle $C^{(y)}=c_1c_2fc_4yc_1$. Note that $c_3\in F^{(y)}_1$. Thus, by the minimality of our choice of $C$, we must have $|U^{(y)}|\geq |U|=1$. Let $u'\in U^{(y)}$. Since $u'$ is adjacent to $c_1,c_2,c_4$, we have $u'\in U\cup F_3\cup F_5\cup Y_4$. However, none of these possibilities can occur. We have $u'\notin U$ since $u$ is the unique vertex of $U$ and $uf\notin E(G)$; $F_3=\varnothing$ by (ii); $u'\notin Y_4$ since $u'y\in E(G)$; and $u'\notin F_5$ since $u'f\in E(G)$ while $F_5$ is anticomplete to $F_1$ by \cref{lem:basic}\ref{it:FF}. This is a contradiction. Hence $Y_1=\varnothing$.

Lastly, if $Y_3\cup Y_4 = \varnothing$, then every neighbor of $f$ is a neighbor of $u$. This follows from the definitions, (i), the fact that $u$ is complete to $Z$, and \cref{lem:basic}. Since $fu\notin E(G)$, this contradicts the reducedness of $G$. Therefore $Y_3\cup Y_4\neq\varnothing$.
\end{proofc}

\begin{claim}\label{claim:U}
    $U$ is empty
\end{claim}

\begin{proofc}
   Suppose not. By \cref{claim:min_rim}, we have
\[U=\{u\},\quad F_1=\{f\},\quad Y_1=\varnothing,\quad Y_3\cup Y_4\neq\varnothing, \quad\text{and}\quad u\text{ is complete to }R_i\cup Z\text{ for all }i\in[5].\]
Thus, together with \cref{lem:basic}, with respect to the core $(C;f,u)$, any profile $T[\varepsilon_1\varepsilon_2]$ is forbidden if $T\in\{U,F_1,Y_1\}$, or if $T\in\{R_1,R_2,R_3,R_4,R_5,Z\}$ and $\varepsilon_2=0$. Moreover, every vertex in $Y_3\cup Y_4$ has profile $Y_3[10]$ or $Y_4[10]$, respectively. Since $Y_3\cup Y_4\neq\varnothing$, at least one of these profiles is realized.

\textbf{Case 1: There exists $y\in Y_3$ and $y'\in Y_4$.} Define the core $K=(C,f,u,y,y')$. We refer the reader to Appendix~\ref{app:B1} for the admissible profiles and elimination rounds. The surviving admissible profiles are
$$F_5[0010],\quad F_2[0001],\quad R_3[0111],\quad R_4[0111].$$

Moreover, no two distinct profiles from this list can both be realized. To see this, let $v$ and $w$ realize distinct profiles $\sigma$ and $\tau$ from the above list. Using \cref{lem:basic}, the following table gives either an induced $2K_2$ or a $K_4$ for every possible pair.

\begin{center}\small
\begin{tabular}[h!]{@{}l l l l@{}}
\toprule
$\sigma$ & $\tau$ & edge $vw$ gives & nonedge $vw$ gives\\
\midrule
$F_{5}[0010]$ & $F_{2}[0001]$ & $K_{4}$ on $c_{3},c_{4},v,w$   & $2K_{2}$ with edges $vc_{2}$, $wc_{5}$\\
$F_{5}[0010]$ & $R_{3}[0111]$ & $2K_{2}$ with edges $vw$, $c_{5}f$ & $2K_{2}$ with edges $vc_{3}$, $wy'$\\
$F_{5}[0010]$ & $R_{4}[0111]$ & $K_{4}$ on $c_{3},y,v,w$       & $2K_{2}$ with edges $vc_{2}$, $wc_{5}$\\
$F_{2}[0001]$ & $R_{3}[0111]$ & $K_{4}$ on $c_{4},y',v,w$      & $2K_{2}$ with edges $vc_{5}$, $wc_{2}$\\
$F_{2}[0001]$ & $R_{4}[0111]$ & $2K_{2}$ with edges $vw$, $c_{2}f$ & $2K_{2}$ with edges $vc_{4}$, $wy$\\
$R_{3}[0111]$ & $R_{4}[0111]$ & $K_{4}$ on $y,y',v,w$          & $2K_{2}$ with edges $vc_{2}$, $wc_{5}$\\
\bottomrule
\end{tabular}
\end{center}

Then $V(G)\setminus V(K)$ has at most one vertex. Indeed, if two vertices realize the same profile, then every vertex outside $K$ realizes that same profile, since no two distinct surviving profiles can be realized simultaneously. These vertices are pairwise nonadjacent by \cref{lem:basic}\ref{it:indep} and have the same neighborhood in $K$, so they have the same neighborhood in $G$, contradicting reducedness.

If $V(G)=V(K)$, or $V(G)\setminus V(K)=\{t\}$ where $t$ realizes either $R_3[0111]$ or $R_4[0111]$, then we have the following $4$-coloring.
$$I_1=\{u\}\quad I_2=\{c_1,f\}\cup\bigl(V(G)\setminus V(K)\bigr)\quad I_3=\{c_2,c_4,y\}\quad I_4=\{c_3,c_5,y'\}.$$ Therefore, we apply Corollary~\ref{cor:mainrecolor} with the color class containing $u$, a contradiction.

Now suppose that $V(G)\setminus V(K)=\{t\}$, where $t$ realizes either $F_5[0010]$ or $F_2[0001]$. We have the following $4$-coloring.
$$I_1=\{c_1\}\quad I_2=\{c_2,c_4,y\}\quad I_3=\{c_3,c_5,y'\}\quad I_4=\{u,f,t\}.$$ Thus $G-c_1$ is $3$-colorable. For any $5$-coloring of $G$, let $S$ be the color class containing $c_1$. Since $G-S$ is an induced subgraph of $G-c_1$, we have $\chi(G-S)\leq 3$. Therefore, Corollary~\ref{cor:mainrecolor} applies with $I=S$, a contradiction.

\textbf{Case 2: Exactly one of $Y_3,Y_4$ is nonempty.} Without loss of generality, let $Y_4=\varnothing$, and let $y\in Y_3$. Define the core $K=(C;f,u,y)$. We refer the reader to Appendix~\ref{app:B2} for the admissible profiles and elimination rounds. The surviving profiles are
\[
Y_{3}[100],\quad Z[110],\quad R_{5}[110],\quad R_{1}[110],\quad
F_{5}[001],\quad Y_{2}[001],\quad R_{4}[011],\quad Z[111].
\]

Since these profiles exhaust all vertices outside $K$, in the colorings below we continue to use $Y_3,Z,R_5,Y_2$, and $R_4$ for the corresponding vertex sets of $G$.

We now split according to whether the profiles $R_1[110]$ and $F_5[001]$ are realized.

\medskip
\noindent\textbf{Neither $R_1[110]$ nor $F_5[001]$ is realized.}
Then every vertex outside $K$ belongs to $Y_3\cup Z\cup R_5\cup Y_2\cup R_4$. We have the following $4$-coloring:
\[
I_1=\{u\}\quad
I_2=\{c_1,f\}\cup Y_2\cup R_4,\qquad
I_3=\{c_2,c_4\}\cup Y_3,\qquad
I_4=\{c_3,c_5\}\cup R_5\cup Z.
\]
Thus $G-u$ is $3$-colorable. For any $5$-coloring of $G$, let $S$ be the color class containing $u$. Since $G-S$ is an induced subgraph of $G-u$, we have $\chi(G-S)\leq 3$. Therefore, Corollary~\ref{cor:mainrecolor} applies with $I=S$, a contradiction.

\medskip
\noindent\textbf{The profile $R_1[110]$ is realized.}
Let $r\in R_1$ realize $R_1[110]$ and extend the core to $K'=(C;f,u,y,r)$. We refer the reader to Appendix~\ref{app:B3} for the admissible profiles and elimination rounds, which leave only $F_5[0011]$. Therefore, $V(G)\setminus V(K')$ has at most one vertex. We have the following $4$-coloring:
\[
I_1=\{c_5\}\quad
I_2=\{c_2,c_4,y\},\quad
I_3=\{c_1,c_3,r\},\quad
I_4=\{u,f\}\cup\bigl(V(G)\setminus V(K')\bigr).
\]
Letting $S$ be the color class containing $c_5$, Corollary~\ref{cor:mainrecolor} applies with $I=S$, a contradiction.

\medskip
\noindent\textbf{The profile $F_5[001]$ is realized.}
Let $g'\in F_5$ realize $F_5[001]$. Then $R_1=\varnothing$. Extend the core to $K''=(C;f,u,y,g')$. We refer the reader to Appendix~\ref{app:B4} for the admissible profiles and elimination rounds, which leave only $R_5[1101]$ and $Y_2[0011]$. These two profiles cannot both be realized: if $v\in R_5$ realizes $R_5[1101]$ and $w\in Y_2$ realizes $Y_2[0011]$, then an edge $vw$ gives a $K_4$ on $c_4,g',v,w$, while a nonedge gives an induced $2K_2$ with edges $vu$ and $wy$.

Therefore, $V(G)\setminus V(K'')$ has at most one vertex. If $V(G)=V(K'')$, or $V(G)\setminus V(K'')=\{t\}$ where $t$ realizes $Y_2[0011]$, then
\[
I_1=\{c_5\}\quad
I_2=\{c_1,c_3\}\cup\bigl(V(G)\setminus V(K'')\bigr),\qquad
I_3=\{c_2,c_4,y\},\qquad
I_4=\{u,f,g'\}
\]
is a $4$-coloring. If $V(G)\setminus V(K'')=\{q'\}$ where $q'$ realizes $R_5[1101]$, then
\[
I_1=\{c_1\}\quad
I_2=\{c_3,c_5,q'\},\qquad
I_3=\{c_2,c_4,y\},\qquad
I_4=\{u,f,g'\}
\]
is a $4$-coloring. In either case, letting $S$ be the color class containing $c_5$ and $c_1$ respectively, and applying Corollary~\ref{cor:mainrecolor} with $I=S$, gives a contradiction.

Thus every possibility leads to a contradiction, and hence $U=\varnothing$.
\end{proofc}

\begin{claim}\label{claim:overlap}
    At least one of $Y_1$ and $F_2\cup F_5$ is empty.
\end{claim}

\begin{proofc}
    Suppose not. Without loss of generality, we can choose $a\in Y_1$ and $g\in F_2$. The vertices $f,a,g$ are pairwise nonadjacent by Lemma \ref{lem:basic}. Consider the core $K=(C,f,a,g)$, where $a\in Y_1[0]$ and $g\in F_2[00]$. Note that $a\preceq_K g$, because
    \[
    N_K(a)=\{c_1,c_3,c_4\}\subseteq\{c_1,c_3,c_4,c_5\}=N_K(g).
    \]
    We refer the reader to Appendix \ref{app:B5} for the admissible profiles and elimination rounds. Among the seventeen surviving profiles, $Y_2[010]$ is the only rescuer for $a$. Therefore there must be a vertex $x\in Y_2[010]$.
 
    Now consider the core $K'=(C,f,a,g,x)$. We refer the reader to Appendix \ref{app:B6} for the admissible profiles and elimination rounds, which leave no surviving profiles. Thus, $V(G)=V(K')$. Then $G$ has the following $4$-coloring
    \[
    I_1=\{c_3\}\quad I_2=\{c_1,c_4\}\quad I_3=\{c_2,c_5,a\}\quad I_4=\{f,g,x\}.
    \]
    Letting $S$ be the color class containing $c_3$, Corollary~\ref{cor:mainrecolor} applies with $I=S$, a contradiction.
\end{proofc}
 
\begin{claim}\label{claim:Y1}
    $Y_1$ is empty.
\end{claim}
 
\begin{proofc}
Suppose that $a\in Y_1$. By Claim~\ref{claim:overlap}, $F_2=F_5=\varnothing$, so the only nonempty $F$-set is $F_1=\{f\}$. Define the core $K=(C;f,a)$. For a profile $T[\varepsilon_1\varepsilon_2]$, we first record the values of $\varepsilon_1$ and $\varepsilon_2$ that are forced by \cref{lem:basic}. Since $f\in F_1$ and $a\in Y_1$, the relevant items are \ref{it:indep}, \ref{it:FY}, \ref{it:YY}, \ref{it:RR}, and \ref{it:FR}.

\begin{itemize}[itemsep=1pt,topsep=2pt]
\item Every vertex of $Y_1$ has $\varepsilon_1=\varepsilon_2=0$: we have $\varepsilon_1=0$ because $F_1$ is anticomplete to $Y_1$ by \ref{it:FY}, and $\varepsilon_2=0$ because $Y_1$ is independent by \ref{it:indep}. Thus, the only $Y_1$-profile is $Y_1[00]$.

\item Every vertex of $Y_2\cup Y_5$ has $\varepsilon_1=0$ and $\varepsilon_2=1$: we have $\varepsilon_1=0$ because $F_1$ is anticomplete to $Y_5\cup Y_1\cup Y_2$ by \ref{it:FY}, and $\varepsilon_2=1$ because $Y_5$ is complete to $Y_1$ and $Y_1$ is complete to $Y_2$ by \ref{it:YY}. Thus, the only profiles are $Y_2[01]$ and $Y_5[01]$.

\item Every vertex of $Y_3\cup Y_4$ has $\varepsilon_1=1$, since $F_1$ is complete to $Y_3\cup Y_4$ by \ref{it:FY}. Thus, the possible profiles are $Y_3[10]$, $Y_3[11]$, $Y_4[10]$, and $Y_4[11]$. The value of $\varepsilon_2$ is not fixed, although by \ref{it:YY+2}, the vertex $a$ is anticomplete to at least one of $Y_3$ and $Y_4$, so $Y_3[11]$ and $Y_4[11]$ cannot both occur.

\item Every vertex of $R_2\cup R_5$ has $\varepsilon_1=1$, since $F_1$ is complete to $R_5\cup R_2$ by \ref{it:FR}. Thus, the possible profiles are $R_2[10]$, $R_2[11]$, $R_5[10]$, and $R_5[11]$.

\item Every vertex of $R_1$ has $\varepsilon_2=1$, since $Y_1$ is complete to $R_1$ by \ref{it:RR}. Thus, the possible profiles are $R_1[01]$ and $R_1[11]$.

\item Neither $\varepsilon_1$ nor $\varepsilon_2$ is forced by \cref{lem:basic} for vertices in $R_3$, $R_4$, or $Z$. However, some profiles are not admissible. Indeed, a vertex $v$ realizing $R_3[10]$, $R_4[10]$, or $Z[10]$ together with the edges $vf$ and $ac_1$ gives an induced $2K_2$; a vertex $v$ realizing $R_3[01]$ together with the edges $va$ and $fc_5$ gives an induced $2K_2$; and a vertex $v$ realizing $R_4[01]$ or $Z[01]$ together with the edges $va$ and $fc_2$ gives an induced $2K_2$. Thus, the only admissible profiles of these types are $R_3[00]$, $R_3[11]$, $R_4[00]$, $R_4[11]$, $Z[00]$, and $Z[11]$.
\end{itemize}

The remaining types contribute no vertices outside the core as $U=\varnothing$ by \cref{claim:U}, $F_1=\{f\}$, $F_2=F_5=\varnothing$ by \cref{claim:overlap}, and $F_3=F_4=\varnothing$ by \ref{it:FF}. We next show a pivotal structural observation used to finish the claim.

\medskip 
 
\textbf{The vertex $a$ is anticomplete to $Y_3\cup Y_4$.} Suppose not. Without loss of generality, suppose that $y\in Y_3\cap N(a)$, whose profile relative to $(C;f,a)$ is $Y_3[11]$. Extend the core to $K'=(C;f,a,y)$. We have $c_2\preceq_{K'}y$ and $c_4\preceq_{K'}y$, since
\[N_{K'}(c_2)=\{c_1,c_3,f\},\qquad N_{K'}(c_4)=\{c_3,c_5,f,a\},\qquad N_{K'}(y)=\{c_1,c_3,c_5,f,a\}.\]
Hence, there are vertices $p$ and $q$ where $p$ realizes an admissible profile containing $c_2$ but not $y$, while $q$ realizes an admissible profile containing $c_4$ but not $y$. We first determine the possible profiles of $p$ and $q$.

\emph{The profile of $p$ is $Y_5[010]$, $R_1[010]$, or $R_1[110]$.}
Since $p$ is adjacent to $c_2$, its type must be one of $Y_2,Y_4,Y_5,R_1$, or $R_3$, as all other types adjacent to $c_2$ are empty or already contained in the core. Since $p$ must be nonadjacent to $y\in Y_3$, the types $Y_2$, $Y_4$, and $R_3$ are impossible by \cref{lem:basic}\ref{it:YY}, \ref{it:RR}. If $p\in Y_5$, then $p$ is nonadjacent to $f$, adjacent to $a$, and nonadjacent to $y$, giving the profile $Y_5[010]$. If $p\in R_1$, then $p$ is adjacent to $a$ and nonadjacent to $y$, while its adjacency to $f$ is not fixed, giving the profiles $R_1[010]$ and $R_1[110]$.

\emph{The profile of $q$ is $Y_1[000]$, $R_5[100]$, or $R_5[110]$.}
Since $q$ is adjacent to $c_4$, its type must be one of $Y_1,Y_2,Y_4,R_3$, or $R_5$, as all other types adjacent to $c_4$ are empty or already contained in the core. Since $q$ must be nonadjacent to $y\in Y_3$, the types $Y_2$, $Y_4$, and $R_3$ are impossible by \cref{lem:basic}\ref{it:YY}, \ref{it:RR}. If $q\in Y_1$, then $q$ is nonadjacent to both $f$ and $a$, and by the rescuer condition it is nonadjacent to $y$, giving the profile $Y_1[000]$. If $q\in R_5$, then $q$ is adjacent to $f$ and nonadjacent to $y$, while its adjacency to $a$ is not fixed, giving the profiles $R_5[100]$ and $R_5[110]$.

From the possible profiles above, $p$ is nonadjacent to $c_4$ and $q$ is nonadjacent to $c_2$, so in particular $p\neq q$. Moreover, $pq\in E(G)$, since otherwise the edges $pc_2$ and $qc_4$ induce a $2K_2$. We now eliminate all but two possible pairs. If $q$ realizes $Y_1[000]$ and $p$ realizes either $Y_5[010]$ or $R_1[010]$, then the edges $pq$ and $fy$ induce a $2K_2$. If $p$ realizes either $R_1[010]$ or $R_1[110]$ and $q$ realizes either $R_5[100]$ or $R_5[110]$, then the edges $pq$ and $c_3y$ induce a $2K_2$. Finally, if $p$ realizes $Y_5[010]$ and $q$ realizes $R_5[110]$, then $(c_1,c_2,f,c_4,a,c_3,c_5,q,y,p)$ is an induced copy of $B_{10}$ in its defining order, as recorded in row~\hyperlink{obs:B14}{B.14} of Appendix~\ref{app:obstruction-witnesses}. Thus, the only pairs remaining are
\[(p,q)\in\{(Y_5[010],R_5[100]),\,(R_1[110],Y_1[000])\}.\]

Suppose first that $p$ realizes $Y_5[010]$ and $q$ realizes $R_5[100]$. Relative to the core $(C;f,a,y,p)$, the vertex $q$ realizes $R_5[1001]$. Extend the core to $K''=(C;f,a,y,p,q)$. Then 
\[N_{K''}(q)=\{c_1,c_4,f,p\}\subseteq\{c_1,c_4,f,y,p\}=N_{K''}(c_5),\]
so $q\preceq_{K''}c_5$. By Appendix~\ref{app:A7}, the only admissible profile containing $q$ but not $c_5$ is $Y_1[00111]$, which is excluded by the $B_{10}$-obstruction in row~\hyperlink{obs:B15}{B.15}. Hence, no vertex of $G-K''$ can be adjacent to $q$ and nonadjacent to $c_5$, contradicting reducedness.

It remains that $p$ realizes $R_1[110]$ and $q$ realizes $Y_1[000]$. Relative to the core $(C;f,a,y,p)$, the vertex $q$ realizes $Y_1[0001]$. Extend the core to $K'''=(C;f,a,y,p,q)$. Then
\[N_{K'''}(q)=\{c_1,c_3,c_4,p\}\subseteq\{c_1,c_3,c_4,y,p\}=N_{K'''}(a),\]
so $q\preceq_{K'''}a$. By Appendix~\ref{app:A8}, the only admissible profiles containing $q$ but not $a$ are $R_5[10111]$ and $Y_4[10101]$, which are excluded by the obstructions in rows~\hyperlink{obs:B17}{B.17} and~\hyperlink{obs:B18}{B.18}, respectively. Hence, no vertex of $G-K'''$ can be adjacent to $q$ and nonadjacent to $a$, again contradicting reducedness.

Therefore, $a$ is anticomplete to $Y_3\cup Y_4$.
    
\medskip
 
\textbf{Completing the claim}. Together with the adjacencies and nonadjacencies recorded at the start of the proof, this determines the profile relative to $K=(C;f,a)$ of every vertex outside the core. The possible profiles are
\[
\renewcommand{\arraystretch}{1.15}
\begin{array}{c@{\quad}c@{\quad}c@{\quad}c@{\quad}c@{\quad}c@{\quad}c@{\quad}c@{\quad}c}
Y_1[00], & Y_2[01], & Y_3[10], & Y_4[10], & Y_5[01], & R_1[01], & R_1[11], & R_2[10], & R_2[11],\\
R_3[00], & R_3[11], & R_4[00], & R_4[11], & R_5[10], & R_5[11], & Z[00], & Z[11]. &
\end{array}
\]

\textbf{Case 1: $Y_2=Y_5=\varnothing$ and $R_1[11]$ is not realized.}
For the core $K=(C;f,a)$, Appendix~\ref{app:B9} leaves no surviving profile, so $V(G)=V(K)$. Then $G$ has the following $4$-coloring:
\[I_1=\{c_1\},\qquad I_2=\{c_3,c_5\},\qquad I_3=\{c_2,c_4\},\qquad I_4=\{f,a\}.\]
Thus, $G-c_1$ is $3$-colorable. For any $5$-coloring of $G$, let $S$ be the color class containing $c_1$. Since $G-S$ is an induced subgraph of $G-c_1$, we have $\chi(G-S)\leq 3$. Therefore, Corollary~\ref{cor:mainrecolor} applies with $I=S$, a contradiction.

\textbf{Case 2: $Y_2\cup Y_5\ne\varnothing$.}
Without loss of generality, choose $x\in Y_2$. Then $x$ realizes the profile $Y_2[01]$. For the core $K=(C;f,a,x)$, we have $c_1\preceq_K x$, because
\[N_K(c_1)=\{c_2,c_5,a\}\subseteq\{c_2,c_4,c_5,a\}=N_K(x).\]
By Appendix~\ref{app:B10}, the only surviving profile containing $c_1$ but not $x$ is $R_5[110]$. Hence, reducedness forces a vertex $q'$ realizing $R_5[110]$. For the core $K'=(C;f,a,x,q')$, Appendix~\ref{app:B11} leaves no surviving profile, so $V(G)=V(K')$. Then $G$ has the following $4$-coloring:
\[I_1=\{c_4\},\qquad I_2=\{f,a\},\qquad I_3=\{c_1,c_3,x\},\qquad I_4=\{c_2,c_5,q'\}.\]
Thus, $G-c_4$ is $3$-colorable. For any $5$-coloring of $G$, let $S$ be the color class containing $c_4$. Since $G-S$ is an induced subgraph of $G-c_4$, we have $\chi(G-S)\leq 3$. Therefore, Corollary~\ref{cor:mainrecolor} applies with $I=S$, a contradiction.

\textbf{Case 3: $Y_2=Y_5=\varnothing$ and $R_1[11]$ is realized.}
Let $r$ realize $R_1[11]$. For the core $K=(C;f,a,r)$, we have $c_1\preceq_K r$, because
\[N_K(c_1)=\{c_2,c_5,a\}\subseteq\{c_2,c_5,f,a\}=N_K(r).\]
By Appendix~\ref{app:B12}, the only surviving profiles containing $c_1$ but not $r$ are $Y_3[100]$ and $Y_4[100]$. Hence, reducedness forces a vertex realizing one of these profiles. By symmetry, let $y$ realize $Y_3[100]$. For the core $K'=(C;f,a,r,y)$, Appendix~\ref{app:B13} leaves no surviving profile, so $V(G)=V(K')$. Then $G$ has the following $4$-coloring:
\[I_1=\{c_3\},\qquad I_2=\{c_1,f\},\qquad I_3=\{c_2,c_5,a\},\qquad I_4=\{c_4,r,y\}.\]
Thus, $G-c_3$ is $3$-colorable. For any $5$-coloring of $G$, let $S$ be the color class containing $c_3$. Since $G-S$ is an induced subgraph of $G-c_3$, we have $\chi(G-S)\leq 3$. Therefore, Corollary~\ref{cor:mainrecolor} applies with $I=S$, a contradiction.

This exhausts all possibilities, and therefore $Y_1=\varnothing$.
\end{proofc}

 We may suppose that $Y_1$ is empty.

\begin{claim}\label{claim:F25}
    $F_2\cup F_5$ is empty.
\end{claim}
 
\begin{proofc}
    Suppose not. By \cref{lem:basic}\ref{it:FF} at least one of $F_5$ and $F_2$ is empty. Without loss of generality, let $g\in F_2$. Note $f$ and $g$ are nonadjacent by \cref{lem:basic}\ref{it:FF}. Define the core $K=(C;f,g)$. For a profile $T[\varepsilon_1\varepsilon_2]$, we first record the values of $\varepsilon_1$ and $\varepsilon_2$ that are forced by \cref{lem:basic} and prior claims.
    
\begin{itemize}[itemsep=0pt,topsep=2pt]

    \item The possible $Y$-profiles are $Y_2[00]$, $Y_3[10]$, $Y_4[11]$, and $Y_5[01]$. Indeed, $Y_1=\varnothing$ by \cref{claim:Y1}, while $F_1$ is complete to $Y_3\cup Y_4$ and anticomplete to $Y_2\cup Y_5$, and $F_2$ is complete to $Y_4\cup Y_5$ and anticomplete to $Y_2\cup Y_3$. Thus every vertex of $Y_2$, $Y_3$, $Y_4$, and $Y_5$ realizes the corresponding profile above.

    \item For four of the five $R$-sets, one of the two adjacencies to $f$ and $g$ is fixed. Since $F_2$ is complete to $R_1\cup R_3$, the possible profiles are $R_1[01]$, $R_1[11]$, $R_3[01]$, and $R_3[11]$. Since $F_1$ is complete to $R_2\cup R_5$, the possible profiles are $R_2[10]$, $R_2[11]$, $R_5[10]$, and $R_5[11]$.

    \item For vertices in $R_4$ or $Z$, neither adjacency to $f$ nor adjacency to $g$ is forced by \cref{lem:basic}. However, the mixed profiles are not admissible. Indeed, if a vertex $v$ realizes $R_4[10]$ or $Z[10]$, then the edges $vf$ and $gc_1$ induce a $2K_2$, while if $v$ realizes $R_4[01]$ or $Z[01]$, then the edges $vg$ and $fc_2$ induce a $2K_2$. Thus, the only admissible profiles of these types are $R_4[00]$, $R_4[11]$, $Z[00]$, and $Z[11]$.

    \item The only remaining $F$-profile is $F_2[00]$. Indeed, $F_1=\{f\}$ and $F_3=F_4=F_5=\varnothing$, while every other vertex of $F_2$ is nonadjacent to both $f$ and $g$ by \cref{lem:basic}\ref{it:indep}, \ref{it:FF}.
\end{itemize}
We now establish a structural observation that will complete the claim.
    
\medskip

\textbf{At least one profile in $\{R_1[11],R_2[11],Y_3[10],Y_4[11],Y_5[01]\}$ is realized in $G$.}
Suppose not. Since every vertex of $Y_3$, $Y_4$, $Y_5$, $R_1\cap N(f)$, and $R_2\cap N(g)$ realizes the corresponding profile in this set, all five sets are empty.

By Appendix~\ref{app:B14}, the only surviving profiles are $R_3[01]$ and $R_5[10]$. Any vertex realizing $R_3[01]$ is adjacent to every vertex realizing $R_5[10]$: otherwise, if $v$ realizes $R_3[01]$ and $w$ realizes $R_5[10]$ with $vw\notin E(G)$, then the edges $vg$ and $wf$ induce a $2K_2$.

Moreover, each of these two profiles is realized by at most one vertex. Indeed, if two distinct vertices realize the same profile, then they lie in the same independent class by \cref{lem:basic}\ref{it:indep}, have the same neighborhood in the core, and are both adjacent to every vertex realizing the other surviving profile. Since no other profile is realized, they have the same neighborhood in $G$, contradicting reducedness.

Thus either $|V(G)\setminus V(K)|\leq 1$, or $V(G)\setminus V(K)=\{r_3,r_5\}$, where $r_3$ realizes $R_3[01]$ and $r_5$ realizes $R_5[10]$. If $|V(G)\setminus V(K)|\leq 1$, then $G$ has the $4$-coloring
\[
I_1=\{c_1\},\qquad
I_2=\{c_3,c_5\}\cup (V(G)\setminus V(K)),\qquad
I_3=\{c_2,c_4\},\qquad
I_4=\{f,g\}.
\]
Thus, $G-c_1$ is $3$-colorable. For any $5$-coloring of $G$, let $S$ be the color class containing $c_1$. Since $G-S$ is an induced subgraph of $G-c_1$, we have $\chi(G-S)\leq 3$. Therefore, Corollary~\ref{cor:mainrecolor} applies with $I=S$, a contradiction.

If $V(G)\setminus V(K)=\{r_3,r_5\}$, then $G$ has the $4$-coloring
\[
I_1=\{c_4\},\qquad
I_2=\{c_1,f,r_3\},\qquad
I_3=\{c_2,g,r_5\},\qquad
I_4=\{c_3,c_5\}.
\]
Thus, $G-c_4$ is $3$-colorable. For any $5$-coloring of $G$, let $S$ be the color class containing $c_4$. Since $G-S$ is an induced subgraph of $G-c_4$, we have $\chi(G-S)\leq 3$. Therefore, Corollary~\ref{cor:mainrecolor} applies with $I=S$, a contradiction.

\medskip

\textbf{Completing the claim.}
By the previous paragraph, at least one of
\[R_1[11], R_2[11], Y_3[10], Y_4[11], Y_5[01]\]
is realized. We rule out these possibilities in three cases.

\textbf{Case 1: $R_1[11]$ or $Y_3[10]$ is realized.}
We first observe that either realization forces the other. Suppose that $r$ realizes $R_1[11]$. Extend the core to $K'=(C;f,g,r)$. We have $c_1\preceq_{K'}r$, since
\[
N_{K'}(c_1)=\{c_2,c_5,g\}\subseteq\{c_2,c_5,f,g\}=N_{K'}(r).
\]
By Appendix~\ref{app:B15}, the only surviving profile containing $c_1$ but not $r$ is $Y_3[100]$. Hence, reducedness forces a vertex $y$ realizing $Y_3[100]$ relative to $K'$. In particular, $y\in Y_3$ and $ry\notin E(G)$.

Conversely, suppose that $y$ realizes $Y_3[10]$. Extend the core to $K'=(C;f,g,y)$. We have $c_2\preceq_{K'}y$, since
\[
N_{K'}(c_2)=\{c_1,c_3,f\}\subseteq\{c_1,c_3,c_5,f\}=N_{K'}(y).
\]
By Appendix~\ref{app:B16}, the only surviving profile containing $c_2$ but not $y$ is $R_1[110]$. Hence, reducedness forces a vertex $r$ realizing $R_1[110]$ relative to $K'$. In particular, $r$ realizes $R_1[11]$ relative to $(C;f,g)$ and $ry\notin E(G)$.

Thus, in either case, $G$ contains the core $K''=(C;f,g,r,y)$ with $r$ realizing $R_1[11]$ and $y$ realizing $Y_3[100]$. By Appendix~\ref{app:B17}, no surviving profile remains, so $V(G)=V(K'')$. Then $G$ has the $4$-coloring
\[
I_1=\{c_5\},\qquad I_2=\{f,g\},\qquad I_3=\{c_1,c_3,r\},\qquad I_4=\{c_2,c_4,y\}.
\]
Thus, $G-c_5$ is $3$-colorable. For any $5$-coloring of $G$, let $S$ be the color class containing $c_5$. Since $G-S$ is an induced subgraph of $G-c_5$, we have $\chi(G-S)\leq 3$. Therefore, Corollary~\ref{cor:mainrecolor} applies with $I=S$, a contradiction.

\textbf{Case 2: $R_2[11]$ or $Y_5[01]$ is realized.}
Again, either realization forces the other. Suppose first that $b$ realizes $Y_5[01]$. Extend the core to $K'=(C;f,g,b)$. We have $c_1\preceq_{K'}b$, since
\[
N_{K'}(c_1)=\{c_2,c_5,g\}\subseteq\{c_2,c_3,c_5,g\}=N_{K'}(b).
\]
By Appendix~\ref{app:B18}, the only surviving profile containing $c_1$ but not $b$ is $R_2[110]$. Hence, reducedness forces a vertex $q$ realizing $R_2[110]$ relative to $K'$. In particular, $q$ realizes $R_2[11]$ relative to $(C;f,g)$ and $qb\notin E(G)$.

Conversely, suppose that $q$ realizes $R_2[11]$. Extend the core to $K'=(C;f,g,q)$. We have $c_2\preceq_{K'}q$, since
\[
N_{K'}(c_2)=\{c_1,c_3,f\}\subseteq\{c_1,c_3,f,g\}=N_{K'}(q).
\]
By Appendix~\ref{app:B19}, the only surviving profile containing $c_2$ but not $q$ is $Y_5[010]$. Hence, reducedness forces a vertex $b$ realizing $Y_5[010]$ relative to $K'$. In particular, $b$ realizes $Y_5[01]$ relative to $(C;f,g)$ and $qb\notin E(G)$.

Thus, in either case, $G$ contains the core $K''=(C;f,g,q,b)$ with $q$ realizing $R_2[11]$ and $b$ realizing $Y_5[010]$. By Appendix~\ref{app:B20}, no surviving profile remains, so $V(G)=V(K'')$. Then $G$ has the $4$-coloring
\[
I_1=\{c_3\},\qquad I_2=\{f,g\},\qquad I_3=\{c_2,c_5,q\},\qquad I_4=\{c_1,c_4,b\}.
\]
Thus, $G-c_3$ is $3$-colorable. For any $5$-coloring of $G$, let $S$ be the color class containing $c_3$. Since $G-S$ is an induced subgraph of $G-c_3$, we have $\chi(G-S)\leq 3$. Therefore, Corollary~\ref{cor:mainrecolor} applies with $I=S$, a contradiction.

\textbf{Case 3: $Y_4[11]$ is realized.}
By the previous two cases, $Y_3=Y_5=\emptyset$, and neither $R_1[11]$ nor $R_2[11]$ is realized. Let $y'\in Y_4$. Then $y'$ realizes $Y_4[11]$. Extend the core to $K'=(C;f,g,y')$. We have $c_3\preceq_{K'}y'$ and $c_5\preceq_{K'}y'$, since
\[
N_{K'}(c_3)=\{c_2,c_4,f,g\},\qquad
N_{K'}(c_5)=\{c_1,c_4,f,g\},\qquad
N_{K'}(y')=\{c_1,c_2,c_4,f,g\}.
\]
Hence, reducedness forces a vertex $q$ adjacent to $c_3$ and nonadjacent to $y'$, and a vertex $s$ adjacent to $c_5$ and nonadjacent to $y'$. We determine their possible profiles.

\emph{The profile of $q$ is $R_2[100]$.}
Since $q$ is adjacent to $c_3$, its type must be one of $F_2$, $R_2$, or $R_4$, as all other types adjacent to $c_3$ are empty or already contained in the core. A vertex of $F_2$ is adjacent to $y'$ by \cref{lem:basic}\ref{it:FY}, while every vertex of $R_4$ is adjacent to $y'$ by \cref{lem:basic}\ref{it:RR}. Since $q$ must be nonadjacent to $y'$, neither type is possible. Thus $q\in R_2$. Every vertex of $R_2$ is adjacent to $f$, while $q$ must be nonadjacent to $g$ because $R_2[11]$ is not realized. Together with $qy'\notin E(G)$, this gives the profile $R_2[100]$.

\emph{The profile of $s$ is $Y_2[000]$ or $R_1[010]$.}
Since $s$ is adjacent to $c_5$, its type must be one of $F_2$, $Y_2$, $R_1$, or $R_4$, as all other types adjacent to $c_5$ are empty or already contained in the core. Again, every vertex of $F_2$ or $R_4$ is adjacent to $y'$ by \cref{lem:basic}\ref{it:FY}, \ref{it:RR}, so neither type is possible. If $s\in Y_2$, then $s$ is nonadjacent to both $f$ and $g$, and also to $y'$, giving the profile $Y_2[000]$. If $s\in R_1$, then $s$ is adjacent to $g$, while it must be nonadjacent to $f$ because $R_1[11]$ is not realized. Together with $sy'\notin E(G)$, this gives the profile $R_1[010]$.

In particular, $q\neq s$. If $s$ realizes $Y_2[000]$ and $qs\notin E(G)$, then the edges $qc_1$ and $sc_4$ induce a $2K_2$; if $qs\in E(G)$, then the edges $qs$ and $gy'$ induce a $2K_2$. If instead $s$ realizes $R_1[010]$ and $qs\notin E(G)$, then the edges $qc_3$ and $sc_5$ induce a $2K_2$; if $qs\in E(G)$, then the edges $qs$ and $c_4y'$ induce a $2K_2$. Thus no possible pair of rescuers exists, contradicting reducedness.

All three cases are impossible, contradicting the fact that at least one of $R_1[11]$, $R_2[11]$, $Y_3[10]$, $Y_4[11]$, and $Y_5[01]$ is realized. Therefore, $F_2\cup F_5=\varnothing$.
\end{proofc}

\medskip

We can now finish the proof. 
At this point of the argument the following sets are empty:
\[
U=\varnothing \text{ by \cref{claim:U}},\qquad
Y_1=\varnothing \text{ by \cref{claim:Y1}},\qquad
F_2=F_5=\varnothing \text{ by \cref{claim:F25}},
\]
and $F_3=F_4=\varnothing$ by \cref{lem:basic}\ref{it:FF}, since $F_1\ne\varnothing$.
Hence
\[
V(G)=V(C)\cup\{f\}\cup Y_2\cup Y_3\cup Y_4\cup Y_5\cup R_1\cup R_2 \cup R_3 \cup R_4 \cup R_5\cup Z .
\]

The vertices $f$ and $c_1$ are nonadjacent, so reducedness gives a vertex $w\in N(c_1)\setminus N(f)$. The neighbors $c_2,c_5$ of $c_1$ are adjacent to $f$, so $w\notin V(C)$. Among the classes listed above, we have by definition
\[
N(c_1)=\{c_2,c_5\}\cup Y_3\cup Y_4\cup R_2\cup R_5 .
\]
But $f\in F_1$ is adjacent to $c_2$ and $c_5$, is complete to $Y_3\cup Y_4$ by \cref{lem:basic}\ref{it:FY}, and is complete to $R_2\cup R_5$ by \cref{lem:basic}\ref{it:FR}, so $N(c_1)\subseteq N(f)$. Thus $c_1\preceq f$ for the nonadjacent pair $c_1,f$, contradicting the reducedness of $G$. Therefore, no such $G$ exists. We have proved that $\mathcal R_5(G)$ is connected for every $4$-chromatic $(2K_2,K_4)$-free graph $G$. Consequently, for every $(2K_2,K_4)$-free graph $G$, the graph $\mathcal R_{\chi(G)+1}(G)$ is connected. Since the class of $(2K_2,K_4)$-free graphs is hereditary, \cref{lem:hereditary} now implies that $\mathcal R_\ell(G)$ is connected for every $\ell\geq\chi(G)+1$. Hence every $(2K_2,K_4)$-free graph is recolorable.

\end{proof}

\noindent \textbf{Declaration of AI use.}
ChatGPT Sol and Claude Opus were used to help write the code for checking admissible profiles and rescuers in the main proofs of the paper. They were also used for proofreading, presentation, and grammar. The mathematical content of this paper: the statements, proof strategy, and final arguments, is the authors' own. We take full responsibility for all content presented in this manuscript.


\appendix

\section{Complete attachment certificates}\label{app:finite-certificates}

In an elimination row involving a hypothetical new vertex $v$, neighborhood comparisons are taken in the enlarged graph $K+v$. Comparisons involving only vertices of the displayed core are taken in $K$. We suppress these subscripts when the ambient graph is clear. This appendix records the finite attachment certificates used in the proof of
\cref{thm:YES}. The certificates are listed in the order in which they are invoked.
For each core, we record its ordered off-cycle vertices, the additional branch assumptions used by the verifier, the admissible profiles compatible with those assumptions, any profiles excluded directly by a $B_{10}$- or $B_{11}$-obstruction, the elimination rounds when needed, the profiles remaining after these exclusions and elimination rounds, and any core comparison used subsequently in the argument. Bit strings are read in the displayed order of the off-cycle vertices.

All finite assertions recorded in these certificates are verified by the computer
program described in \cref{subsec:computer}. In particular, the program verifies
the admissible-profile lists under the stated branch assumptions, the local comparisons and elimination rounds, the final remaining-profile lists, the stated core comparisons, and every obstruction witness cited from \cref{app:obstruction-witnesses}.

Certificates A.1--A.4 are used in Claim~\ref{claim:U}, A.5--A.6 in
Claim~\ref{claim:overlap}, A.7--A.13 in Claim~\ref{claim:Y1}, and
A.14--A.20 in Claim~\ref{claim:F25}. Each additional assumption displayed below is established in the proof before the corresponding certificate is invoked. When the final remaining-profile set is empty, no admissible profile compatible with the branch assumptions can be realized outside the displayed core, and hence the displayed core contains all vertices of $G$.

\begin{certblock}{A.1}\phantomsection\label{app:A1}\label{app:B1}
\noindent\textit{Core:} $K=(C;\,f,\,u,\,y,\,y')$ with $f\in F_{1}$, $u\in U[0]$, $y\in Y_{3}[10]$, $y'\in Y_{4}[101]$.
\par\noindent\textit{Additional assumptions:} No vertex outside the core lies in $U$, $F_{1}$, or $Y_{1}$. Every vertex outside the core in $R_i$ for some $i\in[5]$, or in $Z$, is adjacent to $u$.
\par\smallskip
\begingroup\centering
\scriptsize\setlength{\tabcolsep}{2.2pt}\renewcommand{\arraystretch}{0.80}
\textit{Admissible profiles}\\[1pt]
\begin{tabular*}{0.98\textwidth}{@{\extracolsep{\fill}}c l | c l | c l@{}}
\toprule
cycle type & strings & cycle type & strings & cycle type & strings\\
\midrule
$Z$ & $[0111]$, $[1101]$, $[1110]$ & $F_{2}$ & $[0001]$ & $F_{5}$ & $[0010]$\\
$Y_{2}$ & $[0010]$ & $Y_{3}$ & $[1001]$ & $Y_{4}$ & $[1010]$\\
$Y_{5}$ & $[0001]$ & $R_{1}$ & $[0111]$ & $R_{2}$ & $[1101]$\\
$R_{3}$ & $[0111]$, $[1110]$ & $R_{4}$ & $[0111]$, $[1101]$ & $R_{5}$ & $[1110]$\\
\bottomrule
\end{tabular*}
\par\endgroup
\smallskip
\begingroup\centering
\scriptsize\setlength{\tabcolsep}{2.2pt}\renewcommand{\arraystretch}{0.76}
\textit{Elimination rounds}\\[1pt]
\begin{tabular*}{0.98\textwidth}{@{\extracolsep{\fill}}c l l | c l l@{}}
\toprule
round & profile & comparison & round & profile & comparison\\
\midrule
1 & $Z[0111]$ & $v\preceq c_{1}$ & 1 & $Z[1101]$ & $v\preceq c_{4}$\\
1 & $Z[1110]$ & $v\preceq c_{3}$ & 1 & $Y_{2}[0010]$ & $v\preceq f$\\
1 & $Y_{3}[1001]$ & $y\preceq v$ & 1 & $Y_{4}[1010]$ & $y'\preceq v$\\
1 & $Y_{5}[0001]$ & $v\preceq f$ & 1 & $R_{1}[0111]$ & $c_{1}\preceq v$\\
1 & $R_{2}[1101]$ & $v\preceq c_{2}$ & 1 & $R_{3}[1110]$ & $c_{3}\preceq v$\\
1 & $R_{4}[1101]$ & $c_{4}\preceq v$ & 1 & $R_{5}[1110]$ & $v\preceq c_{5}$\\
\bottomrule
\end{tabular*}
\par\endgroup
\smallskip\noindent\textit{Survivors:} $F_{5}[0010]$, $F_{2}[0001]$, $R_{3}[0111]$, $R_{4}[0111]$.
\end{certblock}

\begin{certblock}{A.2}\phantomsection\label{app:A2}\label{app:B2}
\noindent\textit{Core:} $K=(C;\,f,\,u,\,y)$ with $f\in F_{1}$, $u\in U[0]$, $y\in Y_{3}[10]$.
\par\noindent\textit{Additional assumptions:} No vertex outside the core lies in $U$, $F_{1}$, $Y_{1}$, or $Y_{4}$. Every vertex outside the core in $R_i$ for some $i\in[5]$, or in $Z$, is adjacent to $u$.
\par\smallskip
\begingroup\centering
\scriptsize\setlength{\tabcolsep}{2.2pt}\renewcommand{\arraystretch}{0.80}
\textit{Admissible profiles}\\[1pt]
\begin{tabular*}{0.98\textwidth}{@{\extracolsep{\fill}}c l | c l | c l@{}}
\toprule
cycle type & strings & cycle type & strings & cycle type & strings\\
\midrule
$Z$ & $[011]$, $[110]$, $[111]$ & $F_{2}$ & $[000]$ & $F_{5}$ & $[001]$\\
$Y_{2}$ & $[001]$ & $Y_{3}$ & $[100]$ & $Y_{5}$ & $[000]$\\
$R_{1}$ & $[011]$, $[110]$ & $R_{2}$ & $[110]$ & $R_{3}$ & $[011]$, $[111]$\\
$R_{4}$ & $[011]$, $[110]$ & $R_{5}$ & $[110]$, $[111]$ &  & \\
\bottomrule
\end{tabular*}
\par\endgroup
\smallskip
\begingroup\centering
\scriptsize\setlength{\tabcolsep}{2.2pt}\renewcommand{\arraystretch}{0.76}
\textit{Elimination rounds}\\[1pt]
\begin{tabular*}{0.98\textwidth}{@{\extracolsep{\fill}}c l l | c l l@{}}
\toprule
round & profile & comparison & round & profile & comparison\\
\midrule
1 & $Z[011]$ & $v\preceq c_{1}$ & 1 & $F_{2}[000]$ & $v\preceq u$\\
1 & $Y_{5}[000]$ & $v\preceq u$ & 1 & $R_{1}[011]$ & $c_{1}\preceq v$\\
1 & $R_{3}[011]$ & $v\preceq c_{3}$ & 1 & $R_{4}[110]$ & $c_{4}\preceq v$\\
1 & $R_{5}[111]$ & $v\preceq c_{5}$ & 2 & $R_{2}[110]$ & $c_{2}\preceq v$\\
2 & $R_{3}[111]$ & $c_{3}\preceq v$ &  &  & \\
\bottomrule
\end{tabular*}
\par\endgroup
\smallskip\noindent\textit{Survivors:} $Y_{3}[100]$, $Z[110]$, $R_{5}[110]$, $R_{1}[110]$, $F_{5}[001]$, $Y_{2}[001]$, $R_{4}[011]$, $Z[111]$.
\end{certblock}

\begin{certblock}{A.3}\phantomsection\label{app:A3}\label{app:B3}
\noindent\textit{Core:} $K'=(C;\,f,\,u,\,y,\,r)$ with $f\in F_{1}$, $u\in U[0]$, $y\in Y_{3}[10]$, $r\in R_{1}[110]$.
\par\noindent\textit{Additional assumptions:} No vertex outside the core lies in $U$, $F_{1}$, $Y_{1}$, or $Y_{4}$. Every vertex outside the core in $R_i$ for some $i\in[5]$, or in $Z$, is adjacent to $u$.
\par\smallskip
\begingroup\centering
\scriptsize\setlength{\tabcolsep}{2.2pt}\renewcommand{\arraystretch}{0.80}
\textit{Admissible profiles}\\[1pt]
\begin{tabular*}{0.98\textwidth}{@{\extracolsep{\fill}}c l | c l | c l@{}}
\toprule
cycle type & strings & cycle type & strings & cycle type & strings\\
\midrule
$Z$ & $[1100]$ & $F_{2}$ & $[0001]$ & $F_{5}$ & $[0011]$\\
$Y_{2}$ & $[0011]$ & $Y_{3}$ & $[1000]$ & $Y_{5}$ & $[0000]$\\
$R_{1}$ & $[0110]$, $[1100]$ & $R_{2}$ & $[1101]$ & $R_{3}$ & $[0110]$, $[1110]$\\
$R_{4}$ & $[1100]$ & $R_{5}$ & $[1111]$ &  & \\
\bottomrule
\end{tabular*}
\par\endgroup
\smallskip
\begingroup\centering
\scriptsize\setlength{\tabcolsep}{2.2pt}\renewcommand{\arraystretch}{0.76}
\textit{Elimination rounds}\\[1pt]
\begin{tabular*}{0.98\textwidth}{@{\extracolsep{\fill}}c l l | c l l@{}}
\toprule
round & profile & comparison & round & profile & comparison\\
\midrule
1 & $Z[1100]$ & $v\preceq c_{4}$ & 1 & $F_{2}[0001]$ & $v\preceq u$\\
1 & $Y_{2}[0011]$ & $v\preceq f$ & 1 & $Y_{3}[1000]$ & $y\preceq v$\\
1 & $Y_{5}[0000]$ & $v\preceq f$ & 1 & $R_{1}[0110]$ & $c_{1}\preceq v$\\
1 & $R_{1}[1100]$ & $r\preceq v$ & 1 & $R_{3}[0110]$ & $v\preceq c_{3}$\\
1 & $R_{4}[1100]$ & $c_{4}\preceq v$ & 1 & $R_{5}[1111]$ & $v\preceq c_{5}$\\
2 & $R_{2}[1101]$ & $c_{2}\preceq v$ & 2 & $R_{3}[1110]$ & $c_{3}\preceq v$\\
\bottomrule
\end{tabular*}
\par\endgroup
\smallskip\noindent\textit{Survivors:} $F_{5}[0011]$.
\end{certblock}

\begin{certblock}{A.4}\phantomsection\label{app:A4}\label{app:B4}
\noindent\textit{Core:} $K''=(C;\,f,\,u,\,y,\,g')$ with $f\in F_{1}$, $u\in U[0]$, $y\in Y_{3}[10]$, $g'\in F_{5}[001]$.
\par\noindent\textit{Additional assumptions:} No vertex outside the core lies in $U$, $F_{1}$, $Y_{1}$, $Y_{4}$, or $R_{1}$. Every vertex outside the core in $R_i$ for some $i\in[5]$, or in $Z$, is adjacent to $u$.
\par\smallskip
\begingroup\centering
\scriptsize\setlength{\tabcolsep}{2.2pt}\renewcommand{\arraystretch}{0.80}
\textit{Admissible profiles}\\[1pt]
\begin{tabular*}{0.98\textwidth}{@{\extracolsep{\fill}}c l | c l | c l@{}}
\toprule
cycle type & strings & cycle type & strings & cycle type & strings\\
\midrule
$Z$ & $[0110]$, $[1101]$, $[1111]$ & $F_{5}$ & $[0010]$ & $Y_{2}$ & $[0011]$\\
$Y_{3}$ & $[1001]$ & $Y_{5}$ & $[0000]$ & $R_{2}$ & $[1101]$\\
$R_{3}$ & $[0110]$, $[1111]$ & $R_{4}$ & $[1101]$ & $R_{5}$ & $[1101]$, $[1110]$\\
\bottomrule
\end{tabular*}
\par\endgroup
\smallskip
\begingroup\centering
\scriptsize\setlength{\tabcolsep}{2.2pt}\renewcommand{\arraystretch}{0.76}
\textit{Elimination rounds}\\[1pt]
\begin{tabular*}{0.98\textwidth}{@{\extracolsep{\fill}}c l l | c l l@{}}
\toprule
round & profile & comparison & round & profile & comparison\\
\midrule
1 & $Z[0110]$ & $v\preceq c_{1}$ & 1 & $Z[1111]$ & $v\preceq c_{3}$\\
1 & $F_{5}[0010]$ & $g'\preceq v$ & 1 & $Y_{5}[0000]$ & $v\preceq f$\\
1 & $R_{2}[1101]$ & $c_{2}\preceq v$ & 1 & $R_{3}[0110]$ & $v\preceq c_{3}$\\
1 & $R_{3}[1111]$ & $c_{3}\preceq v$ & 1 & $R_{4}[1101]$ & $c_{4}\preceq v$\\
1 & $R_{5}[1110]$ & $c_{5}\preceq v$ & 2 & $Y_{3}[1001]$ & $y\preceq v$\\
3 & $Z[1101]$ & $v\preceq c_{2}$ &  &  & \\
\bottomrule
\end{tabular*}
\par\endgroup
\smallskip\noindent\textit{Survivors:} $R_{5}[1101]$, $Y_{2}[0011]$.
\end{certblock}

\begin{certblock}{A.5}\phantomsection\label{app:A5}\label{app:B5}
\noindent\textit{Core:} $K=(C;\,f,\,a,\,g)$ with $f\in F_{1}$, $a\in Y_{1}[0]$, $g\in F_{2}[00]$.
\par\noindent\textit{Additional assumptions:} No vertex outside the core lies in $U$ or $F_{1}$.
\par\smallskip
\begingroup\centering
\scriptsize\setlength{\tabcolsep}{2.2pt}\renewcommand{\arraystretch}{0.80}
\textit{Admissible profiles}\\[1pt]
\begin{tabular*}{0.98\textwidth}{@{\extracolsep{\fill}}c l | c l | c l@{}}
\toprule
cycle type & strings & cycle type & strings & cycle type & strings\\
\midrule
$Z$ & $[000]$, $[111]$ & $F_{2}$ & $[000]$ & $Y_{1}$ & $[000]$\\
$Y_{2}$ & $[010]$ & $Y_{3}$ & $[100]$, $[110]$ & $Y_{4}$ & $[101]$, $[111]$\\
$Y_{5}$ & $[011]$ & $R_{1}$ & $[011]$, $[111]$ & $R_{2}$ & $[100]$, $[101]$, $[111]$\\
$R_{3}$ & $[001]$, $[111]$ & $R_{4}$ & $[000]$, $[111]$ & $R_{5}$ & $[100]$, $[101]$, $[111]$\\
\bottomrule
\end{tabular*}
\par\endgroup
\smallskip
\begingroup\centering
\scriptsize\setlength{\tabcolsep}{2.2pt}\renewcommand{\arraystretch}{0.76}
\textit{Elimination rounds}\\[1pt]
\begin{tabular*}{0.98\textwidth}{@{\extracolsep{\fill}}c l l | c l l@{}}
\toprule
round & profile & comparison & round & profile & comparison\\
\midrule
1 & $Z[000]$ & $v\preceq c_{1}$ & 1 & $Y_{3}[110]$ & $c_{2}\preceq v$\\
1 & $R_{4}[000]$ & $v\preceq f$ & 1 & $R_{4}[111]$ & $c_{4}\preceq v$\\
2 & $R_{5}[111]$ & $c_{5}\preceq v$ &  &  & \\
\bottomrule
\end{tabular*}
\par\endgroup
\smallskip\noindent\textit{Survivors:} $Y_{1}[000]$, $F_{2}[000]$, $R_{2}[100]$, $R_{5}[100]$, $Y_{3}[100]$, $Y_{2}[010]$, $R_{3}[001]$, $R_{2}[101]$, $R_{5}[101]$, $Y_{4}[101]$, $R_{1}[011]$, $Y_{5}[011]$, $Z[111]$, $R_{2}[111]$, $R_{3}[111]$, $Y_{4}[111]$, $R_{1}[111]$.
\par\noindent\textit{Core comparison:} $a\preceq g$; the survivors containing $a$ but not $g$ are $Y_{2}[010]$.
\par\noindent\textit{Obstructions used:} \hyperlink{obs:B1}{B.1}.
\end{certblock}

\begin{certblock}{A.6}\phantomsection\label{app:A6}\label{app:B6}
\noindent\textit{Core:} $K=(C;\,f,\,a,\,g,\,x)$ with $f\in F_{1}$, $a\in Y_{1}[0]$, $g\in F_{2}[00]$, $x\in Y_{2}[010]$.
\par\noindent\textit{Additional assumptions:} No vertex outside the core lies in $U$ or $F_{1}$.
\par\smallskip
\begingroup\centering
\scriptsize\setlength{\tabcolsep}{2.2pt}\renewcommand{\arraystretch}{0.80}
\textit{Admissible profiles}\\[1pt]
\begin{tabular*}{0.98\textwidth}{@{\extracolsep{\fill}}c l | c l | c l@{}}
\toprule
cycle type & strings & cycle type & strings & cycle type & strings\\
\midrule
$Z$ & $[0000]$, $[1111]$ & $F_{2}$ & $[0000]$ & $Y_{1}$ & $[0001]$\\
$Y_{2}$ & $[0100]$ & $Y_{3}$ & $[1001]$, $[1101]$ & $Y_{4}$ & $[1011]$, $[1110]$\\
$Y_{5}$ & $[0110]$, $[0111]$ & $R_{1}$ & $[0110]$, $[1110]$, $[1111]$ & $R_{2}$ & $[1001]$, $[1011]$, $[1111]$\\
$R_{3}$ & $[1110]$ & $R_{4}$ & $[0000]$, $[1111]$ & $R_{5}$ & $[1011]$\\
\bottomrule
\end{tabular*}
\par\endgroup
\smallskip
\begingroup\centering
\scriptsize\setlength{\tabcolsep}{2.2pt}\renewcommand{\arraystretch}{0.76}
\textit{Elimination rounds}\\[1pt]
\begin{tabular*}{0.98\textwidth}{@{\extracolsep{\fill}}c l l | c l l@{}}
\toprule
round & profile & comparison & round & profile & comparison\\
\midrule
1 & $Z[0000]$ & $v\preceq c_{1}$ & 1 & $Z[1111]$ & $v\preceq c_{4}$\\
1 & $Y_{3}[1001]$ & $c_{2}\preceq v$ & 1 & $Y_{3}[1101]$ & $c_{2}\preceq v$\\
1 & $Y_{4}[1011]$ & $c_{5}\preceq v$ & 1 & $Y_{4}[1110]$ & $c_{3}\preceq v$\\
1 & $Y_{5}[0110]$ & $c_{1}\preceq v$ & 1 & $Y_{5}[0111]$ & $c_{1}\preceq v$\\
1 & $R_{1}[1110]$ & $c_{1}\preceq v$ & 1 & $R_{1}[1111]$ & $c_{1}\preceq v$\\
1 & $R_{2}[1011]$ & $c_{2}\preceq v$ & 1 & $R_{2}[1111]$ & $c_{2}\preceq v$\\
1 & $R_{4}[0000]$ & $v\preceq f$ & 1 & $R_{4}[1111]$ & $c_{4}\preceq v$\\
2 & $F_{2}[0000]$ & $g\preceq v$ & 2 & $Y_{1}[0001]$ & $a\preceq v$\\
2 & $Y_{2}[0100]$ & $x\preceq v$ & 2 & $R_{1}[0110]$ & $c_{1}\preceq v$\\
2 & $R_{2}[1001]$ & $c_{2}\preceq v$ & 2 & $R_{3}[1110]$ & $c_{3}\preceq v$\\
2 & $R_{5}[1011]$ & $c_{5}\preceq v$ &  &  & \\
\bottomrule
\end{tabular*}
\par\endgroup
\smallskip\noindent\textit{Survivors:} None.
\par\noindent\textit{Obstructions used:} \hyperlink{obs:B2}{B.2}, \hyperlink{obs:B3}{B.3}, \hyperlink{obs:B4}{B.4}, \hyperlink{obs:B5}{B.5}, \hyperlink{obs:B6}{B.6}, \hyperlink{obs:B7}{B.7}, \hyperlink{obs:B8}{B.8}, \hyperlink{obs:B9}{B.9}, \hyperlink{obs:B10}{B.10}, \hyperlink{obs:B11}{B.11}, \hyperlink{obs:B12}{B.12}, \hyperlink{obs:B13}{B.13}.
\end{certblock}

\begin{certblock}{A.7}\phantomsection\label{app:A7}\label{app:B7}
\noindent\textit{Core:} $K=(C;\,f,\,a,\,y,\,p,\,q)$ with $f\in F_{1}$, $a\in Y_{1}[0]$, $y\in Y_{3}[11]$, $p\in Y_{5}[010]$, $q\in R_{5}[1001]$.
\par\noindent\textit{Additional assumptions:} No vertex outside the core lies in $U$, $F_{1}$, $F_{2}$, or $F_{5}$.
\par\noindent\textit{Directly excluded:} $Y_{1}[00111]$ (\hyperlink{obs:B15}{B.15}), $R_{5}[11010]$ (\hyperlink{obs:B16}{B.16}).
\par\smallskip
\begingroup\centering
\scriptsize\setlength{\tabcolsep}{2.2pt}\renewcommand{\arraystretch}{0.80}
\textit{Admissible profiles}\\[1pt]
\begin{tabular*}{0.98\textwidth}{@{\extracolsep{\fill}}c l | c l | c l@{}}
\toprule
cycle type & strings & cycle type & strings & cycle type & strings\\
\midrule
$Z$ & $[00000]$, $[11000]$ & $R_{1}$ & $[01101]$ & $Y_{1}$ & $[00110]$, $[00111]$\\
$R_{2}$ & $[10010]$, $[11000]$ & $Y_{2}$ & $[01101]$, $[01111]$ & $R_{3}$ & $[00110]$, $[11110]$\\
$Y_{3}$ & $[10010]$, $[11000]$, $[11001]$ & $R_{4}$ & $[11001]$ & $Y_{4}$ & $[10110]$\\
$R_{5}$ & $[10010]$, $[10110]$, $[11010]$ & $Y_{5}$ & $[01001]$ &  & \\
\bottomrule
\end{tabular*}
\par\endgroup
\smallskip\noindent\textit{Elimination:} None needed.

\smallskip\noindent\textit{Profiles remaining after the direct exclusions:} $R_{1}[01101]$, $R_{2}[10010]$, $R_{2}[11000]$, $R_{3}[00110]$, $R_{3}[11110]$, $R_{4}[11001]$, $R_{5}[10010]$, $R_{5}[10110]$, $Y_{1}[00110]$, $Y_{2}[01101]$, $Y_{2}[01111]$, $Y_{3}[10010]$, $Y_{3}[11000]$, $Y_{3}[11001]$, $Y_{4}[10110]$, $Y_{5}[01001]$, $Z[00000]$, $Z[11000]$.
\par\noindent\textit{Core comparison:} $q\preceq c_{5}$; no profile remaining after the direct exclusions contains $q$ but not $c_{5}$.
\par\noindent\textit{Obstructions used:} \hyperlink{obs:B15}{B.15}, \hyperlink{obs:B16}{B.16}.
\end{certblock}

\begin{certblock}{A.8}\phantomsection\label{app:A8}\label{app:B8}
\noindent\textit{Core:} $K=(C;\,f,\,a,\,y,\,p,\,q)$ with $f\in F_{1}$, $a\in Y_{1}[0]$, $y\in Y_{3}[11]$, $p\in R_{1}[110]$, $q\in Y_{1}[0001]$.
\par\noindent\textit{Additional assumptions:} No vertex outside the core lies in $U$, $F_{1}$, $F_{2}$, or $F_{5}$.
\par\noindent\textit{Directly excluded:} $R_{5}[10111]$ (\hyperlink{obs:B17}{B.17}), $Y_{4}[10101]$ (\hyperlink{obs:B18}{B.18}).
\par\smallskip
\begingroup\centering
\scriptsize\setlength{\tabcolsep}{2.2pt}\renewcommand{\arraystretch}{0.80}
\textit{Admissible profiles}\\[1pt]
\begin{tabular*}{0.98\textwidth}{@{\extracolsep{\fill}}c l | c l | c l@{}}
\toprule
cycle type & strings & cycle type & strings & cycle type & strings\\
\midrule
$Z$ & $[00000]$ & $R_{2}$ & $[10010]$, $[11010]$ & $Y_{1}$ & $[00010]$, $[00110]$\\
$R_{3}$ & $[00110]$, $[11101]$ & $Y_{2}$ & $[01101]$, $[01111]$ & $R_{4}$ & $[00000]$, $[11001]$\\
$Y_{3}$ & $[10000]$, $[11000]$, $[11001]$ & $R_{5}$ & $[10110]$, $[10111]$ & $R_{1}$ & $[01101]$, $[11001]$\\
$Y_{4}$ & $[10101]$ &  &  &  & \\
\bottomrule
\end{tabular*}
\par\endgroup
\smallskip\noindent\textit{Elimination:} None needed.

\smallskip\noindent\textit{Profiles remaining after the direct exclusions:} $R_{1}[01101]$, $R_{1}[11001]$, $R_{2}[10010]$, $R_{2}[11010]$, $R_{3}[00110]$, $R_{3}[11101]$, $R_{4}[00000]$, $R_{4}[11001]$, $R_{5}[10110]$, $Y_{1}[00010]$, $Y_{1}[00110]$, $Y_{2}[01101]$, $Y_{2}[01111]$, $Y_{3}[10000]$, $Y_{3}[11000]$, $Y_{3}[11001]$, $Z[00000]$.
\par\noindent\textit{Core comparison:} $q\preceq a$; no profile remaining after the direct exclusions contains $q$ but not $a$.
\par\noindent\textit{Obstructions used:} \hyperlink{obs:B17}{B.17}, \hyperlink{obs:B18}{B.18}.
\end{certblock}

\begin{certblock}{A.9}\phantomsection\label{app:A9}\label{app:B9}
\noindent\textit{Core:} $K=(C;\,f,\,a)$ with $f\in F_{1}$, $a\in Y_{1}[0]$.
\par\noindent\textit{Additional assumptions:} No vertex outside the core lies in $U$, $F_{1}$, $F_{2}$, $F_{5}$, $Y_{2}$, or $Y_{5}$. No vertex outside the core realizes $Y_{3}[11]$, $Y_{4}[11]$, or $R_{1}[11]$.
\par\smallskip
\begingroup\centering
\scriptsize\setlength{\tabcolsep}{2.2pt}\renewcommand{\arraystretch}{0.80}
\textit{Admissible profiles}\\[1pt]
\begin{tabular*}{0.98\textwidth}{@{\extracolsep{\fill}}c l | c l | c l@{}}
\toprule
cycle type & strings & cycle type & strings & cycle type & strings\\
\midrule
$Z$ & $[00]$, $[11]$ & $Y_{1}$ & $[00]$ & $Y_{3}$ & $[10]$\\
$Y_{4}$ & $[10]$ & $R_{1}$ & $[01]$ & $R_{2}$ & $[10]$, $[11]$\\
$R_{3}$ & $[00]$, $[11]$ & $R_{4}$ & $[00]$, $[11]$ & $R_{5}$ & $[10]$, $[11]$\\
\bottomrule
\end{tabular*}
\par\endgroup
\smallskip
\begingroup\centering
\scriptsize\setlength{\tabcolsep}{2.2pt}\renewcommand{\arraystretch}{0.76}
\textit{Elimination rounds}\\[1pt]
\begin{tabular*}{0.98\textwidth}{@{\extracolsep{\fill}}c l l | c l l@{}}
\toprule
round & profile & comparison & round & profile & comparison\\
\midrule
1 & $Z[00]$ & $v\preceq c_{1}$ & 1 & $Y_{3}[10]$ & $c_{2}\preceq v$\\
1 & $Y_{4}[10]$ & $c_{5}\preceq v$ & 1 & $R_{1}[01]$ & $c_{1}\preceq v$\\
1 & $R_{2}[11]$ & $c_{2}\preceq v$ & 1 & $R_{3}[00]$ & $v\preceq f$\\
1 & $R_{3}[11]$ & $c_{3}\preceq v$ & 1 & $R_{4}[00]$ & $v\preceq f$\\
1 & $R_{4}[11]$ & $c_{4}\preceq v$ & 1 & $R_{5}[11]$ & $c_{5}\preceq v$\\
2 & $Z[11]$ & $v\preceq c_{3}$ & 2 & $Y_{1}[00]$ & $a\preceq v$\\
2 & $R_{2}[10]$ & $c_{2}\preceq v$ & 2 & $R_{5}[10]$ & $c_{5}\preceq v$\\
\bottomrule
\end{tabular*}
\par\endgroup
\smallskip\noindent\textit{Survivors:} None.
\end{certblock}

\begin{certblock}{A.10}\phantomsection\label{app:A10}\label{app:B10}
\noindent\textit{Core:} $K=(C;\,f,\,a,\,x)$ with $f\in F_{1}$, $a\in Y_{1}[0]$, $x\in Y_{2}[01]$.
\par\noindent\textit{Additional assumptions:} No vertex outside the core lies in $U$, $F_{1}$, $F_{2}$, $F_{3}$, $F_{4}$, or $F_{5}$. No vertex outside the core realizes $Y_{3}[111]$ or $Y_{4}[110]$.
\par\smallskip
\begingroup\centering
\scriptsize\setlength{\tabcolsep}{2.2pt}\renewcommand{\arraystretch}{0.80}
\textit{Admissible profiles}\\[1pt]
\begin{tabular*}{0.98\textwidth}{@{\extracolsep{\fill}}c l | c l | c l@{}}
\toprule
cycle type & strings & cycle type & strings & cycle type & strings\\
\midrule
$Z$ & $[000]$, $[110]$, $[111]$ & $Y_{1}$ & $[001]$ & $Y_{2}$ & $[010]$\\
$Y_{3}$ & $[101]$ & $Y_{4}$ & $[101]$ & $Y_{5}$ & $[010]$, $[011]$\\
$R_{1}$ & $[010]$, $[110]$, $[111]$ & $R_{2}$ & $[101]$, $[111]$ & $R_{3}$ & $[000]$, $[110]$\\
$R_{4}$ & $[000]$, $[001]$, $[110]$, $[111]$ & $R_{5}$ & $[101]$, $[110]$ &  & \\
\bottomrule
\end{tabular*}
\par\endgroup
\smallskip
\begingroup\centering
\scriptsize\setlength{\tabcolsep}{2.2pt}\renewcommand{\arraystretch}{0.76}
\textit{Elimination rounds}\\[1pt]
\begin{tabular*}{0.98\textwidth}{@{\extracolsep{\fill}}c l l | c l l@{}}
\toprule
round & profile & comparison & round & profile & comparison\\
\midrule
1 & $Z[000]$ & $v\preceq c_{1}$ & 1 & $R_{1}[010]$ & $c_{1}\preceq v$\\
1 & $R_{4}[110]$ & $v\preceq c_{4}$ & 2 & $R_{3}[110]$ & $v\preceq c_{3}$\\
2 & $R_{4}[111]$ & $v\preceq c_{4}$ &  &  & \\
\bottomrule
\end{tabular*}
\par\endgroup
\smallskip\noindent\textit{Survivors:} $R_{3}[000]$, $R_{4}[000]$, $Y_{5}[010]$, $Y_{2}[010]$, $Z[110]$, $R_{5}[110]$, $R_{1}[110]$, $Y_{1}[001]$, $R_{4}[001]$, $R_{2}[101]$, $R_{5}[101]$, $Y_{4}[101]$, $Y_{3}[101]$, $Y_{5}[011]$, $Z[111]$, $R_{2}[111]$, $R_{1}[111]$.
\par\noindent\textit{Core comparison:} $c_{1}\preceq x$; the survivors containing $c_{1}$ but not $x$ are $R_{5}[110]$.
\end{certblock}

\begin{certblock}{A.11}\phantomsection\label{app:A11}\label{app:B11}
\noindent\textit{Core:} $K=(C;\,f,\,a,\,x,\,q')$ with $f\in F_{1}$, $a\in Y_{1}[0]$, $x\in Y_{2}[01]$, $q'\in R_{5}[110]$.
\par\noindent\textit{Additional assumptions:} No vertex outside the core lies in $U$, $F_{1}$, $F_{2}$, $F_{3}$, $F_{4}$, or $F_{5}$. No vertex outside the core realizes $Y_{3}[1110]$ or $Y_{4}[1100]$.
\par\smallskip
\begingroup\centering
\scriptsize\setlength{\tabcolsep}{2.2pt}\renewcommand{\arraystretch}{0.80}
\textit{Admissible profiles}\\[1pt]
\begin{tabular*}{0.98\textwidth}{@{\extracolsep{\fill}}c l | c l | c l@{}}
\toprule
cycle type & strings & cycle type & strings & cycle type & strings\\
\midrule
$Z$ & $[0000]$, $[1100]$ & $Y_{1}$ & $[0011]$ & $Y_{2}$ & $[0100]$\\
$Y_{3}$ & $[1011]$ & $Y_{5}$ & $[0101]$, $[0111]$ & $R_{1}$ & $[0101]$, $[1101]$, $[1111]$\\
$R_{2}$ & $[1010]$, $[1011]$, $[1110]$ & $R_{3}$ & $[1100]$ & $R_{4}$ & $[0011]$, $[1111]$\\
$R_{5}$ & $[1010]$, $[1100]$ &  &  &  & \\
\bottomrule
\end{tabular*}
\par\endgroup
\smallskip
\begingroup\centering
\scriptsize\setlength{\tabcolsep}{2.2pt}\renewcommand{\arraystretch}{0.76}
\textit{Elimination rounds}\\[1pt]
\begin{tabular*}{0.98\textwidth}{@{\extracolsep{\fill}}c l l | c l l@{}}
\toprule
round & profile & comparison & round & profile & comparison\\
\midrule
1 & $Z[0000]$ & $v\preceq c_{1}$ & 1 & $Z[1100]$ & $v\preceq c_{3}$\\
1 & $Y_{3}[1011]$ & $c_{2}\preceq v$ & 1 & $Y_{5}[0101]$ & $c_{1}\preceq v$\\
1 & $Y_{5}[0111]$ & $c_{1}\preceq v$ & 1 & $R_{1}[0101]$ & $c_{1}\preceq v$\\
1 & $R_{1}[1101]$ & $c_{1}\preceq v$ & 1 & $R_{1}[1111]$ & $c_{1}\preceq v$\\
1 & $R_{2}[1010]$ & $c_{2}\preceq v$ & 1 & $R_{2}[1011]$ & $c_{2}\preceq v$\\
1 & $R_{2}[1110]$ & $c_{2}\preceq v$ & 1 & $R_{4}[0011]$ & $v\preceq c_{4}$\\
1 & $R_{4}[1111]$ & $c_{4}\preceq v$ & 2 & $Y_{1}[0011]$ & $a\preceq v$\\
2 & $Y_{2}[0100]$ & $x\preceq v$ & 2 & $R_{3}[1100]$ & $c_{3}\preceq v$\\
2 & $R_{5}[1010]$ & $c_{5}\preceq v$ & 2 & $R_{5}[1100]$ & $q'\preceq v$\\
\bottomrule
\end{tabular*}
\par\endgroup
\smallskip\noindent\textit{Survivors:} None.
\par\noindent\textit{Obstructions used:} \hyperlink{obs:B19}{B.19}, \hyperlink{obs:B20}{B.20}, \hyperlink{obs:B21}{B.21}, \hyperlink{obs:B22}{B.22}, \hyperlink{obs:B23}{B.23}.
\end{certblock}

\begin{certblock}{A.12}\phantomsection\label{app:A12}\label{app:B12}
\noindent\textit{Core:} $K=(C;\,f,\,a,\,r)$ with $f\in F_{1}$, $a\in Y_{1}[0]$, $r\in R_{1}[11]$.
\par\noindent\textit{Additional assumptions:} No vertex outside the core lies in $U$, $F_{1}$, $F_{2}$, $F_{3}$, $F_{4}$, or $F_{5}$. No vertex outside the core realizes $Y_{3}[110]$ or $Y_{4}[110]$.
\par\smallskip
\begingroup\centering
\scriptsize\setlength{\tabcolsep}{2.2pt}\renewcommand{\arraystretch}{0.80}
\textit{Admissible profiles}\\[1pt]
\begin{tabular*}{0.98\textwidth}{@{\extracolsep{\fill}}c l | c l | c l@{}}
\toprule
cycle type & strings & cycle type & strings & cycle type & strings\\
\midrule
$Z$ & $[000]$, $[110]$ & $Y_{1}$ & $[001]$ & $Y_{2}$ & $[010]$, $[011]$\\
$Y_{3}$ & $[100]$ & $Y_{4}$ & $[100]$ & $Y_{5}$ & $[010]$, $[011]$\\
$R_{1}$ & $[010]$, $[110]$ & $R_{2}$ & $[101]$, $[111]$ & $R_{3}$ & $[000]$, $[001]$, $[110]$\\
$R_{4}$ & $[000]$, $[001]$, $[110]$ & $R_{5}$ & $[101]$, $[111]$ &  & \\
\bottomrule
\end{tabular*}
\par\endgroup
\smallskip
\begingroup\centering
\scriptsize\setlength{\tabcolsep}{2.2pt}\renewcommand{\arraystretch}{0.76}
\textit{Elimination rounds}\\[1pt]
\begin{tabular*}{0.98\textwidth}{@{\extracolsep{\fill}}c l l | c l l@{}}
\toprule
round & profile & comparison & round & profile & comparison\\
\midrule
1 & $Z[000]$ & $v\preceq c_{1}$ & 1 & $R_{1}[010]$ & $c_{1}\preceq v$\\
1 & $R_{2}[101]$ & $c_{2}\preceq v$ & 1 & $R_{5}[101]$ & $c_{5}\preceq v$\\
2 & $R_{3}[110]$ & $v\preceq c_{3}$ & 2 & $R_{4}[110]$ & $v\preceq c_{4}$\\
\bottomrule
\end{tabular*}
\par\endgroup
\smallskip\noindent\textit{Survivors:} $R_{3}[000]$, $R_{4}[000]$, $Y_{4}[100]$, $Y_{3}[100]$, $Y_{5}[010]$, $Y_{2}[010]$, $Z[110]$, $R_{1}[110]$, $R_{3}[001]$, $Y_{1}[001]$, $R_{4}[001]$, $Y_{5}[011]$, $Y_{2}[011]$, $R_{2}[111]$, $R_{5}[111]$.
\par\noindent\textit{Core comparison:} $c_{1}\preceq r$; the survivors containing $c_{1}$ but not $r$ are $Y_{4}[100]$, $Y_{3}[100]$.
\end{certblock}

\begin{certblock}{A.13}\phantomsection\label{app:A13}\label{app:B13}
\noindent\textit{Core:} $K=(C;\,f,\,a,\,r,\,y)$ with $f\in F_{1}$, $a\in Y_{1}[0]$, $r\in R_{1}[11]$, $y\in Y_{3}[100]$.
\par\noindent\textit{Additional assumptions:} No vertex outside the core lies in $U$, $F_{1}$, $F_{2}$, $F_{3}$, $F_{4}$, or $F_{5}$. No vertex outside the core realizes $Y_{3}[1100]$ or $Y_{4}[1101]$.
\par\smallskip
\begingroup\centering
\scriptsize\setlength{\tabcolsep}{2.2pt}\renewcommand{\arraystretch}{0.80}
\textit{Admissible profiles}\\[1pt]
\begin{tabular*}{0.98\textwidth}{@{\extracolsep{\fill}}c l | c l | c l@{}}
\toprule
cycle type & strings & cycle type & strings & cycle type & strings\\
\midrule
$Z$ & $[0000]$ & $Y_{1}$ & $[0010]$, $[0011]$ & $Y_{2}$ & $[0101]$, $[0111]$\\
$Y_{3}$ & $[1000]$ & $Y_{5}$ & $[0101]$, $[0111]$ & $R_{1}$ & $[0101]$, $[1100]$\\
$R_{2}$ & $[1010]$ & $R_{3}$ & $[0011]$, $[1101]$ & $R_{4}$ & $[0000]$, $[1100]$\\
$R_{5}$ & $[1011]$, $[1111]$ &  &  &  & \\
\bottomrule
\end{tabular*}
\par\endgroup
\smallskip
\begingroup\centering
\scriptsize\setlength{\tabcolsep}{2.2pt}\renewcommand{\arraystretch}{0.76}
\textit{Elimination rounds}\\[1pt]
\begin{tabular*}{0.98\textwidth}{@{\extracolsep{\fill}}c l l | c l l@{}}
\toprule
round & profile & comparison & round & profile & comparison\\
\midrule
1 & $Z[0000]$ & $v\preceq c_{1}$ & 1 & $Y_{1}[0011]$ & $a\preceq v$\\
1 & $Y_{2}[0101]$ & $c_{1}\preceq v$ & 1 & $Y_{2}[0111]$ & $c_{1}\preceq v$\\
1 & $Y_{5}[0101]$ & $c_{1}\preceq v$ & 1 & $Y_{5}[0111]$ & $c_{1}\preceq v$\\
1 & $R_{1}[0101]$ & $c_{1}\preceq v$ & 1 & $R_{2}[1010]$ & $c_{2}\preceq v$\\
1 & $R_{3}[0011]$ & $v\preceq f$ & 1 & $R_{4}[0000]$ & $v\preceq c_{4}$\\
1 & $R_{5}[1011]$ & $c_{5}\preceq v$ & 1 & $R_{5}[1111]$ & $c_{5}\preceq v$\\
2 & $Y_{1}[0010]$ & $a\preceq v$ & 2 & $Y_{3}[1000]$ & $y\preceq v$\\
2 & $R_{1}[1100]$ & $r\preceq v$ & 2 & $R_{3}[1101]$ & $c_{3}\preceq v$\\
2 & $R_{4}[1100]$ & $c_{4}\preceq v$ &  &  & \\
\bottomrule
\end{tabular*}
\par\endgroup
\smallskip\noindent\textit{Survivors:} None.
\par\noindent\textit{Obstructions used:} \hyperlink{obs:B24}{B.24}, \hyperlink{obs:B25}{B.25}, \hyperlink{obs:B26}{B.26}, \hyperlink{obs:B27}{B.27}.
\end{certblock}

\begin{certblock}{A.14}\phantomsection\label{app:A14}\label{app:B14}
\noindent\textit{Core:} $K=(C;\,f,\,g)$ with $f\in F_{1}$, $g\in F_{2}[0]$.
\par\noindent\textit{Additional assumptions:} No vertex outside the core lies in $U$, $F_{1}$, $F_{3}$, $F_{4}$, $F_{5}$, $Y_{1}$, $Y_{3}$, $Y_{4}$, or $Y_{5}$. No vertex outside the core realizes $R_{1}[11]$ or $R_{2}[11]$.
\par\smallskip
\begingroup\centering
\scriptsize\setlength{\tabcolsep}{2.2pt}\renewcommand{\arraystretch}{0.80}
\textit{Admissible profiles}\\[1pt]
\begin{tabular*}{0.98\textwidth}{@{\extracolsep{\fill}}c l | c l | c l@{}}
\toprule
cycle type & strings & cycle type & strings & cycle type & strings\\
\midrule
$Z$ & $[00]$, $[11]$ & $F_{2}$ & $[00]$ & $Y_{2}$ & $[00]$\\
$R_{1}$ & $[01]$ & $R_{2}$ & $[10]$ & $R_{3}$ & $[01]$, $[11]$\\
$R_{4}$ & $[00]$, $[11]$ & $R_{5}$ & $[10]$, $[11]$ &  & \\
\bottomrule
\end{tabular*}
\par\endgroup
\smallskip
\begingroup\centering
\scriptsize\setlength{\tabcolsep}{2.2pt}\renewcommand{\arraystretch}{0.76}
\textit{Elimination rounds}\\[1pt]
\begin{tabular*}{0.98\textwidth}{@{\extracolsep{\fill}}c l l | c l l@{}}
\toprule
round & profile & comparison & round & profile & comparison\\
\midrule
1 & $Z[00]$ & $v\preceq c_{1}$ & 1 & $Z[11]$ & $v\preceq c_{4}$\\
1 & $Y_{2}[00]$ & $v\preceq f$ & 1 & $R_{1}[01]$ & $c_{1}\preceq v$\\
1 & $R_{2}[10]$ & $c_{2}\preceq v$ & 1 & $R_{3}[11]$ & $c_{3}\preceq v$\\
1 & $R_{4}[00]$ & $v\preceq c_{4}$ & 1 & $R_{4}[11]$ & $c_{4}\preceq v$\\
2 & $R_{5}[11]$ & $v\preceq c_{5}$ & 3 & $F_{2}[00]$ & $g\preceq v$\\
\bottomrule
\end{tabular*}
\par\endgroup
\smallskip\noindent\textit{Survivors:} $R_{5}[10]$, $R_{3}[01]$.
\end{certblock}

\begin{certblock}{A.15}\phantomsection\label{app:A15}\label{app:B15}
\noindent\textit{Core:} $K=(C;\,f,\,g,\,r)$ with $f\in F_{1}$, $g\in F_{2}[0]$, $r\in R_{1}[11]$.
\par\noindent\textit{Additional assumptions:} No vertex outside the core lies in $U$, $F_{1}$, $F_{3}$, $F_{4}$, $F_{5}$, or $Y_{1}$.
\par\smallskip
\begingroup\centering
\scriptsize\setlength{\tabcolsep}{2.2pt}\renewcommand{\arraystretch}{0.80}
\textit{Admissible profiles}\\[1pt]
\begin{tabular*}{0.98\textwidth}{@{\extracolsep{\fill}}c l | c l | c l@{}}
\toprule
cycle type & strings & cycle type & strings & cycle type & strings\\
\midrule
$Z$ & $[000]$, $[110]$ & $F_{2}$ & $[001]$ & $Y_{2}$ & $[000]$, $[001]$\\
$Y_{3}$ & $[100]$ & $Y_{4}$ & $[110]$ & $Y_{5}$ & $[010]$\\
$R_{1}$ & $[010]$, $[110]$ & $R_{2}$ & $[101]$, $[111]$ & $R_{3}$ & $[010]$, $[011]$, $[110]$\\
$R_{4}$ & $[000]$, $[001]$, $[110]$ & $R_{5}$ & $[101]$, $[111]$ &  & \\
\bottomrule
\end{tabular*}
\par\endgroup
\smallskip
\begingroup\centering
\scriptsize\setlength{\tabcolsep}{2.2pt}\renewcommand{\arraystretch}{0.76}
\textit{Elimination rounds}\\[1pt]
\begin{tabular*}{0.98\textwidth}{@{\extracolsep{\fill}}c l l | c l l@{}}
\toprule
round & profile & comparison & round & profile & comparison\\
\midrule
1 & $Z[000]$ & $v\preceq c_{1}$ & 1 & $Z[110]$ & $v\preceq c_{5}$\\
1 & $Y_{2}[000]$ & $v\preceq f$ & 1 & $Y_{4}[110]$ & $c_{3}\preceq v$\\
1 & $R_{2}[101]$ & $c_{2}\preceq v$ & 1 & $R_{4}[110]$ & $c_{4}\preceq v$\\
1 & $R_{5}[111]$ & $v\preceq c_{5}$ & 2 & $R_{1}[010]$ & $c_{1}\preceq v$\\
\bottomrule
\end{tabular*}
\par\endgroup
\smallskip\noindent\textit{Survivors:} $R_{4}[000]$, $Y_{3}[100]$, $R_{3}[010]$, $Y_{5}[010]$, $R_{3}[110]$, $R_{1}[110]$, $R_{4}[001]$, $Y_{2}[001]$, $F_{2}[001]$, $R_{5}[101]$, $R_{3}[011]$, $R_{2}[111]$.
\par\noindent\textit{Core comparison:} $c_{1}\preceq r$; the survivors containing $c_{1}$ but not $r$ are $Y_{3}[100]$.
\end{certblock}

\begin{certblock}{A.16}\phantomsection\label{app:A16}\label{app:B16}
\noindent\textit{Core:} $K=(C;\,f,\,g,\,y)$ with $f\in F_{1}$, $g\in F_{2}[0]$, $y\in Y_{3}[10]$.
\par\noindent\textit{Additional assumptions:} No vertex outside the core lies in $U$, $F_{1}$, $F_{3}$, $F_{4}$, $F_{5}$, or $Y_{1}$.
\par\smallskip
\begingroup\centering
\scriptsize\setlength{\tabcolsep}{2.2pt}\renewcommand{\arraystretch}{0.80}
\textit{Admissible profiles}\\[1pt]
\begin{tabular*}{0.98\textwidth}{@{\extracolsep{\fill}}c l | c l | c l@{}}
\toprule
cycle type & strings & cycle type & strings & cycle type & strings\\
\midrule
$Z$ & $[000]$, $[110]$, $[111]$ & $F_{2}$ & $[000]$ & $Y_{2}$ & $[001]$\\
$Y_{3}$ & $[100]$ & $Y_{4}$ & $[111]$ & $Y_{5}$ & $[011]$\\
$R_{1}$ & $[011]$, $[110]$ & $R_{2}$ & $[100]$, $[110]$ & $R_{3}$ & $[011]$, $[111]$\\
$R_{4}$ & $[000]$, $[110]$ & $R_{5}$ & $[100]$, $[101]$, $[110]$, $[111]$ &  & \\
\bottomrule
\end{tabular*}
\par\endgroup
\smallskip
\begingroup\centering
\scriptsize\setlength{\tabcolsep}{2.2pt}\renewcommand{\arraystretch}{0.76}
\textit{Elimination rounds}\\[1pt]
\begin{tabular*}{0.98\textwidth}{@{\extracolsep{\fill}}c l l | c l l@{}}
\toprule
round & profile & comparison & round & profile & comparison\\
\midrule
1 & $Z[000]$ & $v\preceq c_{1}$ & 1 & $Z[111]$ & $v\preceq c_{5}$\\
1 & $Y_{2}[001]$ & $v\preceq f$ & 1 & $R_{4}[110]$ & $c_{4}\preceq v$\\
1 & $R_{5}[110]$ & $v\preceq c_{5}$ &  &  & \\
\bottomrule
\end{tabular*}
\par\endgroup
\smallskip\noindent\textit{Survivors:} $R_{4}[000]$, $F_{2}[000]$, $R_{2}[100]$, $R_{5}[100]$, $Y_{3}[100]$, $Z[110]$, $R_{2}[110]$, $R_{1}[110]$, $R_{5}[101]$, $R_{3}[011]$, $R_{1}[011]$, $Y_{5}[011]$, $R_{5}[111]$, $R_{3}[111]$, $Y_{4}[111]$.
\par\noindent\textit{Core comparison:} $c_{2}\preceq y$; the survivors containing $c_{2}$ but not $y$ are $R_{1}[110]$.
\end{certblock}

\begin{certblock}{A.17}\phantomsection\label{app:A17}\label{app:B17}
\noindent\textit{Core:} $K=(C;\,f,\,g,\,r,\,y)$ with $f\in F_{1}$, $g\in F_{2}[0]$, $r\in R_{1}[11]$, $y\in Y_{3}[100]$.
\par\noindent\textit{Additional assumptions:} No vertex outside the core lies in $U$, $F_{1}$, $F_{3}$, $F_{4}$, $F_{5}$, or $Y_{1}$.
\par\noindent\textit{Directly excluded:} $R_{3}[0111]$ (\hyperlink{obs:B28}{B.28}).
\par\smallskip
\begingroup\centering
\scriptsize\setlength{\tabcolsep}{2.2pt}\renewcommand{\arraystretch}{0.80}
\textit{Admissible profiles}\\[1pt]
\begin{tabular*}{0.98\textwidth}{@{\extracolsep{\fill}}c l | c l | c l@{}}
\toprule
cycle type & strings & cycle type & strings & cycle type & strings\\
\midrule
$Z$ & $[0000]$, $[1100]$ & $F_{2}$ & $[0010]$ & $Y_{2}$ & $[0011]$\\
$Y_{3}$ & $[1000]$ & $Y_{4}$ & $[1101]$ & $Y_{5}$ & $[0101]$\\
$R_{1}$ & $[0101]$, $[1100]$ & $R_{2}$ & $[1010]$, $[1110]$ & $R_{3}$ & $[0101]$, $[0111]$, $[1101]$\\
$R_{4}$ & $[0000]$, $[1100]$ & $R_{5}$ & $[1011]$, $[1111]$ &  & \\
\bottomrule
\end{tabular*}
\par\endgroup
\smallskip
\begingroup\centering
\scriptsize\setlength{\tabcolsep}{2.2pt}\renewcommand{\arraystretch}{0.76}
\textit{Elimination rounds}\\[1pt]
\begin{tabular*}{0.98\textwidth}{@{\extracolsep{\fill}}c l l | c l l@{}}
\toprule
round & profile & comparison & round & profile & comparison\\
\midrule
1 & $Z[0000]$ & $v\preceq c_{1}$ & 1 & $Z[1100]$ & $v\preceq c_{5}$\\
1 & $Y_{2}[0011]$ & $v\preceq f$ & 1 & $Y_{3}[1000]$ & $y\preceq v$\\
1 & $Y_{4}[1101]$ & $c_{3}\preceq v$ & 1 & $Y_{5}[0101]$ & $c_{1}\preceq v$\\
1 & $R_{1}[1100]$ & $r\preceq v$ & 1 & $R_{2}[1010]$ & $c_{2}\preceq v$\\
1 & $R_{2}[1110]$ & $c_{2}\preceq v$ & 1 & $R_{3}[0101]$ & $v\preceq c_{3}$\\
1 & $R_{4}[0000]$ & $v\preceq f$ & 1 & $R_{4}[1100]$ & $c_{4}\preceq v$\\
1 & $R_{5}[1011]$ & $v\preceq c_{5}$ & 1 & $R_{5}[1111]$ & $v\preceq c_{5}$\\
2 & $F_{2}[0010]$ & $g\preceq v$ & 2 & $R_{1}[0101]$ & $c_{1}\preceq v$\\
2 & $R_{3}[1101]$ & $c_{3}\preceq v$ &  &  & \\
\bottomrule
\end{tabular*}
\par\endgroup
\smallskip\noindent\textit{Survivors:} None.
\par\noindent\textit{Obstructions used:} \hyperlink{obs:B28}{B.28}, \hyperlink{obs:B29}{B.29}, \hyperlink{obs:B30}{B.30}.
\end{certblock}

\begin{certblock}{A.18}\phantomsection\label{app:A18}\label{app:B18}
\noindent\textit{Core:} $K=(C;\,f,\,g,\,b)$ with $f\in F_{1}$, $g\in F_{2}[0]$, $b\in Y_{5}[01]$.
\par\noindent\textit{Additional assumptions:} No vertex outside the core lies in $U$, $F_{1}$, $F_{3}$, $F_{4}$, $F_{5}$, or $Y_{1}$.
\par\smallskip
\begingroup\centering
\scriptsize\setlength{\tabcolsep}{2.2pt}\renewcommand{\arraystretch}{0.80}
\textit{Admissible profiles}\\[1pt]
\begin{tabular*}{0.98\textwidth}{@{\extracolsep{\fill}}c l | c l | c l@{}}
\toprule
cycle type & strings & cycle type & strings & cycle type & strings\\
\midrule
$Z$ & $[000]$, $[110]$, $[111]$ & $F_{2}$ & $[001]$ & $Y_{2}$ & $[000]$, $[001]$\\
$Y_{3}$ & $[101]$ & $Y_{4}$ & $[111]$ & $Y_{5}$ & $[010]$\\
$R_{1}$ & $[010]$, $[110]$ & $R_{2}$ & $[101]$, $[110]$ & $R_{3}$ & $[010]$, $[011]$, $[110]$, $[111]$\\
$R_{4}$ & $[000]$, $[110]$ & $R_{5}$ & $[101]$, $[111]$ &  & \\
\bottomrule
\end{tabular*}
\par\endgroup
\smallskip
\begingroup\centering
\scriptsize\setlength{\tabcolsep}{2.2pt}\renewcommand{\arraystretch}{0.76}
\textit{Elimination rounds}\\[1pt]
\begin{tabular*}{0.98\textwidth}{@{\extracolsep{\fill}}c l l | c l l@{}}
\toprule
round & profile & comparison & round & profile & comparison\\
\midrule
1 & $Z[000]$ & $v\preceq c_{1}$ & 1 & $Z[111]$ & $v\preceq c_{5}$\\
1 & $R_{4}[110]$ & $c_{4}\preceq v$ &  &  & \\
\bottomrule
\end{tabular*}
\par\endgroup
\smallskip\noindent\textit{Survivors:} $R_{4}[000]$, $Y_{2}[000]$, $R_{3}[010]$, $R_{1}[010]$, $Y_{5}[010]$, $Z[110]$, $R_{2}[110]$, $R_{3}[110]$, $R_{1}[110]$, $Y_{2}[001]$, $F_{2}[001]$, $R_{2}[101]$, $R_{5}[101]$, $Y_{3}[101]$, $R_{3}[011]$, $R_{5}[111]$, $R_{3}[111]$, $Y_{4}[111]$.
\par\noindent\textit{Core comparison:} $c_{1}\preceq b$; the survivors containing $c_{1}$ but not $b$ are $R_{2}[110]$.
\end{certblock}

\begin{certblock}{A.19}\phantomsection\label{app:A19}\label{app:B19}
\noindent\textit{Core:} $K=(C;\,f,\,g,\,q)$ with $f\in F_{1}$, $g\in F_{2}[0]$, $q\in R_{2}[11]$.
\par\noindent\textit{Additional assumptions:} No vertex outside the core lies in $U$, $F_{1}$, $F_{3}$, $F_{4}$, $F_{5}$, or $Y_{1}$.
\par\smallskip
\begingroup\centering
\scriptsize\setlength{\tabcolsep}{2.2pt}\renewcommand{\arraystretch}{0.80}
\textit{Admissible profiles}\\[1pt]
\begin{tabular*}{0.98\textwidth}{@{\extracolsep{\fill}}c l | c l | c l@{}}
\toprule
cycle type & strings & cycle type & strings & cycle type & strings\\
\midrule
$Z$ & $[000]$, $[110]$ & $F_{2}$ & $[000]$, $[001]$ & $Y_{2}$ & $[001]$\\
$Y_{3}$ & $[100]$ & $Y_{4}$ & $[110]$ & $Y_{5}$ & $[010]$\\
$R_{1}$ & $[011]$, $[111]$ & $R_{2}$ & $[100]$, $[110]$ & $R_{3}$ & $[011]$, $[111]$\\
$R_{4}$ & $[000]$, $[001]$, $[110]$ & $R_{5}$ & $[100]$, $[101]$, $[110]$ &  & \\
\bottomrule
\end{tabular*}
\par\endgroup
\smallskip
\begingroup\centering
\scriptsize\setlength{\tabcolsep}{2.2pt}\renewcommand{\arraystretch}{0.76}
\textit{Elimination rounds}\\[1pt]
\begin{tabular*}{0.98\textwidth}{@{\extracolsep{\fill}}c l l | c l l@{}}
\toprule
round & profile & comparison & round & profile & comparison\\
\midrule
1 & $Z[000]$ & $v\preceq c_{1}$ & 1 & $Z[110]$ & $v\preceq c_{5}$\\
1 & $Y_{2}[001]$ & $v\preceq f$ & 1 & $R_{1}[011]$ & $c_{1}\preceq v$\\
1 & $R_{3}[111]$ & $c_{3}\preceq v$ & 1 & $R_{4}[110]$ & $c_{4}\preceq v$\\
2 & $Y_{4}[110]$ & $c_{5}\preceq v$ & 3 & $R_{2}[100]$ & $c_{2}\preceq v$\\
4 & $R_{5}[110]$ & $v\preceq c_{5}$ &  &  & \\
\bottomrule
\end{tabular*}
\par\endgroup
\smallskip\noindent\textit{Survivors:} $R_{4}[000]$, $F_{2}[000]$, $R_{5}[100]$, $Y_{3}[100]$, $Y_{5}[010]$, $R_{2}[110]$, $R_{4}[001]$, $F_{2}[001]$, $R_{5}[101]$, $R_{3}[011]$, $R_{1}[111]$.
\par\noindent\textit{Core comparison:} $c_{2}\preceq q$; the survivors containing $c_{2}$ but not $q$ are $Y_{5}[010]$.
\par\noindent\textit{Obstructions used:} \hyperlink{obs:B31}{B.31}.
\end{certblock}

\begin{certblock}{A.20}\phantomsection\label{app:A20}\label{app:B20}
\noindent\textit{Core:} $K=(C;\,f,\,g,\,q,\,b)$ with $f\in F_{1}$, $g\in F_{2}[0]$, $q\in R_{2}[11]$, $b\in Y_{5}[010]$.
\par\noindent\textit{Additional assumptions:} No vertex outside the core lies in $U$, $F_{1}$, $F_{3}$, $F_{4}$, $F_{5}$, or $Y_{1}$.
\par\noindent\textit{Directly excluded:} $Y_{2}[0011]$ (\hyperlink{obs:B32}{B.32}), $R_{5}[1011]$ (\hyperlink{obs:B33}{B.33}).
\par\smallskip
\begingroup\centering
\scriptsize\setlength{\tabcolsep}{2.2pt}\renewcommand{\arraystretch}{0.80}
\textit{Admissible profiles}\\[1pt]
\begin{tabular*}{0.98\textwidth}{@{\extracolsep{\fill}}c l | c l | c l@{}}
\toprule
cycle type & strings & cycle type & strings & cycle type & strings\\
\midrule
$Z$ & $[0000]$, $[1100]$ & $F_{2}$ & $[0011]$ & $Y_{2}$ & $[0010]$, $[0011]$\\
$Y_{3}$ & $[1001]$ & $Y_{4}$ & $[1101]$ & $Y_{5}$ & $[0100]$\\
$R_{1}$ & $[0110]$, $[1110]$ & $R_{2}$ & $[1001]$, $[1100]$ & $R_{3}$ & $[0111]$, $[1111]$\\
$R_{4}$ & $[0000]$, $[1100]$ & $R_{5}$ & $[1001]$, $[1011]$, $[1101]$ &  & \\
\bottomrule
\end{tabular*}
\par\endgroup
\smallskip
\begingroup\centering
\scriptsize\setlength{\tabcolsep}{2.2pt}\renewcommand{\arraystretch}{0.76}
\textit{Elimination rounds}\\[1pt]
\begin{tabular*}{0.98\textwidth}{@{\extracolsep{\fill}}c l l | c l l@{}}
\toprule
round & profile & comparison & round & profile & comparison\\
\midrule
1 & $Z[0000]$ & $v\preceq c_{1}$ & 1 & $Z[1100]$ & $v\preceq c_{3}$\\
1 & $Y_{2}[0010]$ & $v\preceq f$ & 1 & $Y_{3}[1001]$ & $c_{2}\preceq v$\\
1 & $Y_{4}[1101]$ & $c_{5}\preceq v$ & 1 & $Y_{5}[0100]$ & $b\preceq v$\\
1 & $R_{1}[0110]$ & $c_{1}\preceq v$ & 1 & $R_{1}[1110]$ & $c_{1}\preceq v$\\
1 & $R_{2}[1100]$ & $q\preceq v$ & 1 & $R_{3}[0111]$ & $v\preceq c_{3}$\\
1 & $R_{3}[1111]$ & $c_{3}\preceq v$ & 1 & $R_{4}[0000]$ & $v\preceq g$\\
1 & $R_{4}[1100]$ & $c_{4}\preceq v$ & 1 & $R_{5}[1001]$ & $v\preceq c_{5}$\\
2 & $F_{2}[0011]$ & $v\preceq g$ & 2 & $R_{2}[1001]$ & $c_{2}\preceq v$\\
3 & $R_{5}[1101]$ & $c_{5}\preceq v$ &  &  & \\
\bottomrule
\end{tabular*}
\par\endgroup
\smallskip\noindent\textit{Survivors:} None.
\par\noindent\textit{Obstructions used:} \hyperlink{obs:B32}{B.32}, \hyperlink{obs:B33}{B.33}, \hyperlink{obs:B34}{B.34}, \hyperlink{obs:B35}{B.35}, \hyperlink{obs:B36}{B.36}.
\end{certblock}

\section{Explicit obstruction witnesses}\label{app:obstruction-witnesses}

The table records the explicit induced copies of $B_{10}$ and $B_{11}$ used in
the certificates. Here $v$ and $w$ realize the listed profiles, $vw$ records
their adjacency, and the ordered copy is given in defining order. 

\begin{center}
\small
\setlength{\tabcolsep}{4pt}
\renewcommand{\arraystretch}{0.96}
\begin{longtable}{@{}l l l c l l@{}}
\toprule
row & profile of $v$ & profile of $w$ & $vw$ & ordered induced copy & used to exclude\\
\midrule
\endhead
\bottomrule
\endfoot

\multicolumn{6}{@{}l}{\textit{Certificate A.5:} $K=(C;\,f,\,a,\,g)$}\\[1pt]
\hypertarget{obs:B1}{\hypertarget{obs:C1}{B.1}} & $Y_{3}[110]$ & $R_{1}[111]$ & 0 & $B_{10}\,(c_{1},\,c_{2},\,c_{3},\,c_{4},\,c_{5},\,f,\,a,\,g,\,v,\,w)$ & $Y_{3}[110]$\\

\addlinespace[1pt]
\multicolumn{6}{@{}l}{\textit{Certificate A.6:} $K=(C;\,f,\,a,\,g,\,x)$}\\[1pt]
\hypertarget{obs:B2}{\hypertarget{obs:C2}{B.2}} & $Y_{5}[0110]$ & $R_{2}[1111]$ & 0 & $B_{11}\,(c_{1},\,c_{2},\,c_{3},\,c_{4},\,c_{5},\,f,\,a,\,g,\,x,\,v,\,w)$ & $R_{2}[1111]$, $Y_{5}[0110]$\\
\hypertarget{obs:B3}{\hypertarget{obs:C3}{B.3}} & $Y_{4}[1110]$ & $R_{2}[1011]$ & 0 & $B_{11}\,(c_{2},\,x,\,c_{4},\,g,\,c_{1},\,c_{5},\,c_{3},\,v,\,w,\,a,\,f)$ & $R_{2}[1011]$, $Y_{4}[1110]$\\
\hypertarget{obs:B4}{\hypertarget{obs:C4}{B.4}} & $Y_{4}[1110]$ & $R_{2}[1111]$ & 0 & $B_{10}\,(c_{1},\,c_{5},\,f,\,c_{3},\,a,\,c_{4},\,c_{2},\,w,\,v,\,x)$ & $R_{2}[1111]$, $Y_{4}[1110]$\\
\hypertarget{obs:B5}{\hypertarget{obs:C5}{B.5}} & $Y_{4}[1110]$ & $R_{1}[1110]$ & 0 & $B_{10}\,(c_{2},\,x,\,c_{5},\,g,\,v,\,c_{4},\,w,\,c_{1},\,f,\,a)$ & $R_{1}[1110]$\\
\hypertarget{obs:B6}{\hypertarget{obs:C6}{B.6}} & $R_{1}[1110]$ & $Y_{4}[1011]$ & 0 & $B_{11}\,(c_{1},\,a,\,c_{4},\,f,\,c_{2},\,c_{3},\,c_{5},\,w,\,v,\,x,\,g)$ & $R_{1}[1110]$, $Y_{4}[1011]$\\
\hypertarget{obs:B7}{\hypertarget{obs:C7}{B.7}} & $Y_{3}[1001]$ & $R_{1}[1111]$ & 0 & $B_{11}\,(c_{2},\,c_{1},\,c_{5},\,c_{4},\,c_{3},\,g,\,x,\,f,\,a,\,v,\,w)$ & $R_{1}[1111]$, $Y_{3}[1001]$\\
\hypertarget{obs:B8}{\hypertarget{obs:C8}{B.8}} & $Y_{3}[1101]$ & $R_{1}[1111]$ & 0 & $B_{10}\,(c_{1},\,c_{2},\,c_{3},\,c_{4},\,c_{5},\,f,\,a,\,g,\,v,\,w)$ & $R_{1}[1111]$, $Y_{3}[1101]$\\
\hypertarget{obs:B9}{\hypertarget{obs:C9}{B.9}} & $R_{2}[1011]$ & $Y_{4}[1011]$ & 0 & $B_{10}\,(c_{1},\,a,\,c_{3},\,f,\,w,\,c_{4},\,v,\,c_{2},\,g,\,x)$ & $R_{2}[1011]$\\
\hypertarget{obs:B10}{\hypertarget{obs:C10}{B.10}} & $Y_{4}[1011]$ & $R_{1}[1111]$ & 0 & $B_{10}\,(c_{2},\,c_{3},\,g,\,c_{5},\,x,\,c_{4},\,c_{1},\,w,\,v,\,a)$ & $R_{1}[1111]$, $Y_{4}[1011]$\\
\hypertarget{obs:B11}{\hypertarget{obs:C11}{B.11}} & $Y_{5}[0111]$ & $R_{2}[1111]$ & 0 & $B_{10}\,(c_{2},\,c_{1},\,c_{5},\,c_{4},\,c_{3},\,g,\,x,\,f,\,v,\,w)$ & $R_{2}[1111]$, $Y_{5}[0111]$\\
\hypertarget{obs:B12}{\hypertarget{obs:C12}{B.12}} & $Y_{5}[0110]$ & $Z[1111]$ & 1 & $B_{10}\,(x,\,c_{2},\,v,\,g,\,c_{4},\,c_{3},\,w,\,c_{5},\,a,\,f)$ & $Z[1111]$\\
\hypertarget{obs:B13}{\hypertarget{obs:C13}{B.13}} & $Y_{3}[1001]$ & $Z[1111]$ & 1 & $B_{10}\,(a,\,c_{1},\,v,\,f,\,c_{4},\,c_{5},\,w,\,c_{3},\,x,\,g)$ & $Z[1111]$\\

\addlinespace[1pt]
\multicolumn{6}{@{}l}{\textit{Claim~\ref{claim:Y1}, proof that $a$
is anticomplete to $Y_3\cup Y_4$:} $K=(C;\,f,\,a,\,y)$}\\[1pt]
\hypertarget{obs:B14}{\hypertarget{obs:C14}{B.14}} & $Y_{5}[010]$ & $R_{5}[110]$ & 1 & $B_{10}\,(c_{1},\,c_{2},\,f,\,c_{4},\,a,\,c_{3},\,c_{5},\,w,\,y,\,v)$ & Claim~\ref{claim:Y1}, Step~1\\

\addlinespace[1pt]
\multicolumn{6}{@{}l}{\textit{Certificate A.7:} $K=(C;\,f,\,a,\,y,\,p,\,q)$}\\[1pt]
\hypertarget{obs:B15}{\hypertarget{obs:C15}{B.15}} & $Y_{1}[00111]$ & --- & --- & $B_{10}\,(c_{1},\,c_{2},\,f,\,c_{4},\,v,\,c_{3},\,c_{5},\,q,\,y,\,p)$ & \emph{direct}\\
\hypertarget{obs:B16}{\hypertarget{obs:C16}{B.16}} & $R_{5}[11010]$ & --- & --- & $B_{10}\,(c_{1},\,c_{2},\,f,\,c_{4},\,a,\,c_{3},\,c_{5},\,v,\,y,\,p)$ & \emph{direct}\\

\addlinespace[1pt]
\multicolumn{6}{@{}l}{\textit{Certificate A.8:} $K=(C;\,f,\,a,\,y,\,p,\,q)$}\\[1pt]
\hypertarget{obs:B17}{\hypertarget{obs:C17}{B.17}} & $R_{5}[10111]$ & --- & --- & $B_{10}\,(c_{1},\,c_{2},\,c_{3},\,c_{4},\,v,\,f,\,a,\,q,\,y,\,p)$ & \emph{direct}\\
\hypertarget{obs:B18}{\hypertarget{obs:C18}{B.18}} & $Y_{4}[10101]$ & --- & --- & $B_{11}\,(a,\,p,\,f,\,v,\,c_{1},\,c_{2},\,c_{4},\,y,\,q,\,c_{5},\,c_{3})$ & \emph{direct}\\

\addlinespace[1pt]
\multicolumn{6}{@{}l}{\textit{Certificate A.11:} $K=(C;\,f,\,a,\,x,\,q')$}\\[1pt]
\hypertarget{obs:B19}{\hypertarget{obs:C19}{B.19}} & $R_{2}[1110]$ & $Y_{5}[0111]$ & 0 & $B_{11}\,(c_{2},\,c_{1},\,a,\,c_{4},\,f,\,q',\,x,\,c_{3},\,c_{5},\,v,\,w)$ & $R_{2}[1110]$, $Y_{5}[0111]$\\
\hypertarget{obs:B20}{\hypertarget{obs:C20}{B.20}} & $R_{1}[1101]$ & $Y_{3}[1011]$ & 0 & $B_{11}\,(c_{1},\,c_{2},\,f,\,c_{4},\,a,\,c_{3},\,c_{5},\,q',\,x,\,v,\,w)$ & $R_{1}[1101]$, $Y_{3}[1011]$\\
\hypertarget{obs:B21}{\hypertarget{obs:C21}{B.21}} & $R_{2}[1110]$ & $R_{4}[0011]$ & 1 & $B_{10}\,(c_{5},\,c_{1},\,v,\,c_{3},\,c_{4},\,a,\,w,\,f,\,x,\,q')$ & $R_{4}[0011]$\\
\hypertarget{obs:B22}{\hypertarget{obs:C22}{B.22}} & $R_{1}[1101]$ & $R_{4}[0011]$ & 1 & $B_{10}\,(x,\,c_{2},\,v,\,q',\,c_{4},\,f,\,w,\,a,\,c_{5},\,c_{3})$ & $R_{4}[0011]$\\
\hypertarget{obs:B23}{\hypertarget{obs:C23}{B.23}} & $Y_{3}[1011]$ & $R_{1}[1111]$ & 0 & $B_{10}\,(c_{2},\,c_{1},\,a,\,c_{4},\,f,\,q',\,x,\,c_{3},\,w,\,v)$ & $R_{1}[1111]$, $Y_{3}[1011]$\\

\addlinespace[1pt]
\multicolumn{6}{@{}l}{\textit{Certificate A.13:} $K=(C;\,f,\,a,\,r,\,y)$}\\[1pt]
\hypertarget{obs:B24}{\hypertarget{obs:C24}{B.24}} & $Y_{2}[0101]$ & $R_{5}[1111]$ & 0 & $B_{11}\,(c_{1},\,c_{2},\,f,\,c_{4},\,a,\,c_{3},\,c_{5},\,w,\,v,\,r,\,y)$ & $R_{5}[1111]$, $Y_{2}[0101]$\\
\hypertarget{obs:B25}{\hypertarget{obs:C25}{B.25}} & $Y_{5}[0101]$ & $R_{3}[0011]$ & 1 & $B_{10}\,(r,\,a,\,v,\,y,\,f,\,c_{3},\,w,\,c_{5},\,c_{2},\,c_{4})$ & $R_{3}[0011]$\\
\hypertarget{obs:B26}{\hypertarget{obs:C26}{B.26}} & $Y_{1}[0011]$ & $R_{5}[1111]$ & 0 & $B_{10}\,(c_{1},\,c_{2},\,c_{3},\,c_{4},\,w,\,f,\,v,\,a,\,y,\,r)$ & $Y_{1}[0011]$\\
\hypertarget{obs:B27}{\hypertarget{obs:C27}{B.27}} & $Y_{2}[0111]$ & $R_{5}[1111]$ & 0 & $B_{10}\,(c_{2},\,c_{1},\,a,\,c_{4},\,f,\,w,\,v,\,c_{3},\,r,\,y)$ & $R_{5}[1111]$, $Y_{2}[0111]$\\

\addlinespace[1pt]
\multicolumn{6}{@{}l}{\textit{Certificate A.17:} $K=(C;\,f,\,g,\,r,\,y)$}\\[1pt]
\hypertarget{obs:B28}{\hypertarget{obs:C28}{B.28}} & $R_{3}[0111]$ & --- & --- & $B_{10}\,(c_{2},\,c_{1},\,g,\,c_{4},\,f,\,c_{5},\,v,\,c_{3},\,r,\,y)$ & \emph{direct}\\
\hypertarget{obs:B29}{\hypertarget{obs:C29}{B.29}} & $R_{1}[1100]$ & $R_{3}[0101]$ & 1 & $B_{10}\,(c_{2},\,c_{1},\,g,\,c_{4},\,f,\,c_{5},\,w,\,c_{3},\,v,\,y)$ & $R_{3}[0101]$\\
\hypertarget{obs:B30}{\hypertarget{obs:C30}{B.30}} & $R_{3}[0101]$ & $R_{5}[1011]$ & 1 & $B_{11}\,(c_{1},\,c_{2},\,f,\,c_{4},\,g,\,c_{3},\,w,\,c_{5},\,v,\,r,\,y)$ & $R_{3}[0101]$, $R_{5}[1011]$\\

\addlinespace[1pt]
\multicolumn{6}{@{}l}{\textit{Certificate A.19:} $K=(C;\,f,\,g,\,q)$}\\[1pt]
\hypertarget{obs:B31}{\hypertarget{obs:C31}{B.31}} & $Y_{5}[010]$ & $Y_{2}[001]$ & 1 & $B_{10}\,(c_{2},\,c_{1},\,c_{5},\,c_{4},\,c_{3},\,g,\,w,\,f,\,v,\,q)$ & $Y_{2}[001]$\\

\addlinespace[1pt]
\multicolumn{6}{@{}l}{\textit{Certificate A.20:} $K=(C;\,f,\,g,\,q,\,b)$}\\[1pt]
\hypertarget{obs:B32}{\hypertarget{obs:C32}{B.32}} & $Y_{2}[0011]$ & --- & --- & $B_{10}\,(c_{2},\,c_{1},\,c_{5},\,c_{4},\,c_{3},\,g,\,v,\,f,\,b,\,q)$ & \emph{direct}\\
\hypertarget{obs:B33}{\hypertarget{obs:C33}{B.33}} & $R_{5}[1011]$ & --- & --- & $B_{10}\,(c_{1},\,c_{2},\,f,\,c_{4},\,g,\,c_{3},\,v,\,c_{5},\,q,\,b)$ & \emph{direct}\\
\hypertarget{obs:B34}{\hypertarget{obs:C34}{B.34}} & $Y_{5}[0100]$ & $Y_{2}[0010]$ & 1 & $B_{10}\,(c_{2},\,c_{1},\,c_{5},\,c_{4},\,c_{3},\,g,\,w,\,f,\,v,\,q)$ & $Y_{2}[0010]$\\
\hypertarget{obs:B35}{\hypertarget{obs:C35}{B.35}} & $R_{2}[1100]$ & $R_{5}[1001]$ & 1 & $B_{10}\,(c_{1},\,c_{2},\,f,\,c_{4},\,g,\,c_{3},\,w,\,c_{5},\,v,\,b)$ & $R_{5}[1001]$\\
\hypertarget{obs:B36}{\hypertarget{obs:C36}{B.36}} & $R_{5}[1001]$ & $R_{3}[0111]$ & 1 & $B_{11}\,(c_{2},\,c_{1},\,g,\,c_{4},\,f,\,c_{5},\,w,\,c_{3},\,v,\,q,\,b)$ & $R_{3}[0111]$, $R_{5}[1001]$\\

\end{longtable}
\end{center}

\end{document}